\pdfoutput=1
\newif\ifpersonal
\personaltrue 
\documentclass[12pt,a4paper,dvipsnames,x11names]{amsart} 
\usepackage{amsmath,amsthm,amssymb,mathrsfs,mathtools,bm,eucal,tensor} 
\usepackage{stmaryrd}
\usepackage[all,cmtip]{xy}
\usepackage{dsfont}
\usepackage{microtype,lmodern} 
\usepackage[utf8]{inputenc} 
\usepackage[T1]{fontenc} 
\usepackage{enumerate,comment,braket,xspace,tikz-cd,csquotes} 
\usepackage[centering,vscale=0.75,hscale=0.75]{geometry}

\usepackage[hidelinks,colorlinks=true,linkcolor=Red4,citecolor=RoyalBlue]{hyperref}
\usepackage[capitalize]{cleveref}

\usepackage{mathpazo}
\usepackage{euler}

\newtheorem{thm}[subsection]{Theorem}
\newtheorem{theorem}[subsection]{Theorem}
\newtheorem{prop}[subsection]{Proposition}
\newtheorem{proposition}[subsection]{Proposition}
\newtheorem{cor}[subsection]{Corollary}
\newtheorem{lem}[subsection]{Lemma}
\newtheorem{lemma}[subsection]{Lemma}

\newtheorem{eg}[subsection]{Example}

\theoremstyle{definition}

\newtheorem{defin}[subsection]{Definition}
\newtheorem{definition}[subsection]{Definition}

\newtheorem{recollection}[subsection]{Recollection}
\newtheorem{observation}[subsection]{Observation}

\newtheorem{rem}[subsection]{Remark}

\newtheorem{example}[subsection]{Example}
\newtheorem{construction}[subsection]{Construction}

\numberwithin{equation}{subsection}
\newtheorem{notation}[subsection]{Notation}

\usepackage{etoolbox}
\patchcmd{\section}{\scshape}{\bfseries}{}{}
\makeatletter
\renewcommand{\@secnumfont}{\bfseries}
\makeatother

\newcommand{\sep}{\mathrm{sep}}

\newcommand{\cl}{\mathrm{cl}}
\newcommand{\fm}{\mathfrak m}

\newcommand{\dt}{\mathrm{dimtot}}

\newcommand{\sw}{\mathrm{sw}}

\newcommand{\red}{\mathrm{red}}

\newcommand{\sO}{\mathscr{O}}

\newcommand{\cN}{\mathcal N}

\newcommand{\cF}{\mathcal F}
\newcommand{\fc}{\mu}
\newcommand{\fd}{P}
\newcommand{\etabar}{\overline{\eta}}
\newcommand{\bZ}{\mathbb Z}
\newcommand{\sF}{\mathcal{F}}

\newcommand{\bA}{\mathbb A}
\newcommand{\bP}{\mathbb P}
\newcommand{\bQ}{\mathbb Q}
\newcommand{\bT}{\mathbb T}
\newcommand{\bG}{\mathbb G}

\newcommand{\bF}{\mathbb F}
\newcommand{\bN}{\mathbb N}
\newcommand{\bR}{\mathbb R}
\newcommand{\bC}{\mathbb C}

\newcommand{\cS}{\mathcal S}
\newcommand{\cQ}{\mathcal Q}
\newcommand{\fb}{\mathfrak b}
\newcommand{\cE}{\mathcal E}
\newcommand{\cX}{\mathcal X}

\newcommand{\cM}{\mathcal M}
\newcommand{\cK}{\mathcal K}
\newcommand{\cP}{\mathcal P}
\newcommand{\cO}{\mathcal O}
\newcommand{\cC}{\mathcal C}
\newcommand{\cH}{\mathcal H}
\newcommand{\cL}{\mathcal L}
\newcommand{\cI}{\mathcal I}

\newcommand{\jmathbar}{\overline \jmath}

\DeclareMathOperator{\Fl}{Fl}

\DeclareMathOperator{\Tr}{Tr}

\DeclareMathOperator{\rk}{rk}

\DeclareMathOperator{\sbar}{\overline{s}}
\DeclareMathOperator{\ubar}{\overline{u}}
\DeclareMathOperator{\xbar}{\overline{x}}

\DeclareMathOperator{\Cons}{Cons}
\DeclareMathOperator{\lc}{lc}
\DeclareMathOperator{\Perv}{Perv}

\newcommand{\Qlbar}{\overline{\bQ}_\ell}  

\DeclareMathOperator{\Coh}{Coh}

\DeclareMathOperator{\length}{length}
\DeclareMathOperator{\tors}{tors}

\DeclareMathOperator{\Spec}{Spec}

\DeclareMathOperator{\dimtot}{dimtot}
\DeclareMathOperator{\Fr}{Fr}

\DeclareMathOperator{\Rk}{Rk}

\DeclareMathOperator{\Weil}{Weil}
\DeclareMathOperator{\IrrCom}{IrrCom}

\DeclareMathOperator{\modulo}{mod}
\DeclareMathOperator{\LC}{Loc}

\DeclareMathOperator{\Loc}{Loc}
\DeclareMathOperator{\Gr}{Gr}
\DeclareMathOperator{\ft}{tf}

\DeclareMathOperator{\even}{even}
\DeclareMathOperator{\odd}{odd}
\DeclareMathOperator{\pr}{pr}

\usepackage[english]{babel}

\usepackage[english]{babel}

\title[]{Estimates for Betti numbers and relative Hermite-Minkowski theorem for perverse sheaves}

\begin{document}

\author[H.Hu]{Haoyu Hu} 
\address{School of Mathematics, Nanjing University, Hankou Road 22, Nanjing, China}
\email{huhaoyu@nju.edu.cn, huhaoyu1987@gmail.com}

\author[J.-B. Teyssier]{Jean-Baptiste Teyssier}
\address{Institut de Math\'ematiques de Jussieu, 4 place Jussieu, Paris, France}
\email{jean-baptiste.teyssier@imj-prg.fr}

\setcounter{tocdepth}{1}
\begin{abstract}
We prove estimates for the Betti numbers of constructible sheaves in  characteristic $p>0$ depending only on their rank, stratification and wild ramification.
In particular, given a smooth proper variety of dimension $n$ over an algebraically closed field and a divisor $D$ of $X$, for every $0\leq i \leq n$, there is a polynomial $P_i$ of degree $\max \{i,2n-i\}$ such that the $i$-th Betti number of any  rank $r$ local system $\cL$ on $X-D$ is smaller than $P_i(\lc_D(\cL))\cdot r$ where $\lc_D(\cL)$ is the highest logarithmic conductor of $\cL$ at the generic points of $D$.
As application, we show that the Betti numbers of the inverse and higher direct images of a local system are controlled by the rank and the highest logarithmic conductor only. 
We also reprove Deligne's finiteness for simple $\ell$-adic local systems with bounded rank and ramification on a smooth variety over a finite field and extend it in two different directions.
In particular, perverse sheaves over arbitrary singular schemes are allowed and the bounds we obtain are uniform in algebraic families and do not depend on $\ell$.
\end{abstract}
\maketitle

\tableofcontents

\section{Introduction}
Let $U$ be a smooth variety over a finite field $\bF$ of characteristic $p>0$ and let $X$ be a normal compactification of $U$ such that $D:= X-U$ is an effective Cartier divisor.
Let $\ell \neq p$ be a prime number.
A celebrated result of Deligne \cite{DeltoDrin,EK} states that up to geometric isomorphism, there are only finitely many simple $\Qlbar$-local systems pure of weight 0  on $U$ with bounded rank and bounded wild ramification along $D$. 
From the existence of companions due to Lafforgue in the curve case \cite{Laf} and Drinfeld in higher dimension \cite{Drin},  the above number is independent of $\ell$.
Note that Deligne's finiteness is tight to the finiteness of $\bF$  via the use of the Weil conjectures, and breaks down  when $k$ is infinite.
It is nonetheless a natural question to ask whether some of its consequences survive over more general fields than finite fields.\\ \indent

The goal of this paper is to explore the cohomological aspect of this question.
By Deligne's finiteness, we indeed know that there are only a finite number of possible Betti numbers for local systems on $U$ with bounded rank and bounded wild ramification along $D$.
Over an algebraically closed field,  explicit bounds for the Betti numbers of affine varieties were obtained by Katz \cite{KatzBetti}.
Boundedness for the Betti numbers of \textit{tame} local systems  is due to  Orgogozo \cite{Org}.
In wild situations, explicit bounds were obtained for Artin-Schreier sheaves on affine varieties by Bombieri \cite{Bombieri}, Adolphson-Sperber \cite{AdSp} and Katz \cite{KatzBetti}.
All these bounds were recently sharpened by Zhang  \cite{DZhang} and Wan-Zhang  \cite{Wan_Zhang}.
We show the following

\begin{thm}[\cref{general_boundedness_Betti_Xsmooth}]\label{thm_1}
Let $X$ be a proper  smooth scheme of finite type of dimension $n$ over an algebraically closed field $k$ of characteristic $p>0$.
Let $D$ be a reduced effective Cartier divisor of $X$ and put $U:=X-D$.
Then, there exists $\fd\in \mathds{N}[x]^{n+1}$ with $\fd_i$ of degree $i$ for $i=0,\dots, n$ such that for every prime $\ell \neq p$, every $j=0, \dots,2n$ and every  $\cL\in \Loc(U,\Qlbar)$  we have 
\begin{equation}\label{thm_1_eq}
h^j(U,\cL)\leq \fd_{\min(j,2n-j)}(\lc_D(\cL))\cdot \rk_{\Qlbar}\cL  \ .
\end{equation}
\end{thm}
In the above statement, $h^j(U,\cL)$ is the $j$-th Betti number of $\cL$ and $\lc_D(\cL)\in \bQ_{\geq 0}$ is the highest logarithmic conductor of $\cL$ at the generic points of $D$,  as defined by Abbes and Saito \cite{RamImperfect,as ii,logcc,wr}.
What makes the inequality \eqref{thm_1_eq} striking is that while the Betti numbers are \textit{global} invariants of $\Qlbar$-sheaves depending on $\ell$,  the right hand side of \eqref{thm_1_eq} depends only on the \textit{local} behavior of $\cL$ at the generic points of $D$.
So if one bounds the rank and the logarithmic conductor, the Betti numbers get bounded independently of  $\ell$ as in the finite field situation. 
For a characteristic zero analogue for flat bundles, see \cite{HuTeyssierDRBoundedness}.
\\ \indent

In the curve case, the polynomials $P_i$ can be made explicit using the Grothendieck-Ogg-Shafarevich formula \cite{GOS}.
For affine spaces, explicit polynomials can also be produced. 
To this end, let us introduce the following 

 \begin{definition}\label{defin_bi}
Put $b_0(x)=1, b_1(x)=x$ and for $n\geq 2$, define inductively  $b_{n}\in \bN[x]$ by
$$
b_{n}(x)=\sum^{n-1}_{\substack{i=0\\  i\neq n \modulo 2 }}(x+2)\cdot b_{i}(x+3)
+\sum^{n-1}_{\substack{i=0\\ i= n \modulo 2}} (x+3)\cdot b_{i}(x) + \sum^{n-1}_{\substack{i=0\\  i\neq n \modulo 2}}b_{i}(x).
$$
\end{definition}

Then, we have the following

\begin{thm}[\cref{maintheoremAn}]\label{thm_2}
Let $k$ be an algebraically closed field  of characteristic $p>0$ and let $\ell \neq p$ be a prime.
For every $0\leq i\leq n$ and every $\cL\in \Loc(\bA^n_k,\Qlbar)$, we have 
$$
h^i(\bA^n_k,\cL)\leq b_i(\lc_H(\cL))\cdot\rk_{\Qlbar}\cL  
$$
where $H$ is the hyperplane at infinity.
\end{thm}
In particular for $\bA^n_k$, the estimates do not depend on the base field $k$ nor its characteristic.
Note that the degrees of the polynomials $b_i$ are optimal as shown in \cref{sharp_degree}.\\ \indent

It turns out that both the proof of \cref{thm_1} and the applications require more flexibility than allowed by the conditions of \cref{thm_1}, which says nothing when singularities are involved or when $X$ varies in a family.
When $\cL$ is replaced by a constructible sheaf $\cF$ and when $X$ is singular, one faces the problem that some wild ramification may hide in codimension $>1$, which is typically not captured by the logarithmic conductor.
To solve this problem, we generalized the notion of ramification boundedness in \cite{HuTeyssierCohBoundedness} by using coherent sheaves instead of effective Cartier divisors. 
If $\bQ[\Coh(X)]$ is the free $\bQ$-vector space on the set of isomorphism classes of coherent sheaves on a scheme of finite type $X$ over a field $k$, we set the following 

\begin{defin}\label{defin_E_bound}
For $\cE\in \bQ[\Coh(X)]$, we say that a constructible sheaf $\cF$ on $X$ has \textit{log conductors bounded by $\cE$} if for every morphism $f : C\to X$ where $C$ is a smooth curve over $k$ and every $x\in C$, the logarithmic conductor of $\cF|_C$ at $x$ is smaller than the length of the torsion part of $(f^*\cE)_x$ viewed as a module over $\cO_{C,x}$.
\end{defin}

\cref{defin_E_bound} captures all constructible sheaves since for every such sheaf $\cF$, there is $\cE\in \Coh(X)$ such that $\cF$ has log conductors bounded by $\cE$ (see \cite{HuTeyssierCohBoundedness}).
Moreover, semi-continuity results for logarithmic conductors \cite{Hu_Leal} provide explicit bounds in many cases, as shown by the following 

\begin{example}\label{ex_1}
In the setting of \cref{thm_1}, every $\cL\in \Loc(U,\Qlbar)$ has log conductors bounded by $(\lc_D(\cL)+1)\cdot \cO_D$.
\end{example}

The main result of this paper is the following

\begin{theorem}[\cref{general_boundedness_Betti}]\label{thm_3}
Let $X\to S$ be a proper morphism with fibers of dimension $\leq n$ between schemes of finite type over a perfect field of characteristic $p>0$.
Let $\Sigma$ be a stratification on $X$.
Then, there is a  function $\fc : \bQ[\Coh(X)]\to \bQ$ and  $\fd\in \mathds{N}[x]^{n+1}$ with  $P_i$ of degree  $i$ for every $i=0,\dots, n$ such that for every algebraic geometric point $\sbar\to S$, every prime $\ell\neq p$, every  $\cE\in \bQ[\Coh(X)]$, every $j=0, \dots, 2n$  and every  $\Sigma_{\sbar}$-constructible $\Qlbar$-sheaf $\cF$ on $X_{\sbar}$ with log conductors bounded by $\cE_{\sbar}$, we have 
$$
h^j(X_{\sbar},\cF)\leq \fd_{\min(j,2n-j)}(\fc(\cE))\cdot \Rk_{\Qlbar}\cF  \ .
$$
\end{theorem}
In this statement, $\Sigma_{\sbar}$-constructibilty means that the restriction of $\cF$ to the strata of the stratification induced by $\Sigma$ on $X_{\sbar}$ are locally constant and $\Rk_{\Qlbar}\cF$ stands for the maximal rank of the germs of $\cF$.
When $S$ is the spectrum of an algebraically closed field of characteristic $p>0$, when $X$ is smooth and $\Sigma=\{U,D\}$ with $D$ an effective Cartier divisor, \cref{thm_3} combined with \cref{ex_1} immediately implies \cref{thm_1}.
Observe that to get the particular \cref{thm_1} in its sharpest form where the bounding polynomials have decreasing degrees above $n$, we need to prove the more general \cref{thm_3} directly.
The reason is that allowing arbitrary singular base schemes gives extra flexibility for a recursion on the dimension.
In particular, it bypasses the use of De Jong's alterations \cite{DeJong} and the  cohomological descent spectral sequence \cite{Conrad}, which makes middle dimension cohomologies on the $E_1$-page contribute to beyond middle dimension cohomologies of the abutment.
\begin{rem}
Several variations of \cref{thm_3} can be given : one may consider complexes with bounded constructible cohomology or perverse sheaves and one may work with finite coefficients.
See \cref{estimes_general} for precise statements.
\end{rem}

Using Sawin's quantitative sheaf theory \cite{SFFK},  \cref{thm_3} gives control on the Betti numbers of inverse images (\cref{Betti_inverse_image}) in terms of the wild ramification.
For locally constant constructible sheaves, this specializes to the following 

\begin{thm}[\cref{Betti_inverse_image_lc}]
Let $f : Y\to X$  be a morphism between projective schemes over $k$ algebraically closed where $X$ is smooth.
Let $D$ be an effective Cartier divisor of $X$ and put $U:=X-D$ and $V:=Y-f^{-1}(D)$.
Then there is $P\in \mathds{N}[x]$ of degree $\dim X$ such that for every prime $\ell \neq p$ and every $\cL\in \Loc(U,\Qlbar)$,  we have 
$$
\sum_{j\in \bZ}  h^{j}_c(V,\cL|_V) \leq P(\lc_D(\cL))\cdot \rk_{\Qlbar}\cL  \ .
$$
\end{thm}
Similarly, \cref{thm_3} implies estimates for the Betti numbers of higher direct images (\cref{Coho_direct_image}).
For locally constant constructible sheaves, this specializes to the following 

\begin{thm}[\cref{Coho_direct_image_lc}]
Let $f : X\to Y$ be a projective morphism between  projective schemes  over $k$ algebraically closed where $X$ is smooth. 
Let $D$ be an effective Cartier divisor of $X$ and put $j : U:=X-D\hookrightarrow X$.
Then, there is a function $C :  \bQ^+  \times  \bN^+  \to \bN^+$  such that for every prime $\ell \neq p$ and every  $\cL\in \Loc(U,\Qlbar)$,  we have  
$$
\sum_{i,j\in \bZ} h^{j}(Y, R^i f_*j_!\cL) \leq C(\lc_D(\cL), \rk_{\Qlbar}\cL)  \ .
$$
\end{thm}

\begin{rem}
All continuity results from \cite{SFFK} translate thanks to  \cref{thm_3} into estimates for Betti numbers involving the rank and the wild ramification only.
\end{rem}

Our next application of \cref{thm_3} is a relative perverse geometric form of Hermite-Minkowski's theorem in number theory according to which there are only finitely many number fields with bounded discriminant.
For more applications to characteristic cycles and to wild Lefschetz type theorems, see \cite{HTLefschetz}.
To state the relative perverse Hermite-Minkowski theorem, let us introduce the following 
\begin{notation}
Let $(X,\Sigma)$ be a stratified scheme of finite type over a field $k$. 
Let $\cE\in \bQ[\Coh(X)]$ and $r\geq 0$.
We denote by $\Perv_{\Sigma}^{\leq r}(X,\cE,\Qlbar)$ the full subcategory of $\Perv(X,\Qlbar)$  spanned by objects $\cP$ such that for every $j\in \bZ$, the constructible sheaf $\cH^j\cP$ is $\Sigma$-constructible, has log conductors bounded by $\cE$ and satisfies $\Rk_{\Qlbar}\cH^j\cP\leq r$.
\end{notation}
\begin{thm}[\cref{Deligne_finitness}]\label{thm_4}
Let $X\to S$ be a projective morphism between schemes of finite type over a finite field of characteristic $p>0$.
Let $\Sigma$ be a stratification on $X$.
Then, there is a function $N : \bQ[\Coh(X)] \times \bN^+ \times \bN^+ \to \bN^+$ such that for every $\cE\in \bQ[\Coh(X)]$, every $r\geq 0$, every prime $\ell \neq p$ and every closed point $s\in  S$, there are up to geometric isomorphism at most $N(\cE,r,\deg s)$  geometrically simple pure of weight 0 objects in $\Perv_{\Sigma_s}^{\leq r}(X_s,\cE_s,\Qlbar)$.
\end{thm}
When $S$ is a point,  $X$ is normal and $\Sigma=\{U,D\}$ where $D$ is an effective Cartier divisor with $U$ smooth,  \cref{thm_4} reproduces Deligne's finiteness theorem \cite{DeltoDrin,EK}.
For a variant building on the smooth case where $U$ is  normal, see \cite{RemDel}.
See also \cite{Drin81,FKM,Comptage,DF13,Fli,Yu21,Yu23a,Yu23b} for concrete bounds in the curve case. 
In a nutshell, Deligne's approach rests upon the identification of simple local systems with bounded rank and ramification with the irreducible components of a scheme of finite type 
$\cX$ over $\bQ$.
The scheme $\cX$ is constructed by putting all the characteristic polynomials of the local Frobenii together. 
The finite type property of $\cX$ is the crucial aspect of the construction.
It ultimately relies on the fact that in the bounded rank and bounded ramification situation, only a  finite number of  Frobenii are needed. 
Furthermore, the existence of companions 
\cite{Laf,Drin} implies independence in $\ell$ in Deligne's finiteness.
Our approach to \cref{thm_4} is different: it trades off the construction of $\cX$ for the use of Beilinson's singular support and Saito's characteristic cycle.
Still, the bounds we obtain are independent of $\ell$ and don't resort to companions.
At the cost of using the Langlands correspondence \cite{Drin88,Drin89,Laf}, the purity assumption from  \cref{thm_5} can be dropped. 
This is the following 
\begin{thm}[\cref{Deligne_finitness_with_Langlands}]\label{thm_5}
In the setting of \cref{thm_4}, there is a function $N : \bQ[\Coh(X)] \times \bN^+ \times \bN^+ \to \bN^+$ such that for every $\cE\in \bQ[\Coh(X)]$, every $r\geq 0$, every prime $\ell \neq p$ and every closed point $s\in  S$, there are up to geometric isomorphism at most $N(\cE,r,\deg s)$  geometrically simple  objects in $\Perv_{\Sigma_s}^{\leq r}(X_s,\cE_s,\Qlbar)$.
\end{thm}


\subsection*{On the proof of \cref{thm_3}}
The proof of \cref{thm_3} involves a reduction to the case where $f : X\to S$ is projective with geometrically normal fibres (\cref{normal_reduction}).
From this point on and up to some noise coming from lower dimension,  the function $\mu : \bQ[\Coh(X)]\to \bQ$ is concretely constructed in \cref{ex_subaddittive} using the lengths at codimension one points of the restrictions of $\cE\in \Coh(X)$ to the fibers of $f : X\to S$.
Using a result of Kedlaya \cite{Ked} combined with some estimates (\cref{remark_bound_conductor_direct_image})
for the conductor of a finite direct image proved in \cite{HuTeyssierCohBoundedness},  we reduce to prove  \cref{thm_2}.
To treat the $\bA_k^n$ case, we argue by recursion on $n$.
To simplify  the explanation, assume that $n=2$.
By a standard weak Lefschetz argument (\cref{hn-1<bn-1}),  we are reduced to estimate the Euler-Poincaré characteristic of $\cL\in \Loc(\bA_k^2,\Lambda)$.
The main idea is here to twist $\cL$ by a  generic enough Artin-Schreier sheaf $\cN$ so that the discrepancy between the $\chi$-ies of $\cL$ and its twist $\cL^{\prime}$  stays in control (\cref{chi_special_F}).
To estimates the $\chi$ of $\cL^{\prime}$, a well-chosen  pencil (\cref{Lefschetz_pencil_construction}) reduces us to estimate 
 the Euler-Poincaré characteristic of the higher direct images of $\cL^{\prime}$ along some map $\pi_A : \bA_k^2\to \bA_k^1$.
By choosing the conductor of $\cN$ at $\infty$ to be prime to $p$ and $> \lc_H(\cL)$, one sees that $R\pi_{A !}\cL^{\prime}$ is locally constant and concentrated in degree 1 with controlled rank (\cref{local_acyclicity_E_not_H0}).
We are thus left to control the conductor of  $R^1\pi_{A !}\cL^{\prime}$ at $\infty$.
Note that this question amounts to bound the wild ramification of some nearby cycle complex attached to $\cL^{\prime}$.
To achieve this,  we invoke the main result of \cite{HT} to take care of the locus of the special fiber free from any wild horizontal contribution.
Inspired by an argument from \cite[Corollary 1.5.7]{SaitoDirectimage}  relying on a suitable radicial pullback, we show that at the unique problematic point where \cite{HT} breaks down,  $\cN$ can be chosen so that the nearby cycle complex of $\cL^{\prime}$ vanishes (\cref{ULAblow-up}).

\subsection*{Linear Overview}
Section 2 provides an account of Abbes and Saito theory for the logarithmic conductor (\cref{propramfil}) and recall its behaviour under inverse image (\cref{semi_continuity}) and finite direct image (\cref{remark_bound_conductor_direct_image}).
In section 3,  we introduce Beilinson's singular support and Saito's characteristic cycle for étale sheaves.
Section 4 provides a control of the wild ramification of the nearby cycle complex generalizing the main result of \cite{HT} over henselian traits coming from geometric situations (\cref{generalH-T}).
Section 5 provides a crucial conductor estimates (\cref{computation_conductor_twisted}) for the higher direct image of some suitably twisted local system.
Section 6 is devoted to the proof of \cref{thm_2}.
In section 7,  we recall some basic material from \cite{HuTeyssierCohBoundedness} and prove \cref{thm_2} and some immediate consequences.
Variations of \cref{thm_2} dealing with inverse and higher direct images are obtained in section 8 by building on the continuity results from  \cite{SFFK}.
Section 9 is devoted to the proof of \cref{thm_4}.
The appendix 10 is a computation making the bounding polynomials explicit in the $\bA^n_k$ case.

\subsection*{Acknowledgement}
We thank A. Abbes, M. D'Addezio, Y. Cao, G. Chenevier, J.-F. Dat,  A. Eteve, H. Esnault, L. Fu, M. Kerz, V. Lafforgue, H. Qin, T. Saito,  B. Stroh,  D. Wan, E. Yang, D. Zhang,Y. Zhao, W. Zheng for their interest and their comments.
We  thank V. Thatte and A. Walker for inviting the second named author to the London number theory seminar, where an anonymous member of the audience asked whether \cref{thm_1} could help reprove Deligne's finiteness.
H. H. is supported by the National
Natural Science Foundation of China (Grants No. 12471012) and the Natural Science
Foundation of Jiangsu Province (Grant No. BK20231539).
\begin{notation}
We introduce the following running notations.
\begin{enumerate}\itemsep=0.2cm
\item[$\bullet$] $k$  denotes a perfect field of characteristic $p>0$.
\item[$\bullet$] The letter $\Lambda$ will refer to a finite local ring of residue characteristic $\ell \neq p$.

\item[$\bullet$] For a scheme $X$ of finite type over $k$, we denote by $D_{ctf}^b(X,\Lambda)$ the derived category of complexes of $\Lambda$-sheaves of finite tor-dimension with bounded and constructible cohomology sheaves.

\item[$\bullet$]
We let $\Cons_{ctf}(X,\Lambda)$ be the category of constructible sheaves of $\Lambda$-modules of finite tor-dimension over $X$ and
$\LC_{tf}(X,\Lambda)\subset \Cons_{tf}(X,\Lambda)$ the full subcategory spanned by  locally constant constructible sheaves.
By \cite[Lemma 4.4.14]{Wei}, the germs of any $\cL\in \LC_{tf}(X,\Lambda)$ are automatically free  $\Lambda$-modules of finite rank.
 
 \item[$\bullet$]
 $\Perv_{tf}(X,\Lambda)$ will denote the category of perverse sheaves of $\Lambda$-modules of finite tor-dimension over $X$ for the middle perversity function.

\item[$\bullet$]  Let  $X$  be a scheme of finite type over $k$ and let $\Lambda$ be a field of characteristic $\neq p$.
For $\cK \in D_{ctf}^b(X,\Lambda)$, we put
$$
\Rk_{\Lambda} \cK := \max\{\rk_{\Lambda}\cH^i \cK_{\xbar}, \text{ where $i\in \bZ$ and $\xbar\to X$ is algebraic geometric}\} \ .
$$
\item[$\bullet$]
For $r\geq 0$, we let $D_{ctf}^{\leq r}(X,\Lambda)\subset D_{ctf}^b(X,\Lambda)$ be the full subcategory spanned by objects $\cK$ such that $\Rk_{\Lambda} \cK\leq r$, and similarly with perverse complexes.

\item[$\bullet$]
For a finite stratification $\Sigma$ of $X$, we let $D_{\Sigma,tf}^b(X,\Lambda)\subset D_{tf}^b(X,\Lambda)$ be the full subcategory spanned by $\Sigma$-constructible complexes, and similarly with perverse complexes.
\end{enumerate}

\end{notation}

\section{Conductors of étale sheaves}

\subsection{Ramification filtrations}\label{local_fields_notation}

Let $K$ be a henselian discrete valuation field over $k$.
Let $\sO_K$ be the ring of integer of $K$, let $\fm_K$ be the maximal ideal of $\sO_K$ and $F$ the residue field of $\sO_K$. 
Fix $K\subset K^{\sep}$ a separable closure of $K$ and let $G_K$ be the Galois group of $K^{\sep}$ over $K$.
Let  $I_K\subset G_K$ be the inertia subgroup and let  $P_K\subset G_K$ be the wild ramification subgroup.

\begin{recollection}\label{local_fields}

 In \cite{RamImperfect},  Abbes and Saito defined two decreasing filtrations $\{G^r_K\}_{r\in\bQ_{>0}}$ and $\{G^r_{K,\log}\}_{r\in\bQ_{\geq 0}}$ on $G_K$ by closed normal subgroups.   
They are called the {\it the ramification filtration} and {\it the logarithmic ramification filtration} respectively.  
For $r\in\bQ_{\geq 0}$, put
\begin{equation*}
G^{r+}_K=\overline{\bigcup_{s>r}G_K^s}\ \ \ \textrm{and}\ \ \ G^{r+}_{K,\log}=\overline{\bigcup_{s>r}G_{K,\log}^s}.
\end{equation*}

\begin{prop}[{\cite{RamImperfect,as ii,logcc,wr}}]\label{propramfil}
The following properties hold :
\begin{enumerate}
\item
For any $0<r\leq 1$, we have 
$$
G^r_K=G^0_{K,\log}=I_K \text{ and } G^{+1}_K=G^{0+}_{K,\log}=P_K.
$$
\item
For any $r\in \bQ_{\geq 0}$, we have 
$$
G^{r+1}_K\subseteq G^r_{K,\log}\subseteq G^r_K.
$$
If $F$ is perfect, then for any $r\in \bQ_{\geq 0}$, we have 
$$
G^r_{K,\cl}=G^r_{K,\log}=G^{r+1}_K.
$$ 
where $G^r_{K,\cl}$ is the classical wild ramification subgroup as defined in \cite{CL}.
\item
For any $r\in \bQ_{> 0}$, the graded piece $G^r_{K,\log}/G^{r+}_{K,\log}$ is abelian, $p$-torsion and contained in the center of $P_K/G^{r+}_{K,\log}$.
\end{enumerate}
\end{prop}

Let $M$ be a finitely generated $\Lambda$-module with a continuous $P_K$-action. 
 The module $M$ has decompositions 
 \begin{equation}\label{twodecomp}
M=\bigoplus_{r\geq 1}M^{(r)}\ \ \ \textrm{and}\ \ \ M=\bigoplus_{r\geq 0}M_{\log}^{(r)}
\end{equation}
into $P_K$-stable $\Lambda$-submodules where $M^{(1)}=M^{(0)}_{\log}=M^{P_K}$, and such that for every $r\in \bQ_{>0}$,
\begin{align*}
(M^{(r+1)})^{G^{r+1}_K}& =0  \text{ and }    (M^{(r+1)})^{G^{(r+1)+}_K} =M^{(r+1)};\\
(M^{(r)}_{\log})^{G^{r}_{K,\log}}& =0  \text{ and }    (M^{(r)}_{\log})^{G^{r+}_{K,\log}}  =M^{(r)}_{\log}.
\end{align*}
The decompositions \eqref{twodecomp} are respectively called  the {\it slope decomposition} and the {\it logarithmic slope decomposition} of $M$. 
The values $r$ for which $M^{(r)}\neq 0$ (resp. $M^{(r)}_{\log}\neq 0$) are the {\it slopes} (resp. the {\it logarithmic slopes}) of $M$.
We denote by $c_K(M)$ the largest slope of $M$ and refer to $c_K(M)$ as the \textit{conductor of $M$}.
Similarly, we denote by $\lc_K(M)$ the largest logarithmic slope of $M$ and refer to $\lc_K(M)$ as the \textit{logarithmic conductor of $M$}.
We say that $M$ is {\it isoclinic} (resp. {\it logarithmic isoclinic}) if $M$ has only one slope (resp. only one logarithmic slope).

The following is an immediate consequence of  \cref{propramfil}-(2).

\begin{lem}\label{inequalityLogNonLog}
 Let $M$ be a finitely generated $\Lambda$-module with a continuous $P_K$-action. 
Then, 
$$
\lc_{K}(M)  \leq c_K(M)\leq  \lc_{K}(M)+1. 
$$
\end{lem}

If $M$ is free as a $\Lambda$-module, then so are the $M^{(r)}_{\log}$ and the $M^{(r)}$ in virtue of \cite[Lemma 1.5]{KatzGauss}.
In that case, the {\it total dimension} of $M$ is defined by
$$
\dt_K(M):=\sum_{r\geq 1}r\cdot\rk_{\Lambda} M^{(r)}
$$
and the {\it Swan conductor} of $M$ is defined by
$$
\sw_K(M):=\sum_{r\geq 0}r\cdot\rk_{\Lambda} M^{(r)}_{\log} \ .
$$

\begin{lem}[{\cite{RamImperfect}}]\label{inequality_Swan_dimtot}
In the setting of \cref{local_fields}, we have
$$
 \sw_K(M) \leq  \dt_K(M) \leq \sw_K(M)+\rk_{\Lambda}M \ .
$$
If the residue field $F$ is perfect, we have
\begin{align*}
\lc_{K}(M)+1&=c_K(M) \ .  \\
\sw_K(M)+\rk_{\Lambda}M&= \dt_K(M) \ .
\end{align*}
\end{lem}

\begin{eg}[{\cite[Example 1.1.7]{lau}}]\label{Laumon_computation}
Let $\psi : \bF_p \to \Lambda ^{\times}$ be a non trivial character.
Let $M$ be the free rank one $\Lambda$-module with continuous $G_K$-action
corresponding to the Artin-Schreier cover defined by the equation 
\begin{equation*}
t^p-t=\frac{u}{x} \quad u\in K^{\times}, x\in \fm_K
\end{equation*}
If $m:=v_K(x)$ is prime to $p$, then $lc_K(M)=m$ and $c_K(M)=m+1$.
\end{eg}

\begin{lemma}[{\cite[Lemma 1.9]{HT}}]\label{tensordtsw}
Let $M$ and $N$ be free $\Lambda$-modules of finite type with continuous $G_K$-actions.
\begin{enumerate}
\item
If $M$ and $N$ are isoclinic of slope $r$ and $s$ respectively with $r> s$, then $M\otimes_{\Lambda}N$ is isoclinic of slope $r$ and
$$
\dt_K(M\otimes_\Lambda N)=\rk_{\Lambda} N\cdot\rk_{\Lambda} M\cdot r \ .
$$
\item
If $M$ and $N$ are logarithmic isoclinic of logarithmic slope $ r$ and $s$ respectively with  $ r> s$, then $M\otimes_{\Lambda}N$ is logarithmic isoclinic of logarithmic slope $r$ and
$$
\sw_K(M\otimes_\Lambda N)=\rk_{\Lambda} N\cdot\rk_{\Lambda} M\cdot r   \ .
$$
\end{enumerate}
\end{lemma}

\end{recollection}

\begin{recollection}\label{ram_character}
Assume that $\Lambda$ is a field.
Let $M$ be a free $\Lambda$-module of finite type with a continuous $P_K$-action where $\Lambda$ contains a $p$-root of unity.
Assume that $M$ is logarithmic isoclinic of logarithmic slope $r$.
Let $X(r)$ be the set of isomorphism classes of finite characters 
$$
\chi : G_{K,\log}^r /G_{K,\log}^{r+}\to \Lambda^{\times} \ .
$$
As a consequence of \cref{propramfil}-(3) (see also \cite[Lemma 6.7]{Ram_and_clean}), the module $M$ has a unique direct sum decomposition
$$
M := \bigoplus_{\chi\in X(r)} M_{\chi}
$$
into $P_K$-stable sub-$\Lambda$-modules such that as $\Lambda$-module with a continuous $G_{K,\log}^r$-action, $M_{\chi}$ is a direct sum of copies of $\chi : G_{K,\log}^r /G_{K,\log}^{r+}\to \Lambda^{\times}$.

\begin{lem}\label{swan_inequality}
Assume that $\Lambda$ is a field.
Let $M$ be a free $\Lambda$-module of finite type with a continuous $P_K$-action where $\Lambda$ contains a $p$-root of unity.
Let $N$ be a  free $\Lambda$-module of rank 1 with a continuous $P_K$-action. 
Let $r$ be the logarithmic slope of $N$ and let $\chi : G_{K,\log}^r /G_{K,\log}^{r+}\to \Lambda^{\times}$ be the corresponding character.
If $\chi^{-1}$ does not contribute to the character decomposition of $M^{(r)}_{\log}$, we have
$$
\sw_K(M)\leq \sw_K(M\otimes_{\Lambda}N ) \leq \sw_K(M) + r\cdot \rk_{\Lambda} M
$$
\end{lem}

\begin{proof}
We have 
$$
M\otimes_{\Lambda}N =\bigoplus_{s\geq 0}M_{\log}^{(s)} \otimes_{\Lambda}N \ .
$$
By \cref{tensordtsw}, if $s<r$, the tensor product $M_{\log}^{(s)} \otimes_{\Lambda}N$ is logarithmic isoclinic of logarithmic slope $r$ and if $s>r$, it is logarithmic isoclinic of logarithmic slope $s$.
For every $\nu : G_{K,\log}^r /G_{K,\log}^{r+}\to 
\Lambda^{\times}$ contributing to the character 
decomposition of $M_{\log}^{(r)}$, observe that $G_{K,\log}^r
$ acts on $(M_{\log}^{(r)})_{\nu}\otimes_{\Lambda} N$ via a direct sum of copies of $\nu\otimes_{\Lambda}\chi : G_{K,\log}^r /G_{K,\log}^{r+}\to \Lambda^{\times}
$ which is non trivial by assumption.
In particular, $G_{K,\log}^{r+}$ acts trivially on $(M_{\log}^{(r)}
)_{\nu}\otimes_{\Lambda} N$ and $G_{K,\log}^{r}$ acts on  
$(M_{\log}^{(r)})_{\nu}\otimes_{\Lambda} N$ without 
non zero fixed points.
Thus, 
$$
(M\otimes_{\Lambda} N)^{(r)}_{\log}= \bigoplus_{\eta} (M_{\log}^{(r)})_{\nu}\otimes_{\Lambda} N
$$
where $\nu$ runs over the set of characters contributing to $M_{\log}^{(r)}$.
Thus
\begin{align*}
 \sw_K(M\otimes_{\Lambda}N )&  = r \cdot \sum_{s<r} \rk_{\Lambda} M_{\log}^{(s)} + \sum_{s\geq r} s\cdot \rk_{\Lambda} M_{\log}^{(s)} \\
 &     =  \sum_{s<r} (r-s) \cdot\rk_{\Lambda} M_{\log}^{(s)} +  \sw_K(M) \ .
\end{align*}
The conclusion thus follows.
\end{proof}

\end{recollection}

\subsection{Conductor divisors}\label{semi_continuity_conductors}\label{semi_continuity_section}
Let $X$ be a normal scheme of finite type over $k$.
Let $Z$ be an integral Weil divisor and let $\eta\in Z$ be its generic point.
Let $K$ be the fraction field of $\hat{\mathcal{O}}_{X,\eta}$ and fix a separable closure $K^{\sep}$ of $K$.
For  $\cF\in \Cons_{tf}(X,\Lambda)$, the pull-back $\cF|_{\Spec K}$ is a $\Lambda$-module of finite type with continuous $G_{K}$-action.
Using the notations from \cref{local_fields_notation}, we put
$$
c_{Z}(\cF):= c_{K}(\cF|_{\Spec K})  \text{ and }
 \lc_{Z}(\cF):= \lc_{K}(\cF|_{\Spec K})  \ .
$$

\begin{definition}
Let $X$ be a normal scheme of finite type over $k$ and let $\cF\in \Cons_{tf}(X,\Lambda)$.
We define the \textit{conductor divisor  of $\cF$} by
$$
C_X(\cF):=\sum_{Z}  c_{Z}(\cF) \cdot  Z
$$
and the \textit{logarithmic conductor divisor  of  $\cF$} by
$$
LC_X(\cF):=\sum_{Z} \lc_{Z}(\cF) \cdot  Z
$$
where the sums run over the set of integral Weil divisors of $X$.
\end{definition}

\begin{rem}
The above divisors are $\mathds{Q}$-Weil divisors of $X$.
We will sometimes abuse the notations and write $C(\cF)$ instead of $C_X(\cF)$ and similarly in the logarithmic case.
\end{rem}

\begin{definition}
In the setting of \cref{semi_continuity_section}, we define the \textit{generic conductor} and the \textit{generic logarithmic conductor of $\cL$ along $D$} respectively by
$$
c_D(\cL):= \max_{Z} c_{Z}(\cL) \text{ and }
 \lc_D(\cL):= \max_{Z}  \lc_{Z}(\cL)   \ .
$$
where $Z$ runs over the set of irreducible components of $D$.
\end{definition}

 The above divisors enjoy the following semi-continuity property :

\begin{theorem}[{\cite[Theorem 1.4,1.5]{Hu_Leal}}]\label{semi_continuity}
Let $f : Y\to X$ be a morphism of smooth schemes of finite type over $k$.
Let $D$ be an effective Cartier divisor on $X$ and put $U:=X- D$.
Assume that $E:=Y\times_X D$ is an effective  Cartier divisor on $Y$.
For every  $\cL\in \Loc_{tf}(U,\Lambda)$, we have
$$
C_Y( (j_!\cL)|_Y) \leq f^* C_X( j_!\cL) \ .    
$$
If furthermore $D$ has normal crossings,  we have
$$
 LC_Y((j_!\cL)|_Y) \leq f^* LC_X(j_!\cL)     \ .
$$
\end{theorem}


The generic logarithmic conductor divisor satisfies the following compatibility with finite push-forward :

\begin{thm}[{\cite[Theorem 1.3]{HuTeyssierCohBoundedness}}]\label{remark_bound_conductor_direct_image}
Let $f:Y\to X$ be a finite surjective morphism of normal schemes of finite type over $k$. 
Let $D$ be an irreducible effective Cartier divisor on $X$ and put $U:=X-D$.
Put $E:=D\times_X Y$ and  $V:=Y-E$.
Assume that the restriction $f_U:V\to U$ is étale.
Then, for every $\cL\in \Loc_{tf}(V,\Lambda)$ we have 
$$
\lc_{D}(f_{U*}\cL)\leq \lc_{D}(f_{U*}\Lambda) + d\cdot \lc_{E}(\cL)
$$
where $d$ is the generic degree of $f : X \to Y$.
\end{thm}

%


\section{Singular support and characteristic cycle of étale sheaves}

\subsection{The singular support}\label{Sing_support_setting}
Let $X$ be a smooth  scheme of finite type over $k$. 
We denote by $\bT^*X$ the cotangent bundle of $X$.
Let $C\subset $ be a closed conical subset. 
For a point $x\in X$, we put $\bT^*_{x}X=\bT^*X\times_X x$   and $C_{ x}=C\times_X x$.  

\begin{recollection}
Let $h:U\to X$ be a morphism of smooth schemes of finite type over $k$. 
For $u\in U$, we say that $h:U \to X$ is $C$-{\it transversal at $u$} if 
$$
\ker dh_{u} \bigcap C_{h(u)}\subseteq \{0\}\subseteq
\bT^*_{h(u)}X 
$$ 
where $dh_{u}:\bT^*_{h(u)}X\to  \bT^*_{u}U$ is the cotangent map of $h$ at $u$. 
We say that $h:U \to X$ is \textit{$C$-transversal} if it is $C$-transversal at every point of $U$. 
For a $C$-transversal morphism $h:U\to X$, we let $h^\circ C$ be the scheme theoretic image of $C\times_XU$ in $\bT^*U$ by  $dh:\bT^*X\times_XU \to \bT^*U$.\\ \indent
Let $f:X\to Y$ be a morphism of smooth schemes of finite type over $k$.  
For $x\in X$, we say that $f:X \to  Y$ is $C$-{\it transversal at} $x$ if 
$$
df_{x}^{-1}(C_{x})\subseteq \{0\}\subseteq \bT^*_{f(x)}Y
$$
We say that $f:X\to Y$ is $C$-{\it transversal } if it is $C$-transversal at every point of $X$.\\ \indent
Let $(h,f):Y\leftarrow U\to X$ be a pair of morphisms of between smooth schemes of finite type over $k$. 
We say that $(h,f)$ is $C$-{\it transversal} if $h:U \to X$ is $C$-transversal and if $f:U \to  Y$ is $h^\circ C$-transversal. 
\end{recollection}


\begin{definition}
In the setting of \cref{Sing_support_setting}, we say that $\cK\in D^b_c(X,\Lambda)$ is {\it micro-supported on} $C$ if for every $C$-transversal pair $(h,f):Y\leftarrow U \to   X$, the map $f:U \to  Y$ is universally locally acyclic with respect to $h^*\cK$.
\end{definition}

\begin{theorem}[{\cite[Theorem 1.3]{bei}}]\label{beilinson_theorem}
For every  $\cK\in D^b_c(X,\Lambda)$, there is a smallest closed conical subset $SS(\cK)\subset \bT^*X$ on which $\cK$ is micro-supported. 
Furthermore, if $X$ has pure dimension $n$, then $SS(\cK)$ has pure dimension $n$.
\end{theorem}

\begin{definition}
The closed conical subset $SS(\cK)$ is the \textit{singular support of $\cK$}. 
\end{definition}

The conductor can be detected by curves, due to the following :

\begin{prop}[{\cite[Corollary 3.9]{wr}}]\label{equality_DT}
Let $X$ be a smooth scheme over $k$. 
Let $D$ be an effective Cartier divisor on $X$ and put $j : U:=X- D\hookrightarrow X$. 
Let $\cL\in \LC_{tf}(U,\Lambda)$ and let $i:S\to X$ be an immersion over $k$ where $S$ is a smooth curve.
Assume that 
\begin{enumerate}\itemsep=0.2cm
\item  $S$ meets $D$ transversely at a single smooth point $x\in D$.
\item The map $i:S\to X$ is $SS(j_!\cL)$-transversal.
\item The ramification of $j_!\cL$  is non-degenerate at $x$.
\end{enumerate}  
Then, $C_{S}( (j_!\cL)|_{S}) = i^* C_X(j_!\cL)$.
\end{prop}

\begin{prop}[{\cite[Proposition 1.1.8]{saito22}}]\label{basechange}
Let
\begin{equation*}
\xymatrix{\relax
V\ar[r]^{h'}\ar[d]_{j'}\ar@{}|-{\Box}[rd]&U\ar[d]^j\\
W\ar[r]_h&X}
\end{equation*}
be a cartesian diagram of smooth schemes over $k$ such that the vertical arrows are open immersions and such that $h$ is separated.
Let $\cF$ be an object in $D^b_c(X,\Lambda)$ and assume that  $h : W\to X$ is $SS(Rj_*\sF)$-transversal.
Then, the base change morphism
$$
h^*Rj_*\cF\to Rj'_*h'^*\cF
$$
is an isomorphism in $D^b_c(X,\Lambda)$.
\end{prop}

\subsection{The characteristic cycle}
Let $f:X \to   S$ be a morphism between smooth schemes of finite type over $k$ where $S$ is a curve over $k$.
Let $x\in X$ be a closed point  and put $s=f(x)$. 
Note that any local trivialization of $\bT^*S$ in a neighborhood of $s$ gives rise to a local section of $\bT^*X$ in a neighborhood of $x$ by applying $df:\bT^*S\times_S X \to \bT^*X$. 
We abusively denote by $df$ this section. \\ \indent
We say that $x$ is an {\it at most $C$-isolated characteristic point} for $f:X \to  S$ if $f:X\setminus \{x\} \to   S$ is $C$-transversal. 
In that case, the intersection of a cycle $A$ supported on $C$ with $[df]$ is supported at most at a single point in $\bT^*_xX$. 
Since $C$ is conical, the intersection number $(A, [df])_{\bT^*X,x}$ is independent of the chosen local trivialization for $\bT^*S$ in a neighborhood of $s$. 

\begin{theorem}[{\cite[Theorem 5.9]{cc}}]\label{Milnor_formula}
Let $X$ be a smooth scheme of finite type over $k$.
For every  $\cK\in D^b_{ctf}(X,\Lambda)$, there is a unique cycle $CC(\cK)$ of $\bT^{\ast}X$ supported on $SS(\cK)$  such that for every \'etale morphism $h: U \to  X$, for every morphism $f: U \to   S$ with $S$ a smooth curve over $k$,  for every at most $h^{\circ}(SS(\cK))$-isolated characteristic point $u\in U$ for $f:U \to   S$, we have the following Milnor type formula
 \begin{equation}\label{Milnor}
 -\dt (R\Phi_{\ubar}(h^{\ast}\cK, f))=(h^*CC(\cK),[df])_{T^{\ast}U, u},
 \end{equation}
 where $R\Phi_{\ubar}(h^{\ast}\cK, f)$ denotes the stalk of the vanishing cycle of $h^*\cK$ with respect to $f:U \to  S$ at a geometric point $\ubar \to   U$ above $u$.
\end{theorem}

\begin{definition}
The cycle $CC(\cK)$ from \cref{Milnor_formula} is the {\it characteristic cycle of $\cK$}. 
\end{definition}

\begin{example}[{\cite[Lemma 5.11]{cc}}]\label{CCcurve}
Assume that $X$ is a smooth connected curve over $k$.
Let $\cK\in D^b_{ctf}(X,\Lambda)$ and let $U\subset X$ be a dense open subset where $\cK$ is locally constant.
Then,  we have 
\begin{equation*}
CC(\cK)=-\rk_{\Lambda}\cK_{\eta}\cdot[\bT^*_XX]-\sum_{x\in X- U}(\dt_x(\cK)-\rk_{\Lambda}\cK_x)\cdot[\bT^*_xX]
\end{equation*}
where $\eta$ is the generic point of $X$.
\end{example}

\subsection{Characteristic cycle and cohomology}
The following index formula provides a positive characteristic analogue of Kashiwara-Dubson's formula for $\mathcal{D}$-modules.

\begin{theorem}[{\cite[Theorem 7.13]{cc}}]\label{CC_and_chi}
Let $X$ be a smooth projective variety over an algebraically closed field $k$.
For every $\cK\in D^b_{ctf}(X,\Lambda)$, we have 
$$
\chi(X,\cK)= (CC(\cK), \bT^*_{X} X)_{\bT^*X}
$$
\end{theorem}

\begin{rem}\label{chi=chic}
Assume that $\Lambda$ is a finite field.
When $\cK$ is of the form $j_! \cF$ where $j : U\hookrightarrow X$ is an open immersion and $\cF\in D^b_c(U,\Lambda)$, we have $\chi(X,\cK)=\chi_c(U,\cF)=\chi(U,\cF)$ in virtue of \cite{lauchi}.
\end{rem}

For curves, \cref{CC_and_chi} and \cref{CCcurve} specialize to the Grothendieck-Ogg-Shafarevich formula \cite{GOS} :

\begin{theorem}\label{GOS}
Let $X$ be a smooth proper connected curve of genus $g$ over an algebraically closed field $k$.
Let $\cK\in D^b_{ctf}(X,\Lambda)$ and let $U\subset X$ be a dense open subset where $\cK$ is locally constant.
Then, we have 
$$
\chi(X,\cK)= (2-2g)\cdot \rk_{\Lambda}  \cK_{\eta}  - \sum_{x\in X-U}(\dt_x(\cK) -\rk_{\Lambda}\cK_x)
$$
where $\eta$ is the generic point of $X$.
\end{theorem}

%


The combination of \cref{chi=chic} and \cref{GOS} gives the following :

\begin{cor}\label{GOcor_affine}
Let $X$ be a smooth proper connected curve of genus $g$  over an algebraically closed field $k$ and assume that $\Lambda$ is a finite field.
Let $U\subset X$ be an affine dense open subset and put $D:=X-U$.
For every  $\cL\in \LC(U,\Lambda)$, we have  
\begin{equation*}
\left\{
\begin{array}{l}
h^0(U,\cL)\leq \rk_\Lambda \cL, \\
h^1(U,\cL)\leq (2g-1+|D|+|D|\cdot \lc_{D}(\cL))\cdot\rk_\Lambda\cL,\\
h^2(U,\cL)= 0\\
\end{array}\right.
\end{equation*}

\end{cor}

\section{Ramification of nearby cycles}






\subsection{}In this section,  the notations of \cref{local_fields_notation} are in use. 
We further assume that $ \Lambda$ is a finite field of characteristic $\ell \neq p$.
We put $\cS:=\Spec \cO_K$, denote by $s$ its closed point and by  $\bar s$ an algebraic geometric point above $s$. 
We put $\eta=\Spec K$ and let $\eta^t$ be the maximal tame cover of $\eta$ dominated by $\etabar=\Spec \overline K$. 
For a morphism of finite type $f:\cX\to \cS$, consider the following diagram with cartesian squares

$$
\begin{tikzcd}    
\cX_{\overline{\eta}} \ar{r} \ar{d}{f_{\etabar}}   \arrow[rrr,  bend left = 40, "\jmathbar"]& \cX_{\eta^t} \ar{r}  \arrow[rr,  bend left = 25, "\jmath^t"] \arrow{d}& \cX_{\eta}  \ar{r}  \ar{d}{f_{\eta}}  & \cX  \ar{d}{f}  & \cX_{\sbar}  \ar{l}{\overline \iota}  \ar{d}{f_{\sbar}}  \\
\etabar \ar{r}& \eta^t \ar{r}& \eta \ar{r}&  \cS&  \sbar \ar{l} \ .
	\end{tikzcd} 
$$ 
For $\cK\in D^{+}(\cX_{\eta},\Lambda)$, we denote by $R\Psi(\cK,f)=\overline \iota^*R  \jmathbar_* \cK|_{\cX_{\overline{\eta}} }$ (resp. $R\Psi^{t}(\cK,f)=\overline \iota^*R \jmath^{t}_{*}  \cK|_{\cX_{\eta^t} }$) the nearby cycle complex (resp. tame nearby cycle complex) of $\cK$ with respect to $f:\cX\to\cS$. 
The nearby cycle complex is an object of $D^+(\cX_{\sbar}, \Lambda)$ with a $G_K$-action and the tame nearby cycle complex is an object of $D^+(\cX_{\sbar}, \Lambda)$ with a $G_K/P_K$-action.


\begin{definition}
For $\cK\in D^b_c(\mathcal X_{\eta},\Lambda)$, we say that $r\in \bQ_{\geq 0}$ is a {\it logarithmic slope} of $R\Psi(\cK,f)$ if there exists a closed point $x\in \cX_{\sbar}$ and  $i\in \bZ$ such that $r$ is a logarithmic slope of $R^i\Psi_x(\cK,f)$ with respect to the continuous $G_K$-action. 
We say that {\it the logarithmic ramification of $R\Psi(\cK,f)$ is bounded by} $c\in \bQ_{\geq 0}$ if the action $G^{c+}_{K,\log}$ on each $R^i\Psi(\cK,f)$ is trivial.
\end{definition}

 \begin{lemma}[{\cite[Lemma 5.3]{HT}}]\label{tame_nearby_cycle}
 Let $\mathscr K$ be an object of $D^b_c(\mathcal X_{\eta},\Lambda)$ and $\cN$ a locally constant  constructible sheaf of $\Lambda$-modules on $\eta$. 
 Then, for any closed point $x\in \mathcal X_{\overline s}$, we have the following canonical $G_K/P_K$-equivariant isomorphism:
 \begin{equation}
 R^i\Psi^{\mathrm t}_x(\cK\otimes_{\Lambda}^Lf^*_{\eta}\cN, f)\cong (R^i\Psi_x(\cK,f)\otimes_{\Lambda}\cN|_{\overline \eta})^{P_K}
 \end{equation}
 \end{lemma}
 The proof follows is the same as \cite[Lemma 5.3]{HT} where the residue field of $\cO_K$ was assumed to be perfect.
  However the proof is still valid without this assumption.
 
 \begin{proposition}[{\cite[Lemma 5.3, Corollary 5.4]{HT}}]\label{HT18nearby}
Let $\mathscr K$ be an object of $D^b_c(\mathcal X_{\eta},\Lambda)$. 
Then, $r$ is a logarithmic numbering slope of $R\Psi(\mathscr K,f)$ if and only if there exists a locally constant  constructible sheaf $\cN$ of $\Lambda$-modules on $\eta$ whose ramification is logarithmic isoclinic at $s\in \mathcal S$ with $r=\lc_s(\cN)$ such that
\begin{equation*}
R\Psi^{\mathrm t}(\mathscr K\otimes^L_{\Lambda}f_{\eta}^*\mathcal N,f)\neq 0.
\end{equation*}
In particular, the logarithmic ramification of $R\Psi(\mathscr K,f)$ is bounded by $c$ if and only if, for any locally constant constructible sheaf $\cN$ of $\Lambda$-modules on $\eta$ whose ramification is logarithmic isoclinic at $s\in \mathcal S$ with $\lc_s(\cN)>c$, we have
\begin{equation*}
R\Psi^{\mathrm t}(\mathscr K\otimes^L_{\Lambda}f_{\eta}^*\cN,f)=0.
\end{equation*}
\end{proposition}
\begin{proof}
Let $\cN$  a locally constant  constructible sheaf of $\Lambda$-modules on $\eta$ whose ramification is logarithmic isoclinic at $s\in \mathcal S$ with $r=\lc_s(\cN)$ such that
\begin{equation*}
R\Psi^{\mathrm t}(\mathscr K\otimes^L_{\Lambda}f_{\eta}^*\mathcal N,f)\neq 0.
\end{equation*}
Then, there exists a closed point $x\in \mathcal X_{\overline s}$ such that $R\Psi^{\mathrm t}_x(\mathscr K\otimes^L_{\Lambda}f_{\eta}^*\mathcal N,f)\neq 0$. 
By Lemma \ref{tame_nearby_cycle}, we have $(R^i\Psi_x(\mathscr K,f)\otimes_{\Lambda}\cN|_{\overline \eta})^{P_K}\neq 0$ for some $i\in\bZ$. 
Suppose that $R^i\Psi_x(\mathscr K,f)$ does not have a logarithmic slope $r$, then all logarithmic slopes of $R^i\Psi_x(\mathscr K,f)\otimes_{\Lambda}\cN|_{\overline \eta}$ are positive rational numbers. 
Therefore $R^i\Psi_x(\mathscr K,f)\otimes_{\Lambda}\cN|_{\overline \eta}$ is purely wild, which contradicts to the fact that $(R^i\Psi_x(\mathscr K,f)\otimes_{\Lambda}\cN|_{\overline \eta})^{P_K}\neq 0$. 
Hence $R^i\Psi_x(\mathscr K,f)$ has a logarithmic slope $r$. 

Conversely, let $r\in\mathbb Q_{\geq 0}$ be a logarithmic slope of $R\Psi(\mathscr K,f)$. 
Then, there exists a closed point $x\in \mathcal X_{\overline s}$ such that the finite generated $\Lambda$-module $M=\bigoplus_{i\in\bZ} R^i\Psi_x(\mathscr K,f)$ has a logarithmic slope $r$. 
Let $N$ be the dual of $M^r_{\log}$. 
Note that $N$ is also a finite generated $\Lambda$-module with a continuous $G_K$-action of isoclinic logarithmic slope $r$. 
Let $\cN$ be the locally constant  constructible sheaf of $\Lambda$-modules on $\eta$ which is associated to $N$. We have $\lc_s(\mathcal N)=r$. By Lemma \ref{tame_nearby_cycle}, we have 
\begin{equation}
\bigoplus_{i\in\bZ}R^i\Psi_x^{\mathrm t}(\mathscr K\otimes^L_{\Lambda}f_{\eta}^*\mathcal N,f)\cong (M\otimes_{\Lambda} N)^{P_K}.
\end{equation}
Notice that $N^{\vee}\otimes_{\Lambda} N$ is a sub $G_K$-representation of $M\otimes_{\Lambda} N$ and $N^{\vee}\otimes_{\Lambda} N$ has a quotient with the  trivial $P_K$-action. 
Hence 
$$0\neq (N^{\vee}\otimes_{\Lambda} N)^{P_K}\subseteq (M\otimes_{\Lambda} N)^{P_K}.$$ 
Therefore, $R\Psi_x^{\mathrm t}(\mathscr K\otimes^L_{\Lambda}f_{\eta}^*\mathcal N,f)\neq 0$.
\end{proof}

\begin{definition} 
Let $\mathcal Z$ be a reduced closed subscheme of $\mathcal X$. 
We say $(\mathcal X,\mathcal Z)$ is a semi-stable pair over $\mathcal S$ if, \'etale locally, $\mathcal X$ is \'etale over an $\mathcal S$-scheme 
\begin{equation*}
\mathrm{Spec} R[t_1,\cdots, t_d]/(t_{r+1}\cdots t_d-\pi)
\end{equation*}
where $r<d$ and $\pi$ is a uniformizer of $R$, and if $\mathcal Z=\mathcal Z_f\bigcup \mathcal X_s$ with $\mathcal Z_f$ defined by an ideal $(t_1\cdots t_m)\subset R[t_1,\cdots, t_d]/(t_{r+1}\cdots t_d-\pi)$ $(m\leq r)$.
\end{definition}

\begin{definition}
We say that the henselian trait $\mathcal S$ is geometric if $\mathcal S$ is the henselization of a smooth scheme $S$ over a perfect field of characteristic $p>0$ at the generic point of a smooth divisor. 
\end{definition}

\begin{theorem}\label{generalH-T}
Assume that $\mathcal S$ is geometric. 
Let $(\mathcal X,\mathcal Z)$ be a semi-stable pair and we assume that $\mathcal X$ is smooth over $\mathcal S$. 
Let $\mathcal U$ be the complement of $\mathcal Z$ in $\mathcal X$ and let $j:\mathcal U\to\mathcal X$ be the canonical injection. 
Let $\cF\in \LC(\mathcal U,\Lambda)$  such that its ramification at generic points of $\mathcal Z_f$ is tame and let $\lc(\cF)$ be the maximum of the set of logarithmic conductors of $\cF$ at generic points of the special fiber $\mathcal X_s$. 
Then, the logarithmic ramification of $R\Psi(\cF,f)$ is bounded by $\lc(\cF)$.
\end{theorem}
\begin{proof}
The proof follows a similar strategy as \cite[Theorem 5.7]{HT}. 
By  \cref{HT18nearby}, it is sufficient to show that, for any locally constant  constructible sheaf of $\Lambda$-modules $\cN$ on $\eta$ whose ramification is logarithmic isoclinic at $s$ with $\lc_s(\cN)>\lc(\cF)$, we have
\begin{equation}\label{RPsijFoN=0}
R\Psi^{\mathrm t}(j_!\cF\otimes^L_{\Lambda}f_{\eta}^*\cN,f)=0.
\end{equation}
As in step 2 of \cite[Theorem 5.7]{HT}, we reduce to the case where $\mathcal Z_f=\emptyset$. 
Since this is an \'etale local question, we may assume that  $\mathcal X_s$ has one irreducible component.  
Let $n$ be an integer co-prime to $p$ and $\pi$ a uniformizer of $\cO_K$. We put $\mathcal S_n=\mathrm{Spec}(\cO_K[T]/(T^n-\pi))$ and put $\mathcal X_n=\mathcal X\times_{\mathcal S}\mathcal S_n$. 
We have the following diagram
\begin{equation}
\xymatrix{\relax
\mathcal U_n\ar[r]^-(0.5){j_n}\ar[d]_{h_n}\ar@{}|-(0.5)\Box[rd]& \mathcal X_n\ar[d]^{g_n}&\mathcal X_s\ar[l]_-(0.5){\iota_n}\ar@{=}[d]\\
\mathcal U\ar[r]_-(0.5){j}&\mathcal X&\mathcal X_s\ar[l]^-(0.5){\iota}}
\end{equation}
We have the following isomorphism 
\begin{equation*}
R\Psi^{\mathrm t}(\cF\otimes_{\Lambda}f^*_{\eta}\cN,f)=\varinjlim_{(n,p)=1} \iota_n^*Rj_{n*}h_n^*(\cF\otimes_{\Lambda}f^*_{\eta}\cN).
\end{equation*}
Hence, it is sufficient to show that for every positive integer $\displaystyle{n>\frac{1}{\lc_s(N)-\lc(\cF)}}$ that is co-prime to $p$, we have 
\begin{equation}\label{RjnhnFfN=0}
\iota_n^*Rj_{n*}h_n^*(\cF\otimes_{\Lambda}f^*_{\eta}\cN)=0.
\end{equation}
Since $\mathcal S_n/\mathcal S$ is tamely ramified cover at $s$ of degree $n$, we have 
\begin{align}
n\cdot \lc(\cF)&=\lc_{\mathcal X_s}(h_n^*\cF),\\
n\cdot \lc_s(\cN)&=\lc_{s}(\cN|_{\mathcal S_n-\{s\}}).
\end{align}
Hence, we have $\lc_{s}(\cN|_{\mathcal S_n-\{s\}})>\lc_{\mathcal X_s}(h_n^*\cF)+1\geq c_{\mathcal X_s}(h_n^*\cF)$. 
Notice that $\mathcal S$ is geometric, $f:X\to S$ is of finite type and $\cF$ and $\cN$ are constructible. 
By spreading out, \eqref{RjnhnFfN=0} is due to \cref{RjFfN=0} below.
\end{proof}

\begin{proposition}\label{RjFfN=0}
Let $S$ be a smooth scheme over $k$,  let $E$ be an irreducible divisor of $S$ and put  $V:=S-E$.
Let $f:X\to S$ be a smooth morphism of finite type and consider the following commutative diagram
\begin{equation*}
\xymatrix{\relax
U\ar[r]^j\ar[d]_{f_V}\ar@{}|-(0.5){\Box}[rd]&X\ar[d]^f\ar@{}|-(0.5){\Box}[rd]&D\ar[l]\ar[d]\\
V\ar[r]_g&S&E\ar[l]} 
\end{equation*}
with cartesian squares.
Assume that $D$ is irreducible.
Let $\cF\in \Loc(U,\Lambda)$  and $\cN\in \Loc(V,\Lambda)$ such that $\cN$ has only one logarithmic slope at the generic point of $E$ and $\lc_E(\cN)>c_D(\cF)$. 
Then there exists an open dense subset $E_0$ of $E$ such that 
\begin{equation*}
(Rj_*(\cF\otimes_\Lambda f_V^*\cN))|_{f^{-1}(E_0)}=0.
\end{equation*}
\end{proposition}

\begin{proof}
This is an \'etale local question. 
Let $c_1<c_2<\cdots<c_r$ be all slopes of the ramification of $\cN$ at the generic point of $E$. 
We have $c_1\geq \lc_E(\cN)$. 
Let $E_0$ be an open dense subset of $E$ such that the ramification of $\cN$ along $E_0$ is non-degenerate. Since $f:X\to S$ is smooth, the ramification of $f_V^*\cN$ along $D_0=f^{-1}(E_0)$ is also non-degenerate and $c_1<c_2<\cdots<c_r$ are all slopes of the ramification of $f_V^*\cN$ at the generic point of $D$. We have
\begin{equation}
\dimtot_D(\cF\otimes_{\Lambda}f_V^*\cN)=\rk_{\Lambda}\cF\cdot \dimtot_D(f_V^*\cN)
\end{equation}
Let $C$ be a smooth curve over $k$ and $h:C\to X$ a quasi-finite morphism such that:
\begin{itemize}\itemsep=0.2cm
\item[(a)]
$y=h^{-1}(D)$ is a closed point of $C$ with $x=h(y)\subset D_0$;
\item[(b)]
 $h:C\to X$ is $SS(j_!f_V^*\cN)$-transversal.
\end{itemize}
By \cite[Corollary 3.9]{wr}, we have 
\begin{equation}\label{NonC}
\mathrm{NP}_y(f_V^*\cN|_{C_0})=m_y(h^*D)\cdot \mathrm{NP}_D(f_V^*\cN).
\end{equation}
By \cref{semi_continuity}, we have 
\begin{equation}\label{FonC}
c_y(\cF|_{C_0})\leq m_y(h^*D)\cdot c_D(\cF).
\end{equation}
By \eqref{NonC} and \eqref{FonC}, and the fact that all slopes of the ramification of $f_V^*\cN$ at the generic point of $D$ is larger than that of $\cF$, we obtain that
\begin{align*}
\dimtot_y((\cF\otimes_{\Lambda}f_V^*\cN)|_{C_0})&=\rk_{\Lambda}\cF\cdot \dimtot_y(f_V^*\cN|_{C_0})\\
&=\rk_{\Lambda}\cF\cdot m_y(h^*(\mathrm{DT}_X(j_!f_V^*\cN)))\\
&=m_y(h^*(\mathrm{DT}_X(j_!(\cF\otimes_{\Lambda}f_V^*\cN)))).
\end{align*}
By \cite[Proposition 5.6]{Hu_Leal}, we obtain that, for any closed point $x\in D_0$, we have 
\begin{equation}
SS(j_!(\cF\otimes_{\Lambda}f_V^*\cN))_{\overline x}\subseteq SS(j_!f_V^*\cN)_{\overline x}=(f^{\circ}SS(g_!\cN))_{\overline x}.
\end{equation}
In particular, $SS(j_!(\cF\otimes_{\Lambda}f_V^*\cN))_{ x}$ has dimension $1$ for any closed point $x\in D_0$. 
For any closed point $x\in D_0$, we take a smooth $k$-curve and an immersion $h:C\to X$ such that 
\begin{itemize}\itemsep=0.2cm
\item[(a)]
$C$ and $ D$ meet transversely at $x$;
\item[(b)]
 $h:C\to X$ is $SS(j_!(\cF\otimes_{\Lambda}f_V^*\cN))$-transversal.
\end{itemize}
Since $\lc_E(\cN)>c_E(\cF)\geq 0$, the ramifications of $f_V^*\cN$ and $\cF\otimes_{\Lambda}f_V^*\cN$ at the generic of $D$ have slopes $c_1,\dots, c_r>1$. 
By \cite[Corollary 3.9]{wr}, the ramification of $(\cF\otimes_{\Lambda}f_V^*\cN)|_C$ at $x$ has slopes $c_1,\dots, c_r>1$, in particular, is purely wild. 
By \cite[Corollary 3.12]{Hu_Leal}, we obtain that $(Rj_*(\cF\otimes_\Lambda f_V^*\cN))_x=0$. 
Hence $(Rj_*(\cF\otimes_\Lambda f_V^*\cN))|_{D_0}=0.$
\end{proof}

\section{A conductor estimate}

\begin{recollection}\label{Hu_crelle_recollection}
In this section, $\Lambda$ will denote a finite field of characteristic  $\ell\neq p$.
Let $X$ be a smooth scheme of finite type over $k$.
Let $D\subset X$ be an effective Cartier divisor and put $j : U:=X-D\hookrightarrow X$.
We denote by $\cQ(X,D)$ the set of triples $(S,h : S\to X,x)$ where $S$ is a smooth affine connected curve over $k$, where $h : S\to X$ is a quasi-finite morphism over $k$ such that $\{x\}=h(S)\cap D$ and such that $h^{-1}(x)$ is a single closed point of $S$.
\begin{defin}[{\cite[Definition 5.7]{Hu_Leal}}]\label{C-isoclinic}
Let $C\subset \bT^*X$ be a closed conical subset with basis $D$ and with pure 1-dimensional fibers over $D$.
For $\cL\in \LC(U,\Lambda)$, we say that the ramification of $\cL$ along $D$ is \textit{$C$-isoclinic by restricting to curves} if for every $(S,h : S\to X,x)\in \cQ(X,D)$ such that $h : S\to X$ is $C$-transversal at $s:=h^{-1}(x)$, the ramification of $\cL|_{S-\{s\}}$ at $s$ is isoclinic and we have
$$
C_{S}(\cL|_{S-\{s\}}) = h^*C_X(\cL) \ .
$$
\end{defin}

\cref{C-isoclinic}  admits the following purely geometric reformulation :

\begin{lem}[{\cite[Proposition 5.17]{Hu_Leal}}]\label{C-isoclinic_inclusion_SS}
In the setting of \cref{C-isoclinic}, the following are equivalent :
\begin{enumerate}\itemsep=0.2cm
\item the ramification of $\cL$ along $D$ is $C$-isoclinic by restricting to curves.
\item The ramification of $\cL$ at each generic point of $D$ is isoclinic and $SS(j_!\cL)\subseteq \bT^*_X X \bigcup C$.
\end{enumerate}
\end{lem}

When $C$ is simple enough, the inclusion in \cref{C-isoclinic_inclusion_SS} is an equality.
\begin{lem}[{\cite[Lemma 5.9]{Hu_Leal}}]\label{C-isoclinic_computation_SS}
In the setting of \cref{C-isoclinic}, assume that $\cL\neq 0$ and that the map $\IrrCom(C)\to \IrrCom(D)$ induced on the sets of irreducible components is bijective.
Then, 
$$
SS(j_!\cL)= \bT^*_X X \bigcup C \ .
$$
\end{lem}

\begin{lem}[{\cite[Proposition 5.18]{Hu_Leal}}]\label{C-isoclinic_tensor_product}
Let $X$ be a smooth scheme of finite type over $k$.
Let $D\subset X$ be an effective Cartier divisor and put $j : U:=X-D\hookrightarrow X$.
Let $C\subset \bT^*X$ be a closed conical subset with basis $D$ and pure 1-dimensional fibers over $D$. 
Let  $\cL,\cN\in \LC(U,\Lambda)$ where
\begin{enumerate}\itemsep=0.2cm
\item the ramification of $\cN$ along $D$ is $C$-isoclinic by restricting to curves.
\item we have $c_Z(\cN)>c_Z(\cL)$ for every irreducible component $Z$ of $D$.
\end{enumerate}
Then, the ramification of $\cL\otimes_{\Lambda}\cN$ along $D$ is $C$-isoclinic by restricting to curves and 
\begin{align*}
SS(j_!(\cL\otimes_{\Lambda}\cN)) & =SS(j_!\cN) \subseteq  \bT^*_X X \bigcup C  \ , \\
CC(j_!(\cL\otimes_{\Lambda}\cN)) & = \rk_{\Lambda}\cL \cdot CC(j_!\cN) \ .
\end{align*}

\end{lem}

\end{recollection}

\subsection{Conductor estimate}
The goal of what follows is to give an estimate for the logarithmic conductor of the direct image of some twisted sheaf on $\bA^n_k$ (see \cref{computation_conductor_twisted}).
The letter $\Lambda$ will denote a finite field of characteristic  $\ell\neq p$.

\begin{lemma}\label{AS_on_two_lines_complement}
Let $\bA^{2,\circ}_k$ be the complement of the origin in $\bA^2_k=\mathrm{Spec}(k[x,y])$.
Let $D$ (resp. $E$) be the divisor of $\bA^{2,\circ}_k$ defined by $x=0$ (resp. $y=1$).
Put $Z=D\bigcup E$ and $j : U:=\bA^{2,\circ}_k-Z\hookrightarrow \bA^{2,\circ}_k$.  
Let $\cN\in \Loc(U,\Lambda)$ be the Artin-Schreier sheaf on $U$ defined by
\begin{equation}
t^p-t=\lambda\left(\frac{y}{x^{p^2}(y-1)}\right)^m
\end{equation}
where $\lambda\in k^\times$ and  $(m,p)=1$. 
Then, 
$$
C_{\bA^{2,\circ}_k}(j_!\cN) =mp^2\cdot D+(m+1)\cdot E
$$
and the ramification of $\cN$ along $Z$ is $\langle dy\rangle\cdot Z$-isoclinic by restricting to curves.
\end{lemma}

\begin{proof}
Let $(S,h:S\to \bA^{2,\circ}_k,z)\in \cQ(\bA^2_k,Z)$ such that $h : S\to X$ is $\langle dy\rangle\cdot Z$-transversal at $s:=h^{-1}(z)$.
By \cref{equality_DT}, it is enough to show that 
\begin{equation}\label{cLS=mp2Dm+1E}
c_s(\cN|_{S-\{s\}})=mp^2\cdot m_s(h^*D)+(m+1)\cdot m_s(h^*E) 
\end{equation}
where $m_s(h^*D)$ (resp. $m_s(h^*E)$) is the multiplicity of $h^*D$ (resp. $h^*E)$ at $s$.
Put $K=\mathrm{Frac}(\widehat{\mathscr O}_{S,s})$ and $\mathscr O_K=\widehat{\mathscr O}_{S,s}$. 
We first assume that $z=(0,b)\in D-\{(0,1)\}$. 
Choose a uniformizer $T$ of $\mathscr O_{S,s}$ such that $h:S\to \bA^{2,\circ}_k$ is given by 
\begin{align*}
x&\mapsto uT^\alpha,\ \ \ \text{$u\in\mathscr O_{S,s}^\times, \alpha\geq 1$}\\
y&\mapsto b+T.
\end{align*}
Then, the restriction $\cN|_{\mathrm{Spec}(K)}$ corresponds to an Artin-Schreier cover defined by 
\begin{equation*}
t^p-t=\frac{\lambda(b+T)^m}{(uT^{\alpha})^{mp^2}(b-1+T)^m} \ .
\end{equation*}
Note that
\begin{equation}
\left(\frac{b+T}{b-1+T}\right)^m= \left(\frac{b}{b-1}\right)^m \left(1- \frac{m}{b(b-1)}  T + \cdots   \right)
\end{equation}
Since $(m,p)=1$, the coefficient of $T$ is a unit of $k$.
Hence, $\cN|_{\mathrm{Spec}(K)}$ corresponds to an Artin-Schreier cover defined by 
\[
t^p-t=\frac{\lambda' + u' T}{(uT^{\alpha})^{mp^2}} \quad\lambda'\in k^\times, u'\in\widehat{\mathscr O}^\times_{S,s}
\]
At the cost of taking a $\alpha mp^2$ root of $\lambda'$ in $k$ and making a change of variable, we can suppose that $\lambda'=1$.
On the other hand, we have 
\begin{align*}
t^p-t-\frac{1+u'T}{(uT^{\alpha})^{mp^2}}&=t^p - t- \frac{1}{(uT^{\alpha})^{mp^2}}- \frac{u'}{u^{^{mp^2}} T^{\alpha mp^2-1}} \\
        & = t'^p -t' -\frac{u''}{T^{\alpha mp^2-1}}
\end{align*}
where we put $t' =t - \frac{1}{(uT^{\alpha})^{mp}}$ and where $u'' \in \widehat{\mathscr O}^\times_{S,s}$.
By \cref{Laumon_computation}, we deduce
\[
c_s(\cN|_{S-\{s\}})=\dimtot_s(\cN|_{S-\{s\}})=\alpha mp^2=mp^2\cdot m_s(h^*E).
\]
We now assume that $z=(a,1)\in E-\{(0,1)\}$. 
Choose a uniformizer $T$ of $\mathscr O_{S,s}$ such that $h:S\to \bA^{2,\circ}_k$ is given by 
\begin{align*}
x&\mapsto a+uT^\alpha,\ \ \ u\in\mathscr O_{S,s}^\times, \alpha\geq 1,\\
y&\mapsto 1+T.
\end{align*}
The restriction $\cN|_{\mathrm{Spec}(K)}$ corresponds to an Artin-Schreier cover defined by 
\begin{equation*}
t^p-t=\frac{\lambda(1+T)^m}{(a+uT^{\alpha})^{mp^2}T^m} \ .
\end{equation*}
Since $(m,p)=1$, there is $v\in{\mathscr O}^\times_{K}$ with $\displaystyle{v^m=\frac{(a+uT^{\alpha})^{mp^2}}{\lambda(1+T)^m}}$. 
Hence $\cN|_{\mathrm{Spec}({K})}$ corresponds to an Artin-Schreier cover defined by 
\begin{equation*}
t^p-t=\frac{1}{(vT)^m} \ .
\end{equation*}
From \cref{Laumon_computation}, we deduce
\[
c_s(\cN|_{S-\{s\}})=\dimtot_s(\cN|_{S-\{s\}})=m+1=(m+1)\cdot m_s(h^*E)   \ .
\]
We now assume $z=(0,1)=D\bigcap E$.
Choose a uniformizer $T$ of $\mathscr O_{S,s}$ such that $h:S\to \bA^{2,\circ}_k$ is given by 
\begin{align*}
x&\mapsto uT^\alpha,\ \ \ u\in\mathscr O_{S,s}^\times, \alpha\geq 1,\\
y&\mapsto 1+T.
\end{align*}
The restriction $\cN|_{\mathrm{Spec}(K)}$ corresponds to an Artin-Schreier cover defined by 
\begin{equation*}
t^p-t=\frac{\lambda(1+T)^m}{u^{mp^2}T^{\alpha mp^2+m}} \ .
\end{equation*}
Since $\alpha mp^2+m$ is coprime to $p$, there is $v\in{\mathscr O}^\times_{K}$ such that $\displaystyle{v^{\alpha mp^2+m}=\frac{u^{mp^2}}{\lambda(1+T)^m}}$. 
Hence 
$\cN|_{\mathrm{Spec}({K})}$ corresponds to an Artin-Schreier cover defined by 
\begin{equation*}
t^p-t=\frac{1}{(vT)^{\alpha mp^2+m}} \ .
\end{equation*}
By \cref{Laumon_computation}, we deduce
\[
c_s(\cN|_{S-\{s\}})=\dimtot_s(\cN|_{S-\{s\}})=\alpha mp^2+m+1=mp^2\cdot m_s(h^*D)+(m+1)\cdot m_s(h^*E) \ .
\]

\end{proof}

\begin{cor}\label{CC_computation}
In the situation from \cref{AS_on_two_lines_complement}, we have 
\begin{align*}
SS(j_!\cN)&=\bT^*_{\bA^{2,\circ}_k}\bA^{2,\circ}_k\bigcup \langle dy\rangle \cdot Z,\\
CC(j_!\cN)&=[\bT^*_{\bA^{2,\circ}_k}\bA^{2,\circ}_k]+mp^2\cdot[D\cdot\langle dy\rangle]+(m+1)\cdot[\langle dy\rangle \cdot E].
\end{align*}
\end{cor}

\begin{proof}
The computation of $SS$ is a consequence of \cref{C-isoclinic_computation_SS} and  \cref{AS_on_two_lines_complement}.
Then, the computation of $CC$ follows from \cite[Theorem 7.6]{cc}.
\end{proof}

\begin{construction}\label{Lefschetz_pencil_construction}
Put $\bP^n_k=\mathrm{Proj}(k[x_0,\dots, x_n])$ with $n\geq 2$.
Let $H$  be the hyperplane defined by $x_0=0$ and $H_n$ the hyperplane defined by $x_n=0$.
Put
\begin{align*}
\bA^n_k=\mathrm{Spec}\left(k\left[\frac{x_1}{x_0},\dots,\frac{x_n}{x_0}\right]\right),\ \ \ U_n=\mathrm{Spec}\left(k\left[\frac{x_0}{x_n},\frac{x_1}{x_n}\dots,\frac{x_{n-1}}{x_n}\right]\right).
\end{align*}
Put 
$$
z_0=\frac{x_0}{x_n}, z_1=\frac{x_1}{x_n},\dots, z_{n-1}=\frac{x_{n-1}}{x_n}.
$$
Hence, $H\cap U_n$ is given by $z_0=0$.
Let $\pr: \bA^n_k\to\bA^1_k$ be the projection to the first coordinate.
Let $\Lambda$ be a finite field of characteristic  $\ell\neq p$.
Let $\cN_{\lambda}$ be the Artin-Schreier sheaf of $\Lambda$-modules on $\bA^1_k$ defined by the equation $t^p-t=\lambda x^m$ for some $\lambda\in k^\times$ and  $(m,p)=1$.
Then,  the pull-back $(\pr^*\cN_{\lambda})|_{U_n\cap \bA^n_k}$ is the Artin-Schreier sheaf defined by the equation 
$$
\displaystyle{t^p-t=\lambda\left(\frac{z_1}{z_0} \right)^m} \ .
$$
Consider the closed point $x=(0,1,\dots,1)\in H\cap U_n$. 
Let $X\to \bP^n_k$ be the blow-up of $\bP^n_k$ at $x$.
Let $E$ be the exceptional divisor.
Let $\Gr(n,2)$ be the grassmanniann of projective lines in $\mathbb{P}^n_k$ and let $Q\to \Gr(n,2)$ be the universal projective line.
Let $\mathbb{P}^{n-1}_k\subset \Gr(n,2)$  be the subset of projective lines passing through $x$.
Then, there is a pull-back square
$$
\begin{tikzcd}  
X  \ar{r}   \ar{d} \drar[phantom, "\scalebox{0.8}{$\square$}"] &       Q   \ar{r} \ar{d}    &   \mathbb{P}^n_k  \\
            \mathbb{P}^{n-1}_k           \ar{r}   \arrow[leftarrow, u, "\pi "]      &  \Gr(n,2)
\end{tikzcd}  
$$
exhibiting $X$ as the closed subscheme of $\mathbb{P}^{n}_k \times  \mathbb{P}^{n-1}_k$ formed by couples $(y,L)$ where $x,y\in L$.
Furthermore, the restriction $\pi|_E:E\to\bP^{n-1}_k$ is an isomorphism.
If $H'$ is the strict transform of $H$ and if $D\subset \bP^{n-1}_k$ is the hyperplane $\pi(H')$, there is a commutative square
\begin{equation}\label{diagram_pencil}
\xymatrix{\relax
H'\ar[d]\ar[r]\ar@{}|-{\Box}[rd]&X  \ar[d]^-(0.5){\pi}\ar@{}|-{\Box}[rd]&W\ar[l]\ar[d]_{\pi_W}&\bA^n_k\ar[l]^-(0.5){\jmath_W}\ar[ld]^{\pi_A}\ar@/_1.4pc/[ll]_{\jmath}\\
D\ar[r]&\bP^{n-1}_k&\bA^{n-1}_k\ar[l]^-(0.5){}}  \ .
\end{equation}

\end{construction}

\begin{proposition}\label{local_acyclicity_E_not_H0}
In the setting of \cref{Lefschetz_pencil_construction},  let $\cL\in \LC(\bA^n_k,\Lambda)$ and let $\cN_{\lambda}$ be an Artin-Schreier sheaf of $\Lambda$-modules on $\bA^1_k$ defined by $t^p-t=\lambda x^m$ for some $\lambda\in k^\times$ and  $(m,p)=1$. 
Assume that $m=\lc_\infty(\cN_\lambda)>\lc_H(\cL)+1$. 
Then the map 
 $\pi:X \to \bP^{n-1}_k$ is universally locally acyclic relative to $\jmath_!(\pr^*\cN_{\lambda}\otimes_{\Lambda}\cL)$ at every point of $E$ not in $H'$.
\end{proposition}

\begin{proof}
Since the restriction  $\pi|_E:E\to\bP^{n-1}_k$ is an isomorphism, it is enough by \cite{lau} to show that for every $L\in\bP^{n-1}_k- D$, the number 
$$
\dimtot_{(x,L)}(\jmath_!(\pr^*\cN_{\lambda}\otimes_{\Lambda}\cL)|_{\pi^{-1}(L)})
$$
does not depend on $L$.
Notice that for every  $L\in\bP^{n-1}_k -  D$, the composition 
$$
\pi^{-1}(L) - \{(x,L)\}\to\bA^{n}_k\xrightarrow{\pr}\bA^1_k$$ 
is an isomorphism. 
Hence, we have 
$$
\lc_{(x,L)}(\jmath_!\pr^*\cN_{\lambda}|_{\pi^{-1}(L)})=\lc_{\infty}(\cN_{\lambda})=m. 
$$
By \cref{semi_continuity} applied to the composition $\pi^{-1}(L)\to X\xrightarrow{\pi}\bP^n_k$ and to $\cL$, we obtain that for $L\in\bP^{n-1}_k -  D$, we have 
$$
\lc_{(x,L)}(\jmath_!\cL|_{\pi^{-1}(L)})\leq \lc_{E}(\cL)\leq \lc_H(\cL)<m-1.
$$
By \cref{inequality_Swan_dimtot} and \cref{tensordtsw}, for any $L\in\bP^{n-1}_k -  D$, we deduce
\begin{align*}
\dimtot_{(x,L)}(\jmath_!(\pr^*\cN_{\lambda}\otimes_{\Lambda}\cL)|_{\pi^{-1}(L)})  & =\sw_{(x,L)}(\jmath_!(\pr^*\cN_{\lambda}\otimes_{\Lambda}\cL)|_{\pi^{-1}(L)}) + \rk_{\Lambda} \cL\\
& =(m+1)\cdot\rk_{\Lambda}\cL.
\end{align*}
This concludes the proof of \cref{local_acyclicity_E_not_H0}.
\end{proof}

\begin{proposition}\label{ULAblow-up}
In the setting of \cref{Lefschetz_pencil_construction}, let $\cL\in \LC(\bA^n_k,\Lambda)$
 and let $\cN_{\lambda}$ be an Artin-Schreier sheaf of $\Lambda$-modules on $\bA^1_k$ defined by $t^p-t=\lambda x^m$ for some $\lambda\in k^\times$ and  $(m,p)=1$. 
Assume that $m=\lc_\infty(\cN_\lambda)>\lc_H(\cL)+1$. 
Then $\pi: X\to \bP^{n-1}_k$ is universally locally acyclic relative to $\jmath_!(\pr^*\cN_{\lambda}\otimes_{\Lambda}\cL)$ at a  dense  open subset of $H'\bigcap E$.
\end{proposition}
\begin{proof}
We are going to achieve universal local acyclicity after a suitable radicial pullback.
We learned this argument in \cite[Corollary 1.5.7]{SaitoDirectimage}.
Consider the pull-back square
$$
\begin{tikzcd}  
   Y_n \ar{r} \ar{d}   \drar[phantom, "\scalebox{0.8}{$\square$}"]   &  X  \ar{d}{\pi}    \\
   U_n \ar{r}    &      \mathbb{P}^{n}_k        
\end{tikzcd}  
$$
so that $Y_n$ is the blow-up of $U_n$ at $x$.
Hence, 
$Y_n=\{(y,L)\in U_n \times  \mathbb{P}^{n-1}_k \text{ such that } y\in L\}$.
Using the coordinates of $U_n$ from \cref{Lefschetz_pencil_construction}, the scheme $Y_n$ is the subscheme of $U_n \times \mathbb{P}^{n-1}_k$  defined by the equations
$$
z_0u_i=(z_i-1)u_0, 
(z_i-1)u_j=(z_j-1)u_i\ \ \textrm{for}\ \  i\neq j,\ \ i,j\in\{1,2,\dots, n-1\}  \ .
$$
Let $W_1\subset \mathbb{P}^{n-1}_k$ be the complement of the hyperplane defined by $u_1=0$.
Consider the pull-back square
$$
\begin{tikzcd}  
   V_1 \ar{r} \ar{d}   \drar[phantom, "\scalebox{0.8}{$\square$}"]  &  Y_n  \ar{d}   \\
   W_1\ar{r}    &    \mathbb{P}^{n-1}_k 
\end{tikzcd}  
$$
In particular $V_1$ is a locally closed subset of $U_n \times W_1$.
If we put $s_i = u_i/u_1$ for $i\neq 1$, then $V_1\simeq \Spec(k[s_0,s_2,\dots, s_{n-1},s])$ such that $V_1 \to U_n$ is given by 
\begin{align*}
z_0&\mapsto s_0(s-1),\\
z_1&\mapsto s,\\
z_i&\mapsto s_i(s-1)+1,\ \ \ (2\leq i\leq n-1).
\end{align*}
and $V_1 \to W_1$ is the projection on the first $(n-1)$-coordinates.
Put $E_1=V_1\cap E$ and $H_{1}' = V_1\cap H'$.
Then, $V_1 \cap \pi^{-1}(H)=  H'_{1} \cup E_1$ is defined by $s_0(s-1)=0$ and if we denote by $\Omega$ its complement in $V_1$, the pull-back $(\pr^*\cN_{\lambda})|_{\Omega}$ is the Artin-Schreier sheaf defined by the equation 
$$
\displaystyle{t^p-t=\lambda\left(\frac{s}{s_0(s-1)} \right)^m} \ .
$$
Consider the cube with pull-back faces
\[
\begin{tikzcd}
		V_1\arrow{rr} \arrow{dr}{\Fr} \arrow{dd} & & \bA^2_k\arrow{dr}\arrow{dd} \\
		{} & V_1 \arrow[crossing over]{rr} & & \bA^2_k\arrow{dd}{} \\
		 W_1 \arrow{rr} \arrow{dr}& & \bA^1_k\arrow{dr}{\Fr_1} \\
		{} & W_1 \arrow{rr} \arrow[leftarrow, crossing over]{uu} & & \bA^1_k
	\end{tikzcd} 
\]
where the top horizontal arrows are $(s_0,s_2,\dots, s_{n-1},s) \to (s_0,s)$, the lower horizontal arrows are $(s_0,s)\to s_0$ and where $\Fr_1 : s_0 \to s_0^{p^2}$.
Let $Z\subset H'_{1}$ defined by the ideal $(s,s_0)$ and let $j : \Omega \to V_1-Z$ be the canonical inclusion.
Then, $\Fr^*j_{!}(\pr^*\cN_{\lambda})|_{\Omega}$ is the Artin-Schreier sheaf defined by the equation 
$$
\displaystyle{t^p-t=\lambda\left(\frac{s}{s_0^{p^2}(s-1)} \right)^m} \ .
$$
Since $ V_1 \to  \bA^2_k$ is smooth,  \cref{CC_computation} and \cite[Theorem 7.6]{cc} give 
\[
SS(\Fr^* j_{!}(\pr^*\cN_{\lambda})|_{\Omega})=\bT^*_{(V_1-Z)}(V_1-Z)\bigcup((H_{1}'-Z)\cup E_1)\cdot\langle ds\rangle
\]
and 
\begin{align*}
CC(\Fr^* j_{!}(\pr^*\cN_{\lambda})|_{\Omega})&=(-1)^n[\bT^*_{(V_1-Z)}(V_1-Z)]+(-1)^nmp^2\cdot[(H_{1}'-Z)\cdot\langle du\rangle] \\
                & +(-1)^n(m+1)\cdot[E_1\cdot\langle du\rangle]  .
\end{align*}
By \cref{semi_continuity}, we have 
\begin{align*}
c_{H'}(\Fr^* j_{!}(\cL|_{\Omega}))&\leq p^{2}\cdot c_H(\cL)\leq p^{2}\cdot (\lc_H(\cL)+1)<mp^2=c_{ H'}(\Fr^* j_{!}(\pr^*\cN_{\lambda}|_{\Omega}))
\end{align*}
and 
\begin{align*}
c_{E}(\Fr^* j_{!}(\cL|_{\Omega}))&\leq c_H(\cL)\leq \lc_H(\cL)+1<m+1=c_{E}(\Fr^* j_{!}(\pr^*\cN_{\lambda}|_{\Omega})).
\end{align*}
By \cref{AS_on_two_lines_complement} and \cref{C-isoclinic_tensor_product}, we have 
\begin{align*}
SS(\Fr^*  j_{!}((\cL\otimes_{\Lambda}\pr^*\cN_{\lambda})|_{\Omega}))&=SS(\Fr^* j_{!}(\pr^*\cN_{\lambda})|_{\Omega}))
\end{align*}
and 
\begin{align*}
CC(\Fr^*  j_{!}((\cL\otimes_{\Lambda}\pr^*\cN_{\lambda})|_{\Omega}))&=\rk_{\Lambda}\cL \cdot CC(\Fr^* j_{!}(\pr^*\cN_{\lambda})|_{\Omega}).
\end{align*}
Hence $V_1\to W_1$ is $SS(\Fr^*  j_{!}((\cL\otimes_{\Lambda}\pr^*\cN_{\lambda})|_{\Omega}))$-transversal.
Thus, $V_1\to W_1$ is universally locally acyclic with respect to $\Fr^*  j_{!}((\cL\otimes_{\Lambda}\pr^*\cN_{\lambda})|_{\Omega})$. 
Since $\Fr$ is a universal homeomorphism, we deduce that  $V_1\to W_1$ is universally locally acyclic with respect to $j_{!}((\cL\otimes_{\Lambda}\pr^*\cN_{\lambda})|_{\Omega})$. 
Since $Z$ is disjoint from $H' \cap E$, we deduce that  $\pi: X\to \bP^{n-1}_k$ is universally locally acyclic with respect to $\jmath_!(\pr^*\cN_{\lambda}\otimes_{\Lambda}\cL)$ at every point of  $V_1 \cap (H'\bigcap E)$.
This concludes the proof of \cref{ULAblow-up}.
\end{proof}

\begin{proposition}\label{computation_conductor_twisted}
Let $n\geq 2$ be an integer. 
In the setting of \cref{Lefschetz_pencil_construction}, let $\cL\in \LC(\bA^n_k,\Lambda)$ and let $\cN_{\lambda}$ be an Artin-Schreier sheaf of $\Lambda$-modules on $\bA^1_k$ defined by $t^p-t=\lambda x^m$ for some $\lambda\in k^\times$ and  $(m,p)=1$. 
Assume that $m=\lc_\infty(\cN_\lambda)>\lc_H(\cL)+1$. 
Then, the following holds :
\begin{enumerate}\itemsep=0.2cm
\item the complex $R\pi_{A!}(\pr^*\cN_{\lambda}\otimes_{\Lambda}\cL)$ 
is concentrated in degree $1$, has locally constant constructible cohomology sheaves and its formation commutes with base change.

\item We have 
$$
\lc_D(R^1\pi_{A!}(\pr^*\cN_{\lambda}\otimes_{\Lambda}\cL))\leq m.
$$
\end{enumerate}
\end{proposition}

\begin{proof}
With the notations from the diagram (\ref{diagram_pencil}), we have  
$$
R\pi_{A!}(\pr^*\cN_{\lambda}\otimes_{\Lambda}\cL)=R\pi_{W*} j_{W!}(\pr^*\cN_{\lambda}\otimes_{\Lambda}\cL).
$$ 
Hence, (1) follows by \cref{local_acyclicity_E_not_H0}.
We now bound the  logarithmic conductor of $R\pi_{A!}(\pr^*\cN_{\lambda}\otimes_{\Lambda}\cL)$ at the generc point $\xi$ of $D$.
To do this, observe that the proper base change gives 
$$
(R\pi_*\jmath_!(\pr^*\cN_{\lambda}\otimes_{\Lambda}\cL))|_{\bA^{n-1}_k}\simeq R\pi_{A!}(\pr^*\cN_{\lambda}\otimes_{\Lambda}\cL) \ . 
$$
Hence, we have to bound the logarithmic conductor of $R\pi_*\jmath_!(\pr^*\cN_{\lambda}\otimes_{\Lambda}\cL)$ at $\xi$.
Put $\cS:=\mathrm{Spec}(\mathscr O_{{\bP}^{n-1}_{k},\overline\xi}^{\mathrm{sh}})$ and consider the following pull-back squares
\begin{equation*}
\xymatrix{\relax
\cX_{\overline\xi}\ar[d]\ar[r]\ar@{}|-{\Box}[rd]&\cX\ar[d]^-(0.5){\pi_{\cS}}\ar@{}|-{\Box}[rd]\ar[r]&X\ar[d]^{\pi}\\
\overline\xi\ar[r]&\cS\ar[r]&\bP^{n-1}_k \ . }
\end{equation*}
Since the logarithmic conductor of $R\pi_*\jmath_!(\pr^*\cN_{\lambda}\otimes_{\Lambda}\cL)$ at $\xi$ is smaller than the logarithmic conductors of the 
$$
(R^*\Psi_{\pi_S}(\jmath_!(\pr^*\cN_{\lambda}\otimes_\Lambda\cL)|_{\cS}))_{y}  ,   y\in \cX_{\overline\xi}
$$
we are left to bound the logarithmic conductors of the above Galois modules.
Put $\cE := \cS \times_{\bP^{n-1}_k} E$.
Since $\pi|_E:E\to\bP^{n-1}_k$ is an isomorphism, so is the pullback $\pi_{\cS}|_{\cE}:\cE\to\cS $.
Hence, the closed subscheme $\cE\subset \cX$ meets the special fibre of $\pi_{\cS} : \cX \to \cS$ at a unique point $\overline{\eta}$ lying over the generic point of $H'\cap E$.
By \cref{ULAblow-up}, the map $\pi : X\to \bP^{n-1}_k$ is universally locally acyclic with respect to  $\jmath _!(\pr^*\cN_{\lambda}\otimes_{\Lambda}\cL)$ in an open neighborhood of the generic point of $H'\bigcap E$.
Hence, the pull-back $\pi_{\cS} : \cX \to \cS$ is universally locally acyclic at $\overline{\eta}$ with respect to  $\jmath_!(\pr^*\cN_{\lambda}\otimes_{\Lambda}\cL)|_{\cS}$.
In particular, 
$$
(R^*\Psi_{\pi_S}(\jmath_!(\pr^*\cN_{\lambda}\otimes_\Lambda\cL)|_{\cS}))_{\overline{\eta}}\simeq 0    .
$$
By Theorem \ref{generalH-T}, the logarithmic conductor of $(R^*\Psi_{\pi_S}(j_!(\pr^*\cN_{\lambda}\otimes_\Lambda\cL)))_y$ is smaller than $m$ for every $y\in X_{\overline\xi}$ with $y\neq \overline{\eta}$. 
This concludes the proof of \cref{computation_conductor_twisted}.
\end{proof}

\section{Estimates for Betti numbers on affine spaces}
In this section, we assume that  $k$ is algebraically closed  of characteristic $p>0$ and that $\Lambda$ is a finite field of characteristic $\ell \neq p$.
Let $\bA^n_k=\mathrm{Spec}(k[x_1,\dots, x_n])$ and let $\bP^n_k=\mathrm{Proj}(k[x_0,\dots, x_n])$ be the canonical completion of $\bA^n_k$ obtained by adding a homogeneous coordinate $x_0$.
Let $j:\bA^n_k\hookrightarrow \bP^n_k$ be the inclusion and let $H:=\bP^n_k-\bA^n_k$ be the hyperplane at infinity.
The goal of this subsection is to prove the following

\begin{theorem}\label{maintheoremAn}
Let $k$ be an algebraically closed field  of characteristic $p>0$ and let $\Lambda$ be a finite field of characteristic $\ell \neq p$.
For every $0\leq i\leq n$ and every $\cL\in \LC(\bA^n_k,\Lambda)$, we have 
$$
h^i(\bA^n_k,\cL)\leq b_i(\lc_H(\cL))\cdot\rk_{\Lambda}\cL 
$$
where the $b_i$ are defined in \cref{defin_bi}.
\end{theorem}

\begin{rem}
The estimates from \cref{maintheoremAn} do not depend on the base field $k$ nor the coefficients.
\end{rem}

\cref{maintheoremAn} will be proved in subsection \ref{pfmaintheoremAn}.
Before this, let us draw some immediate consequences.

\begin{cor}\label{maintheoremAn_explicit}
Let $k$ be an algebraically closed field  of characteristic $p>0$ and let $\Lambda$ be a finite field of characteristic $\ell \neq p$.
For every $\cL\in \LC(\bA^n_k,\Lambda)$, we have 
\begin{align*}
h^0(\bA^n_k,\cL)&\leq \rk_{\Lambda}\cL,\\
h^1(\bA^n_k,\cL)&\leq \lc_H(\cL)\cdot\rk_{\Lambda}\cL\\
h^2(\bA^n_k,\cL)&\leq (\lc_H(\cL)^2+7\lc_H(\cL)+9)\cdot\rk_{\Lambda}\cL\\
h^{i}(\bA^n_k,\cL)&\leq (\lc_H(\cL)+3i-3)\prod_{j=1}^{i-1}(\lc_H(\cL)+3j+1)\cdot\rk_{\Lambda}\cL\; \text{for every $3\leq i\leq n$.}
\end{align*}
\end{cor}
\begin{proof}
It is a direct consequence of Lemma \ref{explicit_bound_b} and Theorem \ref{maintheoremAn}.
\end{proof}

\begin{cor}\label{maintheoremAn_compact}
Let $k$ be an algebraically closed field  of characteristic $p>0$ and let  $\Lambda$ be a finite field of characteristic $\ell \neq p$.
For every $0\leq i\leq n$ and every $\cL\in \LC(\bA^n_k,\Lambda)$, we have 
$$
h^{2n-i}_c(\bA^n_k,\cL)\leq b_i(\lc_H(\cL))\cdot\rk_{\Lambda}\cL \ .
$$
\end{cor}
\begin{proof}
This follows from \cref{maintheoremAn} by Poincaré duality and the fact that a $\Lambda$-module and its dual have the same logarithmic slopes.
\end{proof}

\begin{example}\label{sharp_degree}
The degrees of the polynomials appearing in \cref{maintheoremAn} are optimal.
Indeed,  if $m \geq 1$ is an integer coprime to $p$ and if $\cN$ is the Artin-Schreier sheaf on $\bA^1_k$ corresponding to $t^p-t=x^m$,  then $h^i(\bA^1_k,\cN)=0$ for $i\neq 1$ and $h^1(\bA^1_k,\cN)=m-1$ by \cref{GOS}.
For $1\leq i\leq n$, let $\pr_i :  \bA_k^n\to \bA_k^1$ be the projection on the $i$-th coordinate and put 
$$
\cM:= \pr^*_1 \cN  \otimes_{\Lambda} \cdots  \otimes_{\Lambda} \pr^*_i \cN \ .
$$
Then, $\cM$ is the Artin-Schreier sheaf on $\bA^n_k$ corresponding to  $t^p-t=x_1^m+ \cdots +x_i^m$.
By \cref{Laumon_computation}, we have $\lc_H(\cM)=m$.
On the other hand, the Künneth formula for the étale cohomology \cite[0F1P]{SP} yields
$$
h^i( \bA_k^n,\cM)= h^1( \bA_k^1,\cN)^i=(m-1)^i\ .
$$
\end{example}

\begin{proposition}\label{hn-1<bn-1}
Let $n\geq 2$ be an integer. 
Assume that \cref{maintheoremAn} is valid for $\bA^{n-1}_k$. 
For every  $\cL\in \LC(\bA^n_k,\Lambda)$ and every $i=0,1,\dots, n-1$, we have  
$$
h^i(\bA^n_k,\cL)\leq b_i(\lc_H(\cL))\cdot\rk_{\Lambda}\cL \ .
$$
\end{proposition}
\begin{proof}
By \cite[Lemma 1.3.7]{SaitoDirectimage}, there is a hyperplane $D\neq H$ such that the map $\iota :D\hookrightarrow \bP^n_k$  is properly $SS(Rj_*\cL)$-transversal. 
Put $U:= \bP^n_k - D$.
The cartesian squares
\begin{equation*}
\xymatrix{\relax
D_0\ar[d]_{j_D}\ar[r]^-(0.5){\iota_0}\ar@{}|-{\Box}[rd]&\bA^n_k\ar[d]^-(0.5){j}\ar@{}|-{\Box}[rd]&U_0\ar[d]^{j_U}\ar[l]_-(0.5){\jmath_0}\\
D\ar[r]_-(0.5)\iota &\bP^n_k&U\ar[l]^-(0.5)\jmath}
\end{equation*}
give rise to the  distinguished triangle
$$
\jmath_!\jmath^*Rj_*\cL\to Rj_*\cL\to \iota_*\iota^*Rj_*\cL\xrightarrow{+1}
$$
Since $\jmath$ is an open immersion, we have 
$$
\jmath^*Rj_*\cL\simeq Rj_{U*}\jmath^*_0\cL \ .
$$
Since $\iota :D\to\bP^n_k$ is properly $SS(Rj_*\cL)$-transversal, \cref{basechange} gives
$$
\iota ^*Rj_*\cL\simeq Rj_{D*}\iota _0^*\cL  \ .
$$ 
Hence, applying $R\Gamma(\bP^n_k,-)$ to the above triangle yields a  distinguished triangle
\begin{equation*}
R\Gamma_c(U, Rj_{U*}\jmath^*_0\cL)\to R\Gamma(\bA^n_k,\cL)\to R\Gamma(D_0,\iota^*_0\cL)\xrightarrow{+1}
\end{equation*}
Note that  $Rj_{U*}\jmath^*_0\cL[n]$ is a perverse sheaf on $U\simeq\bA^n_k$. 
By Artin's vanishing theorem, we deduce
\begin{equation*}
H^i_c(U, Rj_{U*}\jmath^*_0\cL)=0
\end{equation*}
for $i=0,1,\dots,n-1$. 
Hence, 
$$
h^i(\bA^n_k,\cL)\leq h^i(D_0,\iota^*_0\cL)
$$
for $i=0,1,\dots,n-1$. 
Note that $D\simeq \bP^{n-1}_k$, that $D\bigcap H$ is a hyperplane of $D$ with complement $D_0\simeq \bA^{n-1}_k$ in $D$.
By \cref{semi_continuity}, we have 
$$
\lc_{D\cap H}(\iota^*_0\cL)\leq \lc_{H}(\cL) \ .
$$ 
Since  \cref{maintheoremAn} is valid on $\bA^{n-1}_k$ by assumption, we get for every $i=0,1,\dots, n-1$, 
\begin{align*}
h^i(D_0,\iota^*_0\cL)\leq b_i(\lc_{D\cap H}(\iota^*_0\cL))\cdot\rk_{\Lambda}(\iota^*_0\cL)\leq b_i(\lc_{H}(\cL))\cdot\rk_{\Lambda}\cL.
\end{align*}
This concludes the proof of \cref{hn-1<bn-1}.
\end{proof}

\begin{proposition}\label{<chiprN0F<}
Let $n\geq 2$ be an integer. 
Assume that Theorem \ref{maintheoremAn} is valid for $\bA^{n-1}_k$.
Let $\cL\in \LC(\bA^n_k,\Lambda)$.
Let $\pr:\bA^n_k\to\bA^1_k$ be the first projection. 
Let $\cN_{\lambda}$ be an Artin-Schreier sheaf of $\Lambda$-modules on $\bA^1_k$ defined by $t^p-t=\lambda x^m$ for some $\lambda\in k^\times$ and  $(m,p)=1$. 
Assume that $m=\lc_\infty(\cN_\lambda)>\lc_H(\cL)+1$. 
Then, we have 
$$
\chi(\bA^n_k, \pr^*\cN_{\lambda}\otimes_{\Lambda}\cL) \geq -\left(\sum^{n-1}_{\substack{i=0\\  \text{$i$ $\even$ }}} b_{i}(m)\right)(m-1)\cdot\rk_{\Lambda}\cL
$$
and 
$$
\chi(\bA^n_k, \pr^*\cN_{\lambda}\otimes_{\Lambda}\cL)\leq\left(\sum^{n-1}_{\substack{i=0\\  \text{$i$ $\odd$} }} b_{i}(m)\right)(m-1)\cdot\rk_{\Lambda}\cL \ .
$$
\end{proposition}
\begin{proof}
We use the setup from \cref{Lefschetz_pencil_construction}.
Let $X\to \bP_k^n$ be the blow-up of $\bP_k^n$  at $x:=[0:1:\dots : 1]$.
Let $E$ be the exceptional divisor and let $\pi: X\to\bP^{n-1}_k$ be the pencil of lines passing through $x$.
Let $H'$ be  the strict transform of $H$ and put $D:=\pi(H')$. 
We have the following commutative diagram:
\begin{equation*}
\xymatrix{\relax
H'\ar[d]\ar[r]\ar@{}|-{\Box}[rd]&X  \ar[d]^-(0.5){\pi}\ar@{}|-{\Box}[rd]&W\ar[l]\ar[d]_{\pi_W}&\bA^n_k\ar[l]^-(0.5){\jmath_W}\ar[ld]^{\pi_A}\ar@/_1.4pc/[ll]_{\jmath}\\
D\ar[r]&\bP^{n-1}_k&\bA^{n-1}_k\ar[l]^-(0.5){}}
\end{equation*}
By \cref{computation_conductor_twisted}, the complex $R\pi_{A!}(\pr^*\cN_{\lambda}\otimes_{\Lambda}\cL)$
is concentrated in degree $1$, has locally constant constructible cohomology sheaves, its formation commutes with base change and
\begin{equation}\label{bound_log_conductor_R1}
\lc_D(R^1\pi_{A!}(\pr^*\cN_{\lambda}\otimes_{\Lambda}\cL))\leq m \ .
\end{equation}
By proper base change combined with \cref{GOS}, we deduce 
\begin{equation}\label{bound_rank_R1}
\rk_{\Lambda} (R^1\pi_{A!}(\pr^*\cN_{\lambda}\otimes_{\Lambda}\cL))=(m-1)\cdot\rk_{\Lambda}\cL \ .
\end{equation}
On the other hand, we have
\begin{align*}
\chi(\bA^n_k,\pr^*\cN_{\lambda}\otimes_\Lambda\cL)&=\chi_c(\bA^n_k,\pr^*\cN_{\lambda}\otimes_\Lambda\cL)    &  \text{ \cref{chi=chic} }  \\
&=\chi_c(\bA^{n-1}_k, R\pi_{A!}(\pr^*\cN_{\lambda}\otimes_\Lambda\cL)) & \\
&=-\chi_c(\bA^{n-1}_k, R^1\pi_{A!}(\pr^*\cN_{\lambda}\otimes_\Lambda\cL)) & \\
&=-\chi(\bA^{n-1}_k, R^1\pi_{A!}(\pr^*\cN_{\lambda}\otimes_\Lambda\cL))&  \text{ \cref{chi=chic} }  \\
&=\sum^{n-1}_{i=0}(-1)^{i+1}h^i(\bA^{n-1}_k, R^1\pi_{A!}(\pr^*\cN_{\lambda}\otimes_\Lambda\cL)).&   
\end{align*}
Hence, we have 
\begin{align*}
-\sum^{n-1}_{\substack{i=0\\  \text{$i$ even} }}  h^{i}(\bA^{n-1}_k, R^1\pi_{A!}(\pr^*\cN_{\lambda}\otimes_\Lambda\cL))&\leq \chi(\bA^n_k,\pr^*\cN_{\lambda}\otimes_\Lambda\cL)   \\
&\leq \sum^{n-1}_{\substack{i=0\\  \text{$i$ odd} }}  h^{i}(\bA^{n-1}_k, R^1\pi_{A!}(\pr^*\cN_{\lambda}\otimes_\Lambda\cL)).
\end{align*}
Since Theorem \ref{maintheoremAn} is valid for $\bA^{n-1}_k$, we have for every $i=0,1,\dots, n-1$, 
$$
h^i(\bA^{n-1}_k, R^1\pi_{A!}(\pr^*\cN_{\lambda}\otimes_\Lambda\cL))\leq b_i(\lc_D(R^1\pi_{A!}(\pr^*\cN_{\lambda}\otimes_\Lambda\cL))\cdot\rk_{\Lambda}(R^1\pi_{A!}(\pr^*\cN_{\lambda}\otimes_\Lambda\cL)).
$$
From (\ref{bound_log_conductor_R1}) and (\ref{bound_rank_R1}), we deduce 
$$
h^i(\bA^{n-1}_k, R^1\pi_{A!}(\pr^*\cN_{\lambda}\otimes_\Lambda\cL)) \leq b_i(m)\cdot(m-1)\cdot\rk_{\Lambda}\cL    \ .
$$
This concludes the proof of \cref{<chiprN0F<}.
\end{proof}

\begin{proposition}\label{chi_special_F}
Let $n\geq 2$ be an integer.
Assume that \cref{maintheoremAn} is valid for $\bA^{n-1}_k$. 
Let $\pr:\bA^n_k\to\bA^1_k$ be the first projection. 
Then, for every $\cL\in \LC(\bA^n_k,\Lambda)$, the Euler-Poincar\'e characteristic $\chi(\bA^n_k,\cL)$ is smaller than
$$
\left(\sum^{n-1}_{\substack{i=0\\  \text{$i$ $\odd$ }}}(\lc_H(\cL)+2)\cdot b_{i}(\lc_H(\cL)+3)+\sum^{n-1}_{\substack{i=0\\  \text{$i$ $\even$ }}} (\lc_H(\cL)+3)\cdot b_{i}(\lc_H(\cL))\right)\cdot\rk_{\Lambda}\cL
$$
and bigger than 
$$
-\left(\sum^{n-1}_{\substack{i=0\\  \text{$i$ $\even$ }}}(\lc_H(\cL)+2)\cdot b_{i}(\lc_H(\cL)+3)+\sum^{n-1}_{\substack{i=0\\  \text{$i$ $\odd$ }}}(\lc_H(\cL)+3)\cdot b_{i}(\lc_H(\cL))\right)\cdot\rk_{\Lambda}\cL \ .
$$

\end{proposition}
\begin{proof}
Note that enlarging the finite field $\Lambda$ does not change the Betti numbers. 
Hence, we can suppose that $\Lambda$ contains a $p$-root of unity.
Let $\cL\in \LC(\bA^n_k,\Lambda)$. 
Observe that at least one of the consecutive integers $[\lc_H(\cL)]+2$, $[\lc_H(\cL)]+3$ is  prime to $p$. 
Pick one of them and denote it by $m$.
In particular, 
$$
\lc_H(\cL)+1< m\leq \lc_H(\cL)+3 \ .
$$
Let $\cN_{\lambda}$ be an Artin-Schreier sheaf of $\Lambda$-modules on $\bA^1_k$ defined by the equation $t^p-t=\lambda x^m$ for some $\lambda\in k^\times$. 
By the projection formula and \cref{GOS}, we have 
\begin{align*}
\chi(\bA^n_k,\pr^*\cN_{\lambda}\otimes_\Lambda\cL)&=\chi(\bA^1_k,R\pr_*(\pr^*\cN_{\lambda}\otimes_\Lambda\cL))=\chi(\bA^1_k,\cN_{\lambda}\otimes_\Lambda R\pr_*\cL))\\
\chi(\bA^n_k,\cL)&=\chi(\bA^1_k,R\pr_*\cL)\ .
\end{align*}
Since $\cN_\lambda$ is a rank $1$ locally constant sheaf on $\bA^1_k$, \cref{GOS} gives
\begin{align*}
\chi(\bA^n_k,\cL)-\chi(\bA^n_k,\pr^*\cN_{\lambda}\otimes_\Lambda\cL)=\sw_{\infty}(\cN_{\lambda}\otimes_\Lambda R\pr_*\cL)-\sw_{\infty}(R\pr_*\cL)\ .
\end{align*}
By generic universally locally acyclicity, there is a dense open subset $U\subset \bA^1_k$ such that 
\begin{itemize}\itemsep=0.2cm
\item[1.]
For every geometric point $\bar y\to U$, we have $R\Gamma((\bA^n_k)_{\bar y},\cL|_{(\bA^n_k)_{\bar y}})\cong (R\pr_*\cL)|_{\bar y}$
\item[2.]
For every $i\in \bZ$, the sheaf $R^i\mathrm{pr}_*\cL$ is locally constant and constructible in $U$.
\end{itemize}
Let $\eta_{\infty}=\mathrm{Spec}\,(\mathrm{Frac}(\mathscr O^{\mathrm{sh}}_{\bP^1_k,\infty}))$ and let $\overline{\eta}_{\infty}$ be a geometric point over $\eta_{\infty}$. 
Put  $G_{\infty}=\mathrm{Gal}(\overline{\eta}_\infty/\eta_\infty)$ and let $P_\infty\subset  G_\infty$ be the wild inertia subgroup. 
For $s<0$ and $s>n-1$, we have 
$$
(R^s\pr_*\cL)|_{\overline\eta_\infty}=0 \ .
$$
Let $y\in U$.
Since \cref{maintheoremAn} is valid for $\bA^{n-1}_k$, we have 
$$
\rk_\Lambda((R^s\pr_*\cL)|_{\overline\eta_\infty})=\rk_\Lambda((R^s\pr_*\cL)|_{y})\leq b_s(\lc_H(\cL))\cdot\rk_\Lambda\cL 
$$
for every  $0\leq s\leq n-1$.
For $\lambda_1,\lambda_2\in k^\times$ with $\lambda_1-\lambda_2\not\in\bF_p$, the $P_\infty$-representations $\cN_{\lambda_1}|_{\eta_\infty}$ and $\cN_{\lambda_2}|_{\eta_\infty}$ are not isomorphic. 
Hence, their induced characters (see \cref{ram_character}) 
$$
\chi_{\lambda_1},\chi_{\lambda_2} : G_{\infty,\log}^{(r)}/G_{\infty,\log}^{(r+)} \to \Lambda^{\times}
$$ 
are not isomorphic either.
Since $k$ is infinite, we can thus find $\lambda\in k^\times$ such that $\chi_{\lambda}^{-1}$ does not contribute to  
$$
\bigoplus^{n-1}_{i=0}(R^i\pr_*\cL)|_{\eta_\infty}.
$$ 
By  \cref{swan_inequality}, for every $i=0,1,\dots, n-1$, we thus have 
$$
\sw_\infty(R^i\pr_*\cL) \leq \sw_\infty(\cN_{\lambda}\otimes_\Lambda R^i\pr_*\cL)\leq \sw_\infty(R^i\pr_*\cL)+m\cdot\rk_{\Lambda}((R^i\pr_*\cL)|_{\overline\eta_\infty}).
$$
In particular for $i=0,1,\dots, n-1$, we have 
$$
0\leq \sw_\infty(\cN_{\lambda}\otimes_\Lambda R^i\pr_*\cL)-\mathrm{sw}_\infty(R^i\pr_*\cL) \leq m\cdot\rk_{\Lambda}((R^i\pr_*\cL)|_{\overline\eta_\infty}) \ .
$$

We obtain
\begin{align*}
\chi(\bA^n_k,\cL)-\chi(\bA^n_k,\pr^*\cN_\lambda\otimes_\Lambda\cL)&=\sw_{\infty}(\cN_{\lambda}\otimes_\Lambda R\pr_*\cL)-\sw_{\infty}(R\pr_*\cL)\\
&=\sum^{n-1}_{i=0}(-1)^i(\sw_\infty(\cN_{\lambda}\otimes_\Lambda R^i\pr_*\cL)-\sw_{\infty}(R^i\pr_*\cL))\\
&\leq \sum^{n-1}_{\substack{i=0\\  \text{$i$ $\even$ }}}(\sw_\infty(\cN_{\lambda}\otimes_\Lambda R^{i}\pr_*\cL)-\sw_{\infty}(R^{i}\pr_*\cL))\\
&\leq \sum^{n-1}_{\substack{i=0\\  \text{$i$ $\even$ }}} m\cdot\rk_{\Lambda}(R^{i}\pr_*\cL)\nonumber\\
&\leq\left(\sum^{n-1}_{\substack{i=0\\  \text{$i$ $\even$ }}} b_{i}(\lc_H(\cL))\right)\cdot m\cdot\rk_{\Lambda}\cL \ .
\end{align*}
and
\begin{align*}
\chi(\bA^n_k,\cL)-\chi(\bA^n_k,\pr^*\cN_{\lambda}\otimes_\Lambda\cL)&=\sw_{\infty}(\cN_{\lambda}\otimes_\Lambda R\pr_*\cL)-\sw_{\infty}(R\pr_*\cL)\\
&=\sum^{n-1}_{i=0}(-1)^i(\sw_\infty(\cN_{\lambda}\otimes_\Lambda R^i\pr_*\cL)-\sw_{\infty}(R^i\pr_*\cL))\\
&\geq -\sum^{n-1}_{\substack{i=0\\  \text{$i$ $\odd$ }}}(\sw_\infty(\cN_{\lambda}\otimes_\Lambda R^{i}\pr_*\cL)-\sw_{\infty}(R^{i}\pr_*\cL))\\
&\geq -\sum^{n-1}_{\substack{i=0\\  \text{$i$ $\odd$ }}}m\cdot\rk_{\Lambda}(R^{i}\pr_*\cL)\\
&\geq -\left(\sum^{n-1}_{\substack{i=0\\  \text{$i$ $\odd$ }}}b_{i}(\lc_H(\cL))\right)\cdot m\cdot\rk_{\Lambda}\cL \ .
\end{align*}
Together with  \cref{<chiprN0F<}, we have 
\begin{align*}
&\chi(\bA^n_k,\cL)\geq \chi(\bA^n_k,\pr^*\cN_{\lambda}\otimes_\Lambda\cL)-\left(\sum^{n-1}_{\substack{i=0\\  \text{$i$ $\odd$ }}}b_{i}(\lc_H(\cL))\right)\cdot m\cdot\rk_{\Lambda}\cL \\
&\geq-\left(\sum^{n-1}_{\substack{i=0\\  \text{$i$ $\even$ }}} (m-1)\cdot b_{i}(m)+\sum^{n-1}_{\substack{i=0\\  \text{$i$ $\odd$ }}}m\cdot b_{i}(\lc_H(\cL))\right)\cdot\rk_{\Lambda}\cL \\
&\geq-\left(\sum^{n-1}_{\substack{i=0\\  \text{$i$ $\even$ }}}(\lc_H(\cL)+2)\cdot b_{i}(\lc_H(\cL)+3)+\sum^{n-1}_{\substack{i=0\\  \text{$i$ $\odd$ }}}(\lc_H(\cL)+3)\cdot b_{i}(\lc_H(\cL))\right)\cdot\rk_{\Lambda}\cL\ , 
\end{align*}
and
\begin{align*}
&\chi(\bA^n_k,\cL)\leq \chi(\bA^n_k,\pr^*\cN_{\lambda}\otimes_\Lambda\cL)+\left(\sum^{n-1}_{\substack{i=0\\  \text{$i$ $\even$ }}}b_{i}(\lc_H(\cL))\right)\cdot m\cdot\rk_{\Lambda}\cL \\
&\leq\left(\sum^{n-1}_{\substack{i=0\\  \text{$i$ $\odd$ }}}(m-1)\cdot b_{i}(m)+\sum^{n-1}_{\substack{i=0\\  \text{$i$ $\even$ }}}  m\cdot b_{i}(\lc_H(\cL))\right)\cdot\rk_{\Lambda}\cL \\
&\leq\left(\sum^{n-1}_{\substack{i=0\\  \text{$i$ $\odd$ }}}(\lc_H(\cL)+2)\cdot b_{i}(\lc_H(\cL)+3)+\sum^{n-1}_{\substack{i=0\\  \text{$i$ $\even$ }}} (\lc_H(\cL)+3)\cdot b_{i}(\lc_H(\cL))\right)\cdot\rk_{\Lambda}\cL \ .
\end{align*}
This concludes the proof of \cref{chi_special_F}.
\end{proof}

\subsection{Proof of  \cref{maintheoremAn}}\label{pfmaintheoremAn} 
We proceed by induction on $n$. 
When $n=1$, \cref{maintheoremAn} follows from \cref{GOcor_affine}. 
Let $n\geq 2$ and assume that \cref{maintheoremAn} is valid for $\bA^{n-1}_k$. 
In particular,  \cref{hn-1<bn-1} gives 
\begin{equation*}
h^i(\bA^{n}_k,\cL)\leq b_i(\lc_H(\cL))\cdot\rk_{\Lambda}(\cL)\ ,
\end{equation*}
for every $0\leq i\leq n-1$.
Hence, we are left to show that 
\begin{equation*}
h^n(\bA^{n}_k,\cL)\leq b_n(\lc_H(\cL))\cdot\rk_{\Lambda}(\cL) . 
\end{equation*}
Let $\pr:\bA^{n}_k\to\bA^1_k$ be the projection on the first coordinate. 
When $n$ is odd, we deduce
\begin{align*}
h^{n}(\bA^{n}_k,\cL)&=-\chi(\bA^{n}_k,\cL)+\sum^{n-1}_{i=0}(-1)^ih^i(\bA^{n}_k,\cL),\\
&\leq-\chi(\bA^{n}_k,\cL)+\sum^{n-1}_{\substack{i=0\\  \text{$i$ $\even$ }}} h^{i}(\bA^{n}_k,\cL),\\
&\leq \left(\sum^{n-1}_{\substack{i=0\\  \text{$i$ $\even$ }}}(\lc_H(\cL)+2)\cdot b_{i}(\lc_H(\cL)+3)+\sum^{n-1}_{\substack{i=0\\  \text{$i$ $\odd$ }}}(\lc_H(\cL)+3)\cdot b_{i}(\lc_H(\cL))\right)\cdot\rk_{\Lambda}\cL\\
&\ \ \ +\left(\sum^{n-1}_{\substack{i=0\\  \text{$i$ $\even$ }}}  b_{i}(\lc_H(\cL))\right)\cdot\rk_\Lambda\cL\\
&= b_n(\lc_H(\cL))\cdot\rk_\Lambda \cL\ .
\end{align*}
If $n$ is even, we have 
\begin{align*}
h^{n}(\bA^{n}_k,\cL)&=\chi(\bA^{n}_k,\cL)-\sum^{n-1}_{i=0}(-1)^i h^i(\bA^{n}_k,\cL)\\
&\leq\chi(\bA^{n}_k,\cL)+\sum^{n-1}_{\substack{i=0\\  \text{$i$ $\odd$ }}} h^{i}(\bA^{n}_k,\cL)\\
&\leq \left(\sum^{n-1}_{\substack{i=0\\  \text{$i$ $\odd$ }}}(\lc_H(\cL)+2)\cdot b_{i}(\lc_H(\cL)+3)+\sum^{n-1}_{\substack{i=0\\  \text{$i$ $\even$ }}} (\lc_H(\cL)+3)\cdot b_{i}(\lc_H(\cL))\right)\cdot\rk_{\Lambda}\cL \\
&\ \ \  + \left(\sum^{n-1}_{\substack{i=0\\  \text{$i$ $\odd$ }}}b_{i}(\lc_H(\cL))\right)\cdot\rk_\Lambda\cL\\
&\leq b_n(\lc_H(\cL))\cdot\rk_\Lambda\cL  \ .
\end{align*}
In any case  we have $h^{n}(\bA^{n}_k,\cL)\leq b_{n}(\lc_H(\cL))\cdot\rk_\Lambda\cL $. 
This concludes the proof of \cref{maintheoremAn}.

\section{Estimates for Betti numbers of étale sheaves}\label{estimes_general}

\subsection{Bounding the ramification with coherent sheaves}\label{bounded_ramification_section}

Let $X$ be a scheme of finite type over $k$.
We denote by $\bQ[\Coh(X)]$ the free $\bQ$-vector space on the set of isomorphism classes of coherent sheaves on $X$.
Observe that the pullback along every morphism $f : Y\to X$ of schemes of finite type over $k$ induces a morphism of $\bQ$-vector spaces
$$
f^* : \bQ[\Coh(X)]\to \bQ[\Coh(Y)]\ .
$$
Assume now that $X$ is normal and let $\cE\in \Coh(X)$.
If $X^1\subset X $ denotes the set of codimension $1$ points of $X$, 
we define a Weil divisor on $X$ by the formula
$$
T(\cE):= \sum_{\eta\in X^1}    \length_{\cO_{X,\eta}}( \cE|_{X_{\eta}}^{\tors})\cdot \overline{\{\eta\}}
$$
where $X_{\eta}= \Spec \cO_{X,\eta}$ and where $\cE|_{X_{\eta}}^{\tors}$  is the torsion part of $\cE|_{X_{\eta}}$.

\begin{example}
If $R$ is an effective Cartier divisor of $X$ with ideal sheaf $\cI_R$ and if $\cE=\cO_X/\cI_R$, then $T(\cE)=R$.
\end{example}

If $\Weil(X)_{\bQ}$ is the space of $\bQ$-Weil divisors on $X$, the map $T : \Coh(X)\to \Weil(X)_{\bQ}$ induces a map of $\bQ$-vector spaces
$$
T : \bQ[\Coh(X)]\to \Weil(X)_{\bQ} \ .
$$

\begin{definition}\label{bounded_conductor}
Let $X$ be a scheme of finite type over $k$.
Let $\cK\in D_{ctf}^b(X,\Lambda)$ and  $\cE\in \bQ[\Coh(X)]$.
We say that \textit{$\cK$ has log conductors bounded by $\cE$} if for every morphism $f : C\to X$ over $k$ where $C$ is a smooth curve over $k$, we have
$$
LC(\cH^i\cK|_C)\leq T(f^*\cE) 
$$
for every $i\in \bZ$.
We denote by $D_{ctf}^b(X,\cE,\Lambda)$ the full subcategory of $D_{ctf}^b(X,\Lambda)$ spanned by objects having log conductors bounded by $\cE$.
\end{definition}

The following is our main example of sheaf with explicit bound on the log conductors.

\begin{prop}[{\cite[Proposition 5.7]{HuTeyssierCohBoundedness}}]\label{bounded_ramification_ex}
Let $X$ be a normal scheme of finite type over $k$.
Let $D$ be an effective Cartier divisor of $X$ and put $j : U:=X-D\hookrightarrow X$.
Let  $\cL\in \Loc_{\ft}(U,\Lambda)$ and $\cE\in \bQ[\Coh(X)]$.
Then
\begin{enumerate}\itemsep=0.2cm
\item If $j_!\cL$ has log conductors bounded by $\cE$, then $LC_X(j_!\cL) \leq T(\cE)$.
\item If $X$ is smooth over $k$, then $j_!\cL$ has log conductors bounded by $(\lc_D(\cL)+1)\cdot \cO_D$.
\end{enumerate}

\end{prop}

\begin{lem}[{\cite[Lemma 5.6]{HuTeyssierCohBoundedness}}]\label{many_properties}
Let $X$ be a scheme of finite type over $k$ and let  $\cE\in \bQ[\Coh(X)]$.
\begin{enumerate}\itemsep=0.2cm

\item\label{pullback} For every  $\cK\in D_{ctf}^b(X,\cE,\Lambda)$ and every $f : Y\to X$ morphism of  schemes of finite type over $k$, 
we have $f^*\cK\in D_{ctf}^b(Y,f^*\cE,\Lambda)$.

\item\label{sequence}  Consider an exact sequence  in $\Cons_{tf}(X,\Lambda)$
$$
0\to \cF_1\to \cF_2 \to \cF_3\to 0 \ .
$$
Then $\cF_2$ lies has log conductors bounded by $\cE$ if and only if so do $\cF_1$ and $\cF_3$.
\end{enumerate}
\end{lem}

\begin{lem}[{\cite[Proposition 5.9]{HuTeyssierCohBoundedness}}]\label{bounded_ramification_generic_fiber}
Let $f : X\to S$ be a  morphism between schemes of finite type over $k$.
Let $\cE\in \bQ[\Coh(X)]$ and $\cK\in D_{ctf}^b(X,\cE,\Lambda)$.
For every algebraic geometric point $\sbar\to S$, the complex $\cK|_{X_{\sbar}}$ has log conductors bounded by $i_{\sbar}^*\cE$ where $i_{\sbar} : X_{\sbar} \to X$ is the canonical morphism.

\end{lem}

\subsection{Betti bound}

In this paragraph, we formalize the structure of the estimates appearing in the statement of \cref{general_boundedness_Betti}.

\begin{defin}\label{admissible}
Let $X$ be a scheme of finite type over $k$.
We say that a $\bQ$-linear map $\fc : \bQ[\Coh(X)]\to \bQ$ is \textit{admissible} if the following conditions are satisfied :
\begin{enumerate}\itemsep=0.2cm
\item For every $\cE\in \Coh(X)$, we have $\fc(\cE)\in \bN$.
\item For every $\cE_1,\cE_2\in \Coh(X)$, we have $\fc(\cE_1\bigoplus \cE_2)\leq \fc(\cE_1)+\fc(\cE_2)$.
\end{enumerate}
\end{defin}

\begin{defin}
Following \cite[0391]{SP},  let us recall that a morphism $f : X\to S$ between schemes of finite type over $k$ is \textit{normal} if it is flat with geometrically normal fibres.
\end{defin}

\begin{example}\label{ex_subaddittive}
Let $f: X\to S$ be a normal morphism between schemes of finite type over $k$.
For  $\cE\in \Coh(X)$,  we define a function $\chi_{\cE}\colon S\to\bN$ by
$$
\chi_{\cE}:S\to\bN,\ \ \ s\mapsto m(T(\cE_{\sbar})),
$$
where $\overline s\to S$ is an algebraic geometric point above $s\in S$ and where $m(T(\cE_{\sbar}))$ is the maximal multiplicity of $T(\cE_{\sbar})$.
By \cite[Corollary 4.8]{HuTeyssierCohBoundedness}, the quantity
$$
\mu_{f}(\cE):= \sup \chi_{\cE}(S)
$$
is finite.
Then, the induced $\bQ$-linear map
$$
\fc_f : \bQ[\Coh(X)]\to \bQ
$$ 
 is admissible.
\end{example}

\begin{lem}\label{pullback_subadditive}
Let $f: X\to Y$ be a morphism of schemes of finite type over $k$ and let  $\fc : \bQ[\Coh(X)]\to \bQ$ be an admissible function.
Then, $\fc\circ f^* : \bQ[\Coh(Y)]\to \bQ$ is admissible.
\end{lem}

\begin{defin}
Let $f: X\to S$ be a proper morphism between schemes of finite type over $k$ with fibres of dimension $\leq n$.
Let $\Lambda$ be a finite field or a finite extension of $\bQ_{\ell}$ with $\ell \neq p$.
Let $\sbar\to S$ be an algebraic geometric point and let $\cC \subset D_c^b(X_{\sbar},\Lambda)$ be a full subcategory.
Let  $\fc : \bQ[\Coh(X)]\to \bQ$ be an admissible function  and let $\fd\in \mathds{N}[x]^{n+1}$.
We say that \textit{$(\fc,\fd)$ is a Betti bound for $\cC$} if for every $j=0, \dots, 2n$, every  $\cE\in \bQ[\Coh(X)]$  and every  $\cK\in \cC$ with log conductors bounded by $\cE_{\sbar}$, we have 
\begin{equation}\label{betti_bound}
h^j(X_{\sbar},\cK)\leq \fd_{\min(j,2n-j)}(\fc(\cE))\cdot \Rk_{\Lambda}\cK  \ .
\end{equation}
\end{defin}

\begin{rem}
If \eqref{betti_bound} is satisfied for some algebraic geometric point $\sbar\to S$ localized at $s\in S$, then it is satisfied for every algebraic geometric point $\overline{t}\to S$ localized at $s$.
\end{rem}

\begin{defin}
Let $P$ be a property of morphisms of schemes over $k$ and let $n\geq 0$. 
A $P$-relative stratified scheme $(X/S,\Sigma)$ of relative dimension $\leq n$ will refer to a morphism $X\to S$ between schemes of finite type over $k$ satisfying $P$ such that the fibres of $X\to S$ have dimensions $\leq n$ and where $X$ is endowed with a finite stratification $\Sigma$.
\end{defin}

\begin{theorem}
\label{general_boundedness_Betti}
Let $(X/S,\Sigma)$ be a proper relative stratified scheme of relative dimension $\leq n$.
Then, there is an admissible function $\fc : \bQ[\Coh(X)]\to \bQ$ and  $\fd\in \mathds{N}[x]^{n+1}$ with  $P_i$ of degree  $i$ for every $i=0,\dots, n$ and such that for every algebraic geometric point $\sbar\to S$, every finite field $\Lambda$ of characteristic $\ell\neq p$, the couple $(\fc,\fd)$ is a Betti bound for   
$\Cons_{\Sigma_{\sbar}}(X_{\sbar},\Lambda)$.
\end{theorem}

\begin{defin}\label{Betti_function}
A Betti bound $(\fc,\fd)$ as in \cref{general_boundedness_Betti} will be referred as a \textit{Betti bound for $(X/S,\Sigma)$}.
\end{defin}

For the sake of the proof, we need to formulate a priori weaker

\begin{theorem}
\label{general_boundedness_Betti_dominant}
Let $(X/S,\Sigma)$ be a proper relative  stratified scheme of relative dimension $\leq n$.
Then, there is a dominant quasi-finite map $T\to S$ between schemes of finite type over $k$ such that the pullback $(X_T/T,\Sigma_T)$ admits a Betti bound.
\end{theorem}

\subsection{Dévissage}
The goal of what follows is to show that \cref{general_boundedness_Betti} and \cref{general_boundedness_Betti_dominant} are equivalent and to reduce the proof of \cref{general_boundedness_Betti_dominant} to the case where $X\to S$ is a normal morphism.

\begin{lem}\label{lem_equivalence_coho_bound}
Let $(X/S,\Sigma)$ be a proper relative  stratified scheme  of relative dimension $\leq n$.
Assume the existence of a finite family $A$ of jointly surjective quasi-finite morphisms $T_\alpha \to S$ over $k$ such that for every $\alpha \in A$, the pullback $(X_{\alpha}/S_{\alpha},\Sigma_{\alpha})$ admits a Betti bound.
Then, $(X/S,\Sigma)$ admits a Betti bound.
\end{lem}

\begin{proof}
For $\alpha \in A$,  let $h_{\alpha} : X_{ \alpha }\to X$ be the induced morphism and let $(\fc_{ \alpha },\fd_{ \alpha })$ be a Betti bound for $(X_{\alpha}/S_{\alpha},\Sigma_{\alpha})$.
By \cref{pullback_subadditive}, the function 
$$
\fc:=\sum_{\alpha \in A}\fc_{ \alpha }\circ h^*_\alpha : \bQ[\Coh(X)]\to \bQ
$$ 
is admissible.
If we put $\fd:=\sum_{\alpha \in A}\fd_{ \alpha }$, one easily checks that $(\fc,\fd)$ is a Betti bound for $(X/S,\Sigma)$.
\end{proof}

\begin{defin}
For $d\geq 0$, we say that \cref{general_boundedness_Betti} (resp. \cref{general_boundedness_Betti_dominant}) holds in absolute dimension $\leq d$ if for every $n\geq 0$, \cref{general_boundedness_Betti} (resp. \cref{general_boundedness_Betti_dominant})  holds for every  proper relative stratified scheme $(X/S,\Sigma)$  of relative dimension $\leq n$ with $\dim X\leq d$.
\end{defin}

\begin{lem}\label{equivalence_coho_bound}
Let $d\geq 0$.
Then \cref{general_boundedness_Betti} holds in absolute dimension $\leq d$ if and only if \cref{general_boundedness_Betti_dominant} holds in absolute dimension $\leq d$.
\end{lem}

\begin{proof}
Immediate by \cref{lem_equivalence_coho_bound} and the fact that the base scheme $S$ is noetherian.
\end{proof}

\begin{lem}\label{dimension_recursion}
Let $d\geq 1$ and let $(X/S,\Sigma)$ be a proper relative  stratified scheme of relative dimension $\leq n$ with $\dim X\leq d$.
Assume that

\begin{enumerate}\itemsep=0.2cm
\item 
\cref{general_boundedness_Betti} holds in absolute dimension $\leq d-1$. 
\item There is a closed subscheme $i : Z \hookrightarrow X$ of dimension $\leq d-1$ with complement $j : U\hookrightarrow X$, an admissible function $\fc' : \bQ[\Coh(X)]\to \bQ$ and  $\fd'\in 
\bN[x]^{n+1}$ with $\fd'_i$ of degree  $i$ for $i=0,\dots, n$ such that for every algebraic geometric point $\sbar\to S$, every finite field $\Lambda$ of characteristic $\ell\neq p$,  every $\cE\in \bQ[\Coh(X)]$ and every $\cF\in \Cons(X_{\sbar},\cE_{\sbar},\Lambda)$ extension by 0 of an object of $\Cons_{(j^*\Sigma)_{\sbar}}
(U_{\sbar},\Lambda)$, we have 
$$
h^j(X_{\sbar},\cF)\leq \fd'_{\min(j,2n-j)}(\fc'(\cE))\cdot \Rk_{\Lambda}\cF  
$$
for every $j=0, \dots, 2n$.
\end{enumerate}
Then, $(X/S,\Sigma)$ admits a Betti bound.
\end{lem}

\begin{proof}
Note that the fibres of $Z\to S$ have dimension $\leq n$.
Let $(\fc_Z,\fd_Z)$ be a Betti bound for $(Z/S,i^*\Sigma)$.
Put
$$
\fc=\fc' + \fc_Z \circ i^* : \bQ[\Coh(X)]\to \bQ
 \text{ and } \fd=\fd'  + \fd_Z \in \mathds{N}^{n+1}[x]\
$$
and let us show that $(\fc,\fd)$ is a Betti bound for $(X/S,\Sigma)$.
Let $\sbar\to S$ be an algebraic geometric point and let $\Lambda$  be a finite field of characteristic $\ell\neq p$.
Let  $\cE\in \bQ[\Coh(X)]$ and let $\cF\in \Cons_{\Sigma_{\sbar}}(X_{\sbar},\cE_{\sbar},\Lambda)$.
We want to show that for every $j=0,\dots, 2n$, we have 
\begin{equation}\label{eq_dim_reduction}
h^j(X_{\sbar},\cF)\leq \fd_{\min(j,2n-j)}(\fc(\cE))\cdot \Rk_{\Lambda} \cF\ . 
\end{equation}
Consider the localization exact sequence 
$$
0  \to j_{\sbar!} \cF|_{U_{\sbar}}\to \cF \to i_{\sbar*} \cF|_{Z_{\sbar}}\to 0  \ .
$$
By \cref{many_properties}-(\ref{pullback}), the sheaf $\cF|_{Z_{\sbar}}$ is an object of $\Cons_{(i^*\Sigma)_{\sbar}}(Z_{\sbar},\Lambda)$ with log conductors bounded by $(i^*\cE)_{\sbar}$. 
Hence, for every $j=0,\dots, 2n$, we have
$$
h^j(Z_{\sbar},\cF|_{Z_{\sbar}})\leq \fd_{Z,\min(j,2n-j)}(\fc_Z(i^*\cE))\cdot \Rk_{\Lambda} \cF|_{Z_{\sbar}} \leq \fd_{Z,\min(j,2n-j)}(\fc(\cE))\cdot \Rk_{\Lambda} \cF
\ .
$$
Note that $j_{\sbar!} \cF|_{U_{\sbar}}$ has log conductors bounded by $\cE_{\sbar}$ by \cref{many_properties}-(\ref{sequence}).
By assumption, we deduce that 
$$
h^j(X_{\sbar},j_{\sbar!} \cF|_{U_{\sbar}})\leq \fd'_{\min(j,2n-j)}(\fc'(\cE))\cdot \Rk_{\Lambda} j_{\sbar!} \cF|_{U_{\sbar}} \leq \fd'_{\min(j,2n-j)}(\fc(\cE))\cdot \Rk_{\Lambda} \cF
\ .
$$
The conclusion thus follows.
\end{proof}

\begin{lem}\label{dim_reduction}
Let $d\geq 1$ and let $(X/S,\Sigma)$ be a proper relative  stratified scheme of relative dimension $\leq n$ with $\dim X\leq d$.
Assume that

\begin{enumerate}\itemsep=0.2cm
\item 
\cref{general_boundedness_Betti_dominant} holds in absolute dimension $\leq d-1$. 
\item There is a proper morphism $h : Y \to X$ of schemes of finite type over $S$ such that the fibres of $Y\to S$ have dimension $\leq n$ and \cref{general_boundedness_Betti_dominant} holds for  $(Y/S,\Sigma')$ where $\Sigma'$ is any stratification on $Y$.

\item $h$ induces an isomorphism above a dense open subset $U\subset X$.
\end{enumerate}
Then, \cref{general_boundedness_Betti_dominant} holds for $(X/S,\Sigma)$.
\end{lem}

\begin{proof}
Let $\Sigma'$ be a refinement of $\Sigma$ such that $U$ and $Z:= X-U$ are unions of strata of $\Sigma'$.
It is enough to show that \cref{general_boundedness_Betti_dominant} holds for $(X/S,\Sigma')$.
Hence, we are left to prove \cref{dim_reduction} in the case where  $U$ and $Z:= X-U$ are unions of strata of $\Sigma$. \\ \indent
Let $g:T\to S$ be a quasi-finite dominant map as given by \cref{general_boundedness_Betti_dominant} for $(Y/S,h^*\Sigma)$.
At the cost of pulling everything back to $T$, we can suppose that $(Y/S,h^*\Sigma)$ admits a Betti bound.
In that case, we are going to show that so does $(X,\Sigma)$.\\ \indent
By (1) combined with \cref{equivalence_coho_bound}, we know that \cref{general_boundedness_Betti} holds in absolute dimension $\leq d-1$. 
Hence, we are left to show that  $Z$ satisfies the condition (2) of \cref{dimension_recursion}.
Since $U$ is dense open in $X$, we have $\dim Z<\dim X=d$.
Let $\sbar\to S$ be an algebraic geometric point, let $\Lambda$  be a finite field of characteristic $\ell\neq p$, let $\cE\in \bQ[\Coh(X)]$  and let $\cF\in \Cons(X_{\sbar},\cE_{\sbar},\Lambda)$ extension by 0 of an object of $\Cons_{(j^*\Sigma)_{\sbar}}
(U_{\sbar},\Lambda)$.
Since $U$ and $Z$ are unions of strata, $\cF$ is an object of $\Cons_{\Sigma_{\sbar}}(X_{\sbar},\cE_{\sbar},\Lambda)$.
Hence,   $h^*_{\sbar}\cF$ is an object of  $\Cons_{(h^{\ast}\Sigma)_{\sbar}}(Y_{\sbar},(h^*\cE)_{\sbar},\Lambda)$.
Since $h_{\sbar} : Y_{\sbar}\to X_{\sbar}$ is an isomorphism above $U_{\sbar}$, the proper base change implies that the unit map 
$$
\cF \to Rh_{\sbar*}   h^*_{\sbar}\cF
$$ 
is an isomorphism.
Thus, 
$$
h^j(X_{\sbar},\cF) =  h^j(Y_{\sbar},h^*_{\sbar}\cF) \ .
$$ 
Hence if  $(\fc_Y,\fd_Y)$ is a Betti bound for  $(Y/S,h^*\Sigma)$,  we deduce
$$
h^j(X_{\sbar},\cF)\leq \fd_{Y,\min(j,2n-j)}(\fc_Y(h^*\cE))\cdot \Rk_{\Lambda}\cF  \ .
$$
Hence, \cref{dimension_recursion} is satisfied with $\fc'=\fc_Y \circ h^*$ and $\fd'=\fd_Y$.
\end{proof}

\begin{lem}\label{getting_projective}
Let $f : X\to S$ be a proper morphism of schemes of finite type over $k$ of relative dimension $\leq n$.
Then there is a dense open subset $V\subset S$ and a commutative triangle
\begin{equation}\label{cd_getting_projective}
	\begin{tikzcd}
		Y\arrow{rr}{h} \arrow{dr}{g} && X_V  \arrow{ld}{f_V} \\
	&V& 
	\end{tikzcd}
\end{equation}
 where $h : Y\to X_V$ is surjective projective and induces an isomorphism over a dense open subset $U\subset X_V$  with $h^{-1}(U)$ dense in $Y$ and  $g:Y\to V$ projective  of relative dimension $\leq n$.
\end{lem}
\begin{proof}
At the cost of shrinking $S$, we can suppose that $S$ is irreducible with generic point $\eta$.
At the cost of shrinking $S$ further, we can suppose that the generic points of $X$ lie over $\eta$.
By the Chow lemma \cite[Corollaire 5.6.2]{EGA2-2}, there is a surjective projective morphism $h: Y\to X$ over $S$ such that $h$ is an isomorphism over a dense open subset $U\subset X$ with $h^{-1}(U)$ dense in $Y$ and such that $Y\to S$ is projective.
In particular, for a generic point $\xi \in Y$, we necessarily have $\xi \in h^{-1}(U)$.
Hence, there is a generic point $\nu \in U$ specializing on $h(\xi)$.
Thus, $h^{-1}(\nu)$ specializes on $\xi$ and we get $h(\xi)=\nu$, so that $\xi$ lies over $\eta$ as well.
Hence, $h^{-1}(U)_{\eta}$ is dense in  $Y_{\eta}$.
By \cite[Tag 0573]{SP}, there is a dense open subset $V\subset S$ such that for every $s\in V$, the open set $(h^{-1}(U))_{s}$ is  dense in $Y_{s}$. 
Thus, for every $s\in V$, we have
$$
\dim Y_s= \dim (h^{-1}(U))_{s} = \dim U_{s}\leq n \ .
$$ 
The conclusion thus follows.
\end{proof}

\begin{lem}\label{proj_reduction}
Let $d\geq 1$ and assume that 
\begin{enumerate}\itemsep=0.2cm
\item 
\cref{general_boundedness_Betti_dominant} holds in absolute dimension $\leq d-1$. 
\item For every $n\geq 0$, \cref{general_boundedness_Betti_dominant}  holds for every projective relative stratified scheme $(X/S,\Sigma)$ of relative dimension $\leq n$ with $\dim X\leq d$.
\end{enumerate}
Then, \cref{general_boundedness_Betti_dominant} holds in absolute dimension $\leq d$.
\end{lem}

\begin{proof}
Let $n\geq 0$ and let $(X/S,\Sigma)$ be a proper relative  stratified scheme of relative dimension $\leq n$ with $\dim X\leq d$.
At the cost of replacing $S$ by a dense open subset, we can assume by \cref{getting_projective} the existence of a proper morphism $h : Y\to X$ inducing an isomorphism over a dense open subset $U\subset X$  with $h^{-1}(U)$ dense in $Y$ and  $g:Y\to S$ projective  of relative dimension $\leq n$.
In particular, 
$$
\dim Y=\dim h^{-1}(U)=\dim U =\dim X \leq d \ .
$$
By (2),  \cref{general_boundedness_Betti_dominant}  holds for $(Y/S,\Sigma')$ where $\Sigma'$ is any stratification on $Y$.
Then \cref{proj_reduction} follows from \cref{dim_reduction}.
\end{proof}

\begin{lem}\label{normal_reduction}
Let $d\geq 1$ and assume that 
\begin{enumerate}\itemsep=0.2cm
\item 
\cref{general_boundedness_Betti_dominant} holds in absolute dimension $\leq d-1$. 
\item For every $n\geq 0$, \cref{general_boundedness_Betti_dominant}  holds for every projective normal relative stratified scheme  $(X/S,\Sigma)$ of relative dimension $\leq n$ with $\dim X\leq d$ and $S$ integral affine over $k$.
\end{enumerate}
Then, \cref{general_boundedness_Betti_dominant} holds in absolute dimension $\leq d$.
\end{lem}

\begin{proof}
By \cref{proj_reduction},  it is enough to show that for every $n\geq 0$, \cref{general_boundedness_Betti_dominant}  holds for every projective relative stratified scheme $(X/S,\Sigma)$ of relative dimension $\leq n$ with $\dim X\leq d$.
At the cost of replacing $S$ by a dense open subset, we can assume that $S$ is irreducible.
Let $\eta$ be the generic point of $S$ and let $\overline{\eta}\to S$ be an algebraic geometric point lying over $\eta$.\\ \indent
Consider  the normalization map $\alpha_{\overline{\eta}} :      \widetilde{X}^{\red}_{\overline{\eta}}\to  X_{\overline{\eta}}^{\red}$  and the reduction map $\beta_{\overline{\eta}} :  X_{\overline{\eta}}^{\red}\to X_{\overline{\eta}}$.
By spreading out $\alpha_{\overline{\eta}}$ and $\beta_{\overline{\eta}} $,  there is a commutative diagram with cartesian squares 
$$
	\begin{tikzcd}	         
	\widetilde{X}^{\red}_{\overline{\eta}}  \arrow{r} \arrow{d}{\alpha_{\overline{\eta}}}    & Y_T \arrow{d}{\alpha}     &  &  \\
	          X^{\red}_{\overline{\eta}}  \arrow{r} \arrow{d}{\beta_{\overline{\eta}} }    & Z_T \arrow{d}{\beta}     &  &  \\
	           X_{\overline{\eta}}                 \arrow[r,crossing over]  \arrow{d}{f_{\overline{\eta}}                } &   X_T \arrow{d}{f_T}    \arrow{r}  & X \arrow{d}{f}      \\
                    \overline{\eta}        	 \arrow{r}{ } &       T        \arrow{r}{ }      & S  
	\end{tikzcd}
$$
where $T\to S$ is quasi-finite flat with $T$ integral affine over $k$. 
Since $\beta_{\overline{\eta}} : X_{\overline{\eta}}^{\red} \to X_{\overline{\eta}}$ is radicial, finite and surjective, \cite[Théorème 8.10.5]{EGAIV}
implies that at the cost of shrinking $T$, we can suppose that $\beta$ is radicial, finite and surjective as well.
In particular,  $\beta$ is a universal homeomorphism by \cite[Corollaire 18.12.11]{EGAIV}.
Since $\alpha_{\overline{\eta}} : \widetilde{X}^{\red}_{\overline{\eta}}\to  X_{\overline{\eta}}^{\red}$ is finite and an isomorphism over a dense open subset of $U_{\overline{\eta}}\subset X_{\overline{\eta}}^{\red}$ with $\alpha_{\overline{\eta}}^{-1}(U_{\overline{\eta}})$ dense in $\widetilde{X}^{\red}_{\overline{\eta}}$,  we can similarly at the cost of replacing $T$ by a further quasi-finite flat scheme $T'$ over $T$ integral affine over $k$ suppose that $\alpha$ is finite and an isomorphism over a dense open subset of $U\subset Z_T$ with $\alpha^{-1}(U)$ dense in $Y_T$. \\ \indent
To show that \cref{general_boundedness_Betti_dominant} holds for $(X/S,\Sigma)$,  it is enough to show that 
it holds for $(X_T/T,\Sigma_T)$.
By invariance of the étale topos under universal homeomorphism, it is enough to show that it holds for $(Z_T/T,\beta^*\Sigma_T)$.
By \cref{dim_reduction}, it is enough to show that it holds for $(Y_T/T,\Sigma')$ where $\Sigma'$ is any stratification on $Y_T$.
We have
$$
\dim Y_T=\dim \alpha^{-1}(U)=\dim U =\dim Z_T=\dim X_T \leq \dim X \leq d \ .
$$
Note that $Y_T$ is projective over $T$ of relative dimension $\leq n$.
Hence,  we are left to prove that for every $n\geq 0$, \cref{general_boundedness_Betti_dominant}  holds for every projective relative stratified scheme $(X/S,\Sigma)$ of relative dimension $\leq n$ with $\dim X\leq d$, with $S$ integral affine over $k$ where $X_{\overline{\eta}}$ is normal.
Since $S$ is reduced, at the cost of shrinking $S$, we can suppose by generic flatness  that $X$ is flat over $S$.
By  \cite[Théorème 12.2.4]{EGAIV} and at the cost of shrinking $S$ further, we can suppose that the morphism $X\to S$ is normal.
The conclusion then follows by assumption (2).
\end{proof}

\begin{lem}\label{spread_out_Kedlaya}
Let $f : X\to S$ be a projective morphism of schemes of finite type over $k$  with $S$ irreducible and with geometrically reduced generic fibre of pure dimension $n$.
Let $D\subset X$ be a divisor.
Then, there is a dense open subset $V\subset S$ and a finite surjective morphism $h : X_V\to \bP_{V}^n$ sending $D_V$ to the hyperplane $H_V$ such that $h : X_V\to \bP_{V}^n$ is finite étale over $\bA_{V}^n:= \bP_{V}^n-H_V$.
\end{lem}

\begin{proof}
At the cost of shrinking $S$, we can suppose that every generic point of $D$ is mapped to the generic point $\eta$ of $S$.
By \cite[Theorem 1]{Ked}, there exists a finite surjective morphism $h_{\eta} : X_{\eta}\to \bP_{\eta}^n$ sending $D_{\eta}$ in a hyperplane  $H_{\eta}\subset \bP_{\eta} $ and such that $h_{\eta} : X_{\eta}\to \bP_{\eta}$ is finite étale above  $\bA_{\eta}^n :=\bP_{\eta}^n-H_{\eta}$.
At the cost of shrinking $S$, we can suppose that the morphism $h_{\eta} : X_{\eta}\to \bP_{\eta}^n$ spreads out as a finite surjective morphism $h : X\to \bP_{S}^n$ over $S$ such that $h : X\to \bP_{S}$  is finite étale above $\bA_{S}^n:=\bP_{S}^n-H_{S}$.
Since every generic point of $D$ is mapped to $\eta$ and since $h_{\eta}$ sends $D_{\eta}$ in  $H_{\eta}$, we have $D\subset h^{-1}(H_S)$. 
\cref{spread_out_Kedlaya} is thus proved.
\end{proof}

\begin{lem}\label{bette_polynomials}
Let $n\geq 0$ and let  $\fd\in \bN[x]^{2n}$ with  $\fd_i$ of degree smaller than $\min(i,2n-i)$ for $i=0,\dots, 2n$.
Then, there is  $\fd'\in \bN[x]^{n}$ with  $\fd'_i$ of degree $i$ for $i=0,\dots, n$ such that for every $i=0,\dots, 2n$ we have $\fd_i\leq \fd'_{\min(i,2n-i)}$.
\end{lem}

\begin{proof}
For $0\leq i \leq n,$ put $\fd'_i = x^i +  \fd_i+  \fd_{2n-i}$.
\end{proof}

\subsection{Proof of \cref{general_boundedness_Betti}} 
We argue by recursion on $d$ that \cref{general_boundedness_Betti} holds in absolute dimension $\leq d$.
By \cref{equivalence_coho_bound}, it is enough to show that \cref{general_boundedness_Betti_dominant}
holds in absolute dimension $\leq d$.
If $d=0$, the claim is obvious.
Assume that $d\geq 1$.
By \cref{normal_reduction}, it is enough to show that for every $n\geq 0$, \cref{general_boundedness_Betti_dominant}  holds for every projective normal relative stratified scheme $(X/S,\Sigma)$  of relative dimension $\leq n$ with $\dim X\leq d$ and $S$ integral affine over $k$ .\\ \indent
In that case, $X_\eta$ is normal.
By \cite[Tag 0357]{SP}, $X_\eta$  is thus a disjoint union of normal integral schemes $X_{1,\eta},\dots, X_{m,\eta}$.
For $i=1,\dots, m$, let $X_i$ be the closure of $X_{i,\eta}$ in $X$.
At the cost of shrinking $S$, we can suppose by \cite[Tag 054Y]{SP} that $X=\bigcup_{i=1}^m X_i$.
Note that for every $i\neq j$, the intersection $X_i\cap X_j$ does not meet $X_\eta$.
Hence, at the cost of shrinking $S$ further, we have $X=\bigsqcup_{i=1}^m X_i$.
Thus, at the cost of considering the $f|_{X_i} : X_i\to S$, we can  further suppose that the generic fibre of $f : X \to S$ is integral.
Let $0\leq N \leq n$ be its dimension.
At the cost of shrinking $S$, we can also suppose that $X$ is integral.
\\ \indent
Let $D\subset X$ be a divisor containing the strata of $\Sigma$ of dimension $<\dim X$.
By \cref{spread_out_Kedlaya}, at the cost of shrinking $S$, we can suppose the existence of a finite surjective morphism $h : X\to \bP_{S}^N$ sending $D$ to the hyperplane $H_S$ such that $h : X\to \bP_{S}^N$ is finite étale over $\bA_{S}^N:= \bP_{S}^N-H_S$.
Put $i:Z:= h^{-1}(H_{S})\hookrightarrow X$ and note that $\dim Z<\dim X\leq d$. 
To conclude, it is enough to check that $Z$ satisfies the conditions of  \cref{dimension_recursion}.
Let $j : U:=h^{-1}(\bA_{S}^N)\hookrightarrow X$ be the inclusion.
Let $\sbar\to S$ be an algebraic geometric point, let $\Lambda$ be a finite field  of characteristic $\ell\neq p$, let $\cE\in \bQ[\Coh(X)]$  and let $\cF\in \Cons(X_{\sbar},\cE_{\sbar},\Lambda)$ extension by 0 of an object of $\Cons_{(j^*\Sigma)_{\sbar}}
(U_{\sbar},\Lambda)=\Loc(U_{\sbar},\Lambda)$, where the last equality follows from the fact that all the strata of $j^*\Sigma$ are open subsets of $U$.
Since finite direct images are exact, we have 
$$
h^{\bullet}(X_{\sbar},\cF)=h^{\bullet}(\bP_{\sbar}^N, h_{\sbar\ast} \cF)  
$$
where $h_{\sbar\ast} \cF$ is the extension by $0$ of a locally constant constructible sheaf on $\bA_{\sbar}^N$.
Hence, 
$$
h^j(X_{\sbar},\cF)= 0
$$
for  $0\leq j <N$ and $j>2N$ and if $\delta\geq 1$ is the generic degree of $h : X \to  \bP_S^N$,  \cref{maintheoremAn_compact} implies that 
$$
h^j(X_{\sbar},\cF) \leq   \fb_{2N-j}(\lc_{H_{\sbar}}(h_{\sbar\ast} \cF))\cdot \delta\cdot \Rk_{\Lambda}\cF  
$$
for $N\leq j \leq 2N$.
By \cref{remark_bound_conductor_direct_image} applied to  $h_{\sbar} : X_{\sbar}\to \bP_{\sbar}$, we have
$$
\lc_{H_{\sbar}}(h_{\sbar\ast} \cF)\leq \lc_{H_{\sbar}}(h_{\sbar*}(j_{\sbar!}\Lambda)) +  \delta\cdot  \lc_{Z_{\sbar}}(\cF)  \ .
$$
By \cite[Corollary 5.8]{HuTeyssierSemicontinuity} applied to $\bP^N_S\to S$ and $H_S\subset \bP^N_S$ and to the sheaf $h_* j_!\Lambda$, we have furthermore 
$$
\lc_{H_{\sbar}}(h_{\sbar*}j_{\sbar!}\Lambda)    \leq   \lc_{H_{\overline{\eta}}}(h_{\overline{\eta}*}j_{\overline{\eta}!}\Lambda) \ .
$$
Put $\alpha := \lceil \lc_{H_{\overline{\eta}}}(h_{\overline{\eta}*}j_{\overline{\eta}!}\Lambda) \rceil\in \bN$. 
On the other hands, \cref{bounded_ramification_ex} applied to the normal scheme  $X_{\sbar}$ and to the effective Cartier divisor $Z_{\sbar}$ gives
$$
\lc_{Z_{\sbar}}(\cF) \leq c(T(\cE_{\sbar})) \leq \fc_f(\cE)
$$
where $c(T(\cE_{\sbar}))$ is the maximal multiplicity of $T(\cE_{\sbar})$ and where $\fc_f : \bQ[\Coh(X)]\to \bQ$ is defined in \cref{ex_subaddittive}.
Hence,  for $N\leq j \leq 2N$  we have
$$
h^j(X_{\sbar},\cF) \leq   \fb_{2N-j}(\alpha +  \delta\cdot  \fc_f(\cE) )\cdot \delta\cdot \Rk_{\Lambda}\cF   \ .
$$
If we put $\fd_j = 0$ for $j\in [0,2n]\setminus [N,2N]$ and $\fd_j= \delta \cdot \fb_{2N-j}(\alpha +  \delta\cdot (-) )\in \bN[x]$ for $N\leq j\leq 2N$,  we have
$$
h^j(X_{\sbar},\cF) \leq \fd_j(\fc_f(\cE))\cdot \Rk_{\Lambda}\cF 
$$
for $0\leq j \leq 2n$.
 \cref{bette_polynomials} gives the existence of $\fd'\in \bN[x]^{n}$ with  $\fd'_i$ of degree $i$ for $i=0,\dots, n$ such that for every $i=0,\dots, 2n$ we have $\fd_i\leq \fd'_{\min(i,2n-i)}$.
Hence, for  $j=0,\dots, 2n$,  we have 
$$
h^j(X_{\sbar},\cF) \leq  \fd'_{\min(j,2n-j)}(\fc_f(\cE))\cdot \Rk_{\Lambda}\cF  \ .
$$
The proof of \cref{general_boundedness_Betti} is thus complete.

\begin{cor}\label{general_boundedness_complex}
Let $(X/S,\Sigma)$ be a proper relative  stratified scheme of relative dimension $\leq n$ and let $a\leq b$ be integers.
Then, there is an admissible function $\fc : \bQ[\Coh(X)]\to \bQ$ and  $\fd\in \mathds{N}[x]$ of degree $n$ such that for every algebraic geometric point $\sbar\to S$, every finite field $\Lambda$ of characteristic $\ell\neq p$,  every $\cE\in \bQ[\Coh(X)]$ and every  $\cK\in D_{\Sigma_{\sbar}}^{[a,b]}(X_{\sbar},\cE_{\sbar},\Lambda)$, we have 
$$
\sum_{j\in \bZ} h^j(X_{\sbar},\cK)\leq \fd(\fc(\cE))\cdot \Rk_{\Lambda}\cK  \ .
$$
\end{cor}
\begin{proof}
Immediate from \cref{general_boundedness_Betti} using the hypercohomology spectral sequence.
\end{proof}

\begin{cor}\label{Betti_direct_image}
Let $(X/S,\Sigma)$  be a proper relative stratified scheme of relative dimension $\leq n$ and let $a\leq b$ be integers.
Then, there is an admissible function $\fc :  \bQ[\Coh(X)]\to \bQ$ and $\fd\in \mathds{N}[x]$ of degree $n$ such that for every $\cE\in \bQ[\Coh(X)]$,  every finite field $\Lambda$ of characteristic $\ell\neq p$ and  every $\cK\in D_{\Sigma}^{[a,b]}(X,\cE,\Lambda)$,  we have 
$$
Rk_{\Lambda} Rf_*\cK \leq \fd(\fc(\cE))\cdot \Rk_{\Lambda}\cK  \ .
$$
\end{cor}

\begin{proof}
Immediate by proper base change theorem combined with \cref{bounded_ramification_generic_fiber} and  \cref{general_boundedness_complex}.
\end{proof}

\begin{thm}\label{general_boundedness_Perv}
Let $(X,\Sigma)$ be a proper   stratified scheme  of dimension $n$ over an algebraically closed field $k$.
Then, there is an admissible function $\fc : \bQ[\Coh(X)]\to \bQ$ and  $\fd\in \mathds{N}[x]^{n+1}$ with $\fd_i$ of degree $i$ for $i=0,\dots, n$ such that for every finite field $\Lambda$ of characteristic $\ell\neq p$, every $\cE\in \bQ[\Coh(X)]$,  every $j=-n, \dots, n$ and every  $\cP\in \Perv_{\Sigma}(X,\cE,\Lambda)$, we have 
$$
h^j(X,\cP)\leq \fd_{\min(n+j,n-j)}(\fc(\cE))\cdot \Rk_{\Lambda}\cP  \ .
$$
\end{thm}
\begin{proof}
For $0\leq i \leq n$,  let $\iota_i : X_i \hookrightarrow X$ be the closure of the union of the strata of $\Sigma$  of dimension at most $i$.
Then,   $\cH^{-i} \cP $ is supported on $X_i$ and its restriction to $X_i$ is an object of $\Cons_{\iota_i^*\Sigma}(X_i,\iota_i^*\cE,\Lambda)$.
If $(\fc^i,\fd^i)$  is a Betti bound for $(X_i,\iota_i^*\Sigma)$ as in \cref{general_boundedness_Betti}, we thus have
$$
H^j(X,\cH^{-i}\cP) \leq  \fd^i_{\min(j,2i-j)}(\fc^i(\iota^*_i\cE))\cdot \Rk_{\Lambda}\cH^{-i} \cP 
$$
for every $0\leq j \leq 2i$.
Consider the admissible function 
$$
\fc := \sum_{i=1}^n \fc^i\circ \iota^*_i : \bQ[\Coh(X)]\to \bQ \ .
$$
The hypercohomology spectral sequence reads
$$
H^p(X,\cH^q\cP) \Rightarrow H^{p+q}(X,\cP) \ .
$$
Assume that $0\leq i\leq n$.
Then,  the contributions to $H^{-i}(X,\cP)$ come from $H^{m}(X,\cH^{-i-m} \cP)$ for $0\leq m \leq n-i$.
Thus, 
\begin{align*}
h^{-i}(X,\cP) &   \leq  \sum_{m=0}^{n-i}       h^{m}(X,\cH^{-i-m} \cP)          \\
                    & \leq  \sum_{m=0}^{n-i}\fd^{i+m}_{\min(m,2i+m)}(\fc^{i+m}(\iota^*_{i+m}\cE))\cdot \Rk_{\Lambda}\cH^{-i-m} \cP \\
                 &  \leq \sum_{m=0}^{n-i}\fd^{i+m}_m(\fc(\cE))\cdot \Rk_{\Lambda}\cP \ .
\end{align*}
Observe that  $\mathfrak{e}_i:=\sum_{m=0}^{n-i}\fd^{i+m}_m$ has degree $n-i=\min(n-i,n+i)$.
On the other hand, the contributions to $H^{i}(X,\cP)$ come from $H^{m}(X,\cH^{i-m} \cP)$ for $i\leq m \leq n+i$.
Thus, 
\begin{align*}
h^{i}(X,\cP) &   \leq  \sum_{m=i}^{n+i}       h^{m}(X,\cH^{i-m} \cP)          \\
                    & \leq  \sum_{m=i}^{n+i}\fd^{m-i}_{\min(m,m-2i)}(\fc^{m-i}(\iota^*_{m-i}\cE))\cdot \Rk_{\Lambda}\cH^{i-m} \cP \\
                 &  \leq \sum_{m=2i}^{n+i}\fd^{i+m}_{m-2i}(\fc(\cE))\cdot \Rk_{\Lambda}\cP \ .
\end{align*}
Observe that  $\mathfrak{f}_i:= \sum_{m=2i}^{n+i}\fd^{i+m}_{m-2i}$ has degree $n-i=\min(n+i,n-i)$.
Thus, the admissible function $\fc : \bQ[\Coh(X)]\to \bQ$ and the sequence
$$
\fd:=(\mathfrak{e}_n+\mathfrak{f}_n,\dots,\mathfrak{e}_0+\mathfrak{f}_0) \in \bN[x]^{n+1}
$$
do the job.
\end{proof}

\begin{rem}\label{support_condition}
\cref{general_boundedness_Perv} holds more generally for complexes satisfying the support condition.
\end{rem}

When $X$ is smooth and $\cK$ is the extension by 0 of a locally constant constructible sheaf on the complement of an effective Cartier divisor, \cref{general_boundedness_Perv} implies the following

\begin{theorem}\label{general_boundedness_Betti_Xsmooth}
Let $X$ be a proper  smooth scheme of finite type of dimension $n$ over an algebraically closed field $k$.
Let $D$ be a reduced effective Cartier divisor of $X$ and put $j : U:=X-D\hookrightarrow X$.
Then, there exists $\fd\in \mathds{N}[x]^{n+1}$ with $\fd_i$ of degree $i$ for $i=0,\dots, n$ such that for every finite field $\Lambda$ of characteristic $\ell\neq p$, every $j=0, \dots,2n$ and every  $\cL\in \Loc(U,\Lambda)$  we have 
$$
h^j(U,\cL)\leq \fd_{\min(j,2n-j)}(lc_D(\cL))\cdot \rk_{\Lambda}\cL  \ .
$$
\end{theorem}
\begin{proof}
By Poincaré duality, it is enough to consider the analogous statement for $h^j_c$.
Put $\Sigma:=\{U,D\}$ and let $j : U\hookrightarrow X$ be the inclusion. 
Let $(\fc,\fd)$ be a Betti bound for $(X,\Sigma)$ whose existence follows from \cref{general_boundedness_Betti}.
Since $X$ is smooth over $k$, \cref{bounded_ramification_ex}-2 implies that $j_!\cL$ has log conductors bounded by 
$(\lc_D(\cL)+1)\cdot \cO_D$.
On the other hand, 
$$
\fc((\lc_D(\cL)+1)\cdot \cO_D)=  \fc(\cO_D)\cdot (\lc_D(\cL)+1)\ .
$$
We conclude by using the sequence $\fd( \fc(\cO_D)\cdot((-)+1))$ instead of $\fd$.
\end{proof}

\begin{recollection}\label{adic_presentation}
Let $X$ be a scheme of finite type over a field $k$ of characteristic $p>0$ separably closed or finite and let $\ell\neq p$. 
Then, for every $\cK\in D^b_c(X,\overline{\bQ}_\ell)$, there is a finite extension $L/\bQ_{\ell}$ and an integral representative $\cK_{\bullet}=(\cK_m)_{m\geq 0}$ for $\cK$.
If we put $\Lambda_m := \cO_L/\mathfrak{m}_L^m$, the sheaf $\cK_m$ is an object of $D^b_{cft}(X,\Lambda_m)$ such that $\Lambda_m\otimes_{\Lambda_{m+1}}^{L} \cK_{m+1} \simeq \cK_{m}$.
\end{recollection}

The following definition upgrades \cref{bounded_conductor} to $\Qlbar$-coefficients.

\begin{defin}
Let $X$ be a scheme of finite type over a field $k$ of characteristic $p>0$ separably closed or finite and let $\ell\neq p$. 
Let $\cK\in D_{c}^b(X,\Qlbar)$ and  $\cE\in \bQ[\Coh(X)]$.
We say that \textit{$\cK$ has log conductors bounded by $\cE$} if there is a finite extension $L/\bQ_{\ell}$ and an integral representative $\cK_{\bullet}=(\cK_m)_{m\geq 0}$ for $\cK$ such that for every $m\geq 0$, the complex $\cK_m\in D^b_{cft}(X,\Lambda_m)$ has log conductors bounded by $\cE$ in the sense of \cref{bounded_conductor}.
We denote by $D_c^b(X,\cE,\Qlbar)$ the full subcategory of $D_c^b(X,\Qlbar)$ spanned by objects having log conductors bounded by $\cE$.
\end{defin}

\begin{rem}
\cref{general_boundedness_Betti}, \cref{general_boundedness_complex}, \cref{general_boundedness_Perv} and \cref{general_boundedness_Betti_Xsmooth} admit immediate variants for $\Qlbar$ coefficients with the help of the next lemma.
\end{rem}

\begin{lem}[{\cite[Lemma 1.8]{FuWan}}]\label{BettiFl>BettiQl}
Let $X$ be a separated scheme of finite type over an algebraically closed field $k$. 
For every $\mathcal K\in D^b_c(X,\overline{\mathbb Q}_{\ell})$ represented by $\cK_{\bullet}$ in the sense of Recollection 7.35, we have
$$
\dim_{\Lambda_0}H^n(X,\mathcal K_0)\geq \dim_{\overline{\mathbb Q}_\ell}H^n(X,\mathcal K) \ .
$$
for any $n\in\mathbb Z$.
\end{lem}
\begin{proof}
By \cite[10.1.5, 10.1.17]{fu}, there is a bounded complex $M$ of free $\mathcal O_L$-modules of finite ranks such that, for any $i\geq 0$, the complex $M_i=M\otimes_{O_L}\Lambda_i$ represents $R\Gamma(X,\mathcal K_{i})$. 
Let $\pi_L$ be a uniformizer of $\mathcal O_L$. 
We take an exact sequence of bounded complexes of $\mathcal O_L$-modules
$$
0\to M\xrightarrow{\cdot \pi_L}M\to M_0\to 0 \ .
$$
Then, for each integer $n\geq 0$, we have a long exact sequence
$$
H^{n-1}(X,\mathcal K_0)\to H^n(M)\xrightarrow{\cdot\pi_L} H^n(M)\to H^{n}(X,\mathcal K_0) \ ,
$$
which induces an injection $H^n(M)\otimes_{\mathcal O_L}\Lambda_0\to H^{n}(X,\mathcal K_0)$. 
Hence, for any $n\in \mathbb Z$, we have 
\begin{align*}
\dim_{\Lambda_0}H^n(X,\mathcal K_0)&\geq \dim_{\Lambda}(H^n(M)\otimes_{\mathcal O_L}\Lambda_0)\\
&\geq \dim_{\overline{\mathbb Q}_\ell}(H^n(M)\otimes_{\mathcal O_L}\overline{\mathbb Q}_{\ell})\\
&=\dim_{\overline{\mathbb Q}_\ell}(H^n(X,\mathcal K)) \ .
\end{align*}

\end{proof}







\section{Betti numbers of inverse and higher direct images}\label{inverse_direct_image}

\begin{recollection}
Let $X$ be a projective scheme of finite type over $k$ and let $i : X\hookrightarrow \bP_k(E)$ be a closed immersion.
For $1\leq a\leq \rk E$, we let $\overline{\eta}$ be a geometric generic point of $\bG_k(E,a)$ and let $F\subset \bP_{\overline{\eta}}(E_{\overline{\eta}})$ be the corresponding projective subspace of dimension $a-1$ over $\overline{\eta}$.
Following \cite[Definition 6.3]{SFFK}, we define the \textit{complexity to $i$ of $\cK\in D^b_c(X,\Qlbar)$} by
$$
c_i(\cK) := \max_{1\leq a \leq n}\sum_{j\in \bZ} h^{j}(X_{\overline{\eta}}\cap F,\cK|_{X_{\overline{\eta}}\cap F}) \ .
$$
In particular, we have 
$$
\sum_{j\in \bZ} h^{j}(X,\cK) \leq c_i(\cK) \ .
$$

\end{recollection}

\begin{lem}\label{complexity_trivial_sheaf}
Let $f : X\to S$ be a projective morphism  over $k$ and let $i : X\hookrightarrow \bP_k(E)$ be a closed immersion.
For every $\cK \in D^b_c(X,\Qlbar)$, there is $N(\cK)\geq 0$ such that for every $s\in S$, we have $c_{i_s}(\cK|_{X_s})\leq N(\cK)$.
\end{lem}
\begin{proof}
This is a direct consequence of the proper base change theorem and the constructibility of the higher direct images applied to the families of sections by projective subspaces.
\end{proof}

\begin{thm}\label{bound_complexity}
Let $(X/S,\Sigma)$ be a projective relative stratified scheme of relative dimension $\leq n$ and let $i : X\hookrightarrow \bP_S(E)$ be a closed immersion.
Let $a\leq b$ be integers.
Then, there is an admissible function $\mu : \bQ[\Coh(X)]\to \bQ$ and $P\in \mathds{N}[x]$ of degree $n$ such that for every algebraic geometric point $\sbar \to S$, every prime $\ell \neq p$, every $\cE\in \bQ[\Coh(X)]$ and every $\cK\in D_{\Sigma_{\sbar}}^{[a,b]}(X_{\sbar},\cE_{\sbar},\Qlbar)$,  we have  
$$
c_{i_{\sbar}}(\cK)\leq P(\mu(\cE))\cdot \Rk_{\Qlbar}\cK  \ .
$$
\end{thm}
\begin{proof}
For $2\leq r \leq \rk E$, consider the commutative diagram
$$
\begin{tikzcd}  
X(E,r)\arrow{r}   \arrow{d}  &  X\arrow{d}{i} \\
\Fl_S(E,1,r)\arrow{r} \arrow{d}  &   \bP_S(E)   \arrow{d}  \\
 \bG_S(E,r)  \arrow[leftarrow, uu,bend left = 70, near end,  "p_{r,X}^{\vee}"] \arrow{r}   &   S 
\end{tikzcd}
$$
 with cartesian upper square where  $\Fl_S(E,1,r)$ is the universal family of dimension $r$ projective subspaces of $\bP_S(E)$.
Then, \cref{bound_complexity} follows from  \cref{general_boundedness_complex} applied to p$_{r,X}^{\vee}: X(E,r)\to  \bG_S(E,r)$.
\end{proof}

\begin{observation}
Every continuity from \cite{SFFK} translates via \cref{bound_complexity} into some estimate involving the rank and the wild ramification only.
\end{observation}
We formulate below two instances of the above  observation.

\begin{thm}\label{Betti_inverse_image}
Let $f : Y\to X$  be a morphism between projective schemes over $k$ algebraically closed.
Let $\Sigma$ be a stratification on $X$ and let $a\leq b$ be integers.
Then, there is an admissible function $\fc :  \bQ[\Coh(X)]\to \bQ$ and $\fd\in \mathds{N}[x]$ of degree $\dim X$ such that for  every prime $\ell \neq p$, every $\cE\in \bQ[\Coh(X)]$ and every  $\cK\in D_{\Sigma}^{[a,b]}(X,\cE,\Qlbar)$,  we have 
$$
\sum_{j\in \bZ}  h^{j}(Y,f^*\cK) \leq \fd(\fc(\cE))\cdot \Rk_{\Qlbar}\cK  \ .
$$
\end{thm}

\begin{proof}
Choose closed immersion  $i : X\hookrightarrow \bP_k(E)$ and $i' : Y\hookrightarrow \bP_k(E')$.
By  \cref{bound_complexity}, there is an admissible function  $\fc : \bQ[\Coh(X)]\to \bQ$ and  $\fd\in 
\mathds{N}[x]$ of degree $\dim X$ such that for every prime $\ell \neq p$, every $\cE\in \bQ[\Coh(X)]$ and every  $\cK\in D_{\Sigma}
^{[c,d]}(X,\cE,\Qlbar)$,  we have 
$$
c_i(\cK)\leq \fd(\fc(\cE))\cdot \Rk_{\Qlbar}\cK  \ .
$$
By \cite[Theorem 6.8]{SFFK},  there is an integer $C$ depending only on $\rk E$ and $\rk E'$ such that for every prime $\ell \neq p$ and every $\cK\in D^b_c(X,\Qlbar)$,  we have 
$$
\sum_{j\in \bZ}  h^{j}(Y,f^*\cK) \leq c_{i'}(f^*\cK)\leq C\cdot c_{i}(\cK) \ .
$$
The conclusion thus follows.
\end{proof}

\begin{cor}\label{Betti_inverse_image_lc}
Let $f : Y\to X$  be a morphism between projective schemes over $k$ algebraically closed where $X$ is smooth.
Let $D$ be an effective Cartier divisor of $X$ and put $U:=X-D$ and $V:=Y-f^{-1}(D)$.
Then there is $P\in \mathds{N}[x]$ of degree $\dim X$ such that for every prime $\ell \neq p$ and every $\cL\in \Loc(U,\Qlbar)$,  we have 
$$
\sum_{j\in \bZ}  h^{j}_c(V,\cL|_V) \leq P(\lc_D(\cL))\cdot \rk_{\Qlbar}\cL  \ .
$$
\end{cor}

\begin{proof}
Let $j : U\hookrightarrow X$ be the inclusion.
Then, combine \cref{bounded_ramification_ex} and \cref{Betti_inverse_image} applied to $j_!\cL$.
\end{proof}

\begin{thm}\label{Coho_direct_image}
Let $(X/S,\Sigma)$ be a  projective relative stratified scheme over $k$ algebraically closed such that $S$ is projective.
Let $a\leq b$ be integers.
Then, there is an admissible function $\mu : \bQ[\Coh(X)] \to \bQ$ and an increasing function $N : \bQ^+ \times \bQ^+ \to \bN^+$  such that for every prime $\ell \neq p$, every $\cE\in \bQ[\Coh(X)]$ and every  $\cK\in D_{\Sigma}^{[a,b]}(X,\cE,\Qlbar)$,  we have  
$$
\sum_{i,j\in \bZ} h^{j}(S, R^i f_*\cK) \leq N(\mu(\cE),\Rk_{\Qlbar}\cK) \ .
$$
\end{thm}

\begin{proof}
Choose closed embeddings  $i : X\hookrightarrow \bP_k(E)$ and $i^{\prime} : S\hookrightarrow \bP_k(E^{\prime})$.
By  \cref{bound_complexity}, there is a sub-additive function $\fc :  \bQ[\Coh(X)]\to \bQ$ and $\fd\in \mathds{N}[x]$ of degree $\dim X$ such that for every prime $\ell \neq p$, every $\cE\in \bQ[\Coh(X)]$ and every $\cK\in D_{\Sigma}^{[a,b]}(X,\cE,\Qlbar)$, we have 
$$
c_{i}(\cK)\leq \fd(\fc(\cE))\cdot \Rk_{\Qlbar}\cK  \ .
$$
By \cite[Proposition 6.24]{SFFK},  there is an increasing function $M : \bN^+\times  \bN^+  \to \bN^+$  such that for every prime $\ell \neq p$ and every $\cK^{\prime}\in D^b_c(S,\Qlbar)$,  we have 
$$
\sum_{j\in \bZ} c_{i^{\prime}}(\cH^j\cK^{\prime}) \leq M(\rk E^{\prime}, c_{i^{\prime}}(\cK^{\prime})) \ .
$$
By  \cite[Theorem 6.8]{SFFK},  there is  $C\geq 0$ depending only on $(X,i)$,$(S,i')$ and $X\to S$ such that for every $\cK\in D_c^b(X,\Qlbar)$, we have 
$$
c_{i^{\prime}}(Rf_* \cK) \leq C\cdot c_{i}(\cK) \ .
$$
For every prime $\ell \neq p$, every $\cE\in \bQ[\Coh(X)]$ and $\cK\in D_{\Sigma}^{[a,b]}(X,\cE,\Qlbar)$, we thus have 
\begin{align*}
\sum_{j,j'\in \bZ} h^{j}(S,  R^{j^{\prime}}f_* \cK) &  \leq \sum_{j\in \bZ} c_{i^{\prime}}( R^j f_* \cK)\\
&\leq  M(\rk E^{\prime}, c_{i^{\prime}}(Rf_* \cK)) \\ 
&   \leq M(\rk E^{\prime}, C\cdot c_{i}(\cK))\\
&  \leq M(\rk E^{\prime}, C\cdot \lceil \fd(\fc(\cE))\cdot \Rk_{\Qlbar}\cK  \rceil) \ .
\end{align*}
The conclusion thus follows if we put $N(x,y):=M(\rk E^{\prime}, C\cdot \lceil P(x)\cdot y  \rceil)$.
\end{proof}

\begin{cor}\label{Coho_direct_image_lc}
Let $f : X\to Y$ be a projective morphism between  projective schemes  over $k$ algebraically closed where $X$ is smooth. 
Let $D$ be an effective Cartier divisor of $X$ and put $j : U:=X-D\hookrightarrow X$.
Then, there is a function $C :  \bQ^+  \times  \bN^+  \to \bN^+$  such that for every prime $\ell \neq p$ and every  $\cL\in \Loc(U,\Qlbar)$,  we have  
$$
\sum_{i,j\in \bZ} h^{j}(Y, R^i f_*j_!\cL) \leq C(\lc_D(\cL), \rk_{\Qlbar}\cL)  \ .
$$
\end{cor}

\begin{proof}
Combine \cref{bounded_ramification_ex} and \cref{Coho_direct_image} applied to $j_!\cL$.
\end{proof}

\section{Relative  Hermite-Minkowski for perverse sheaves}

\begin{notation}
In this section, $\bF$ denotes a finite field of characteristic $p>0$ and $\overline{\bF}$ an algebraic closure of $\bF$.
Let $S$ be a scheme of finite type over $\bF$.
For a closed point $s\in S$, we denote by $\bF(s)$ the residue field of $s$ and put $\deg s :=[\bF(s),\bF]$.
For $n\geq 1$, there is up to isomorphism only one degree $n$ extension of $\bF(s)$ that we denote by $\bF(s)_n$. 
We denote by $S_n$ the pullback of $S$ to $\bF_n$.
\end{notation}

\begin{notation}
We fix an isomorphism $\Qlbar\simeq \bC$.
Purity will be understood with respect to this choice of isomorphism.
Let $X$ be a scheme of finite type over a finite field $\bF$.
For $\cP\in \Perv(X,\Qlbar)$ and $n\geq 1$, we denote by 
$$
t_{\cP,n} : X(\bF_n)\to \Qlbar
$$ 
the trace function of $\cP$.
Reviewed as a complex-valued function via the above choice of isomorphism $\Qlbar\simeq \bC$, one can apply to it the usual hermitian product on $\bC^{|X(\bF_n)|}$.
\end{notation}

\begin{thm}[{\cite[Theorem 7.13]{SFFK}}]\label{SFFK}
Let $X$ be a geometrically irreducible quasi-projective scheme of finite type over $\bF$ and let $i : X\hookrightarrow \bP_\bF(E)$ be a closed immersion.
Then, there is $C_1,C_2\geq 0$ depending only on $\rk E$ such that for every prime $\ell \neq p$,  the following hold :
\begin{enumerate}\itemsep=0.2cm
\item For every $\cP\in \Perv(X,\Qlbar)$ geometrically simple pure of weight 0,  we have 
$$
\left| ||t_{\cP,1}||^2 -1 \right| \leq C_1\cdot c_{i}(\Qlbar) \cdot c_{i}(\cP)^2 \cdot |\bF|^{-1/2} \ .
$$
\item For every $\cP,\cQ\in \Perv(X,\Qlbar)$ geometrically simple pure of weight 0 not geometrically isomorphic,  we have 
$$
\left|<t_{\cP,1},t_{\cQ,1}>\right|\leq C_2\cdot c_{i}(\Qlbar) \cdot c_{i}(\cP) \cdot c_{i}(\cQ) \cdot |\bF|^{-1/2}
$$
\end{enumerate}
\end{thm}

The following lemma could be proved by tracking the constants appearing in the Lang-Weil estimates \cite{LW}. 
We propose here a short proof based on \cite{Weil2}.

\begin{lem}\label{relative_Lang_Weil}
Let $f:X \to S$ be a proper morphism of relative dimension $\leq d$ between schemes of finite type over $\bF$.
Then, there is  $C\in \bN$ such that for every closed point $s\in S$, we have 
$$
|X_s(\bF(s))_n|\leq C\cdot |\bF(s)|^{nd} \ .
$$
\end{lem}

\begin{proof}
Put $C:= \Rk Rf_*\Qlbar$.
By proper base change, for every geometric point $\sbar \to S$ and every $0\leq j\leq 2d$, we have 
$$
h^j(X_{\sbar},\Qlbar)\leq C \ .
$$
By \cite{Weil2}, the complex $R^j f_*\Qlbar$ is mixed of weight $\leq j$.
Hence for every $0\leq j \leq 2d$ and for every closed point $s\in S$, the eigenvalues of the Frobenius $\Fr_s$ acting on $H^j(X_{\sbar},\Qlbar)$ have weight $\leq j$.
By the Lefschetz trace formula \cite[XII,XV]{SGA5},\cite[Rapport]{SGA4-1/2}, we deduce that for ever $n\geq 1$, we have
$$
|X_s(\bF(s)_n)| =\Tr(\Fr_s^n,H^{\bullet}(X_{\sbar},\Qlbar))\leq 2dC\cdot |\bF(s)|^{nd} \ .
$$
\end{proof}

\begin{lem}\label{simple_perverse_irr_comp}
Let $X$ be a scheme of finite type over a field and let $\cP \in \Perv(X,\Qlbar)$.
If $\cP$ is simple, then it is supported on some irreducible component of $X$.
\end{lem}

\begin{proof}
Immediate from \cite[Corollary 5.5]{KW}.
\end{proof}

The following lemma is a relative variant of \cite[Corollary 7.15]{SFFK}.
An inspection of the proof of loc. cit. shows that it works almost verbatim in the relative setting as well.
We detail why for the sake of completeness.

\begin{lem}\label{finitness_perv}
Let $f : X\to S$ be a projective morphism between schemes of finite type over $\bF$  and let $i : X\hookrightarrow \bP_S(E)$ be a closed immersion over $S$.
Then, there is an increasing function $N : \bR^+ \times \bN^+ \to \bN^+$ such that for every prime $\ell \neq p$, every $c>0$ and every closed point $s\in S$, there is up to geometric isomorphism at most $N(c,\deg s)$ 
geometrically simple pure of weight 0 objects $\cP\in 
\Perv(X_s,\Qlbar)$ with  $c_{i_{s}}(\cP) \leq c$.
\end{lem}
\begin{proof}
If $T\to S$ is a quasi-finite morphism where $T$ is of finite type over $\bF$,   then \cite[Tag 055B]{SP} implies the existence of $a\geq 1$ such that for every closed point $t\in T$ with $s=f(t)\in S$,  we have $[\bF(t):\bF(s)] \leq a$. 
In particular,  
$$
\deg t = \deg s \cdot [\bF(t):\bF(s)]\leq a\cdot \deg s  \ .
$$
Hence if there is an increasing function $N : \bR^+ \times \bN^+ \to \bN^+$ as in \cref{finitness_perv} for the pullback $f_T : X_T\to T$,  for every prime $\ell \neq p$, every $c>0$ and every closed point $t\in T$ with $s=f(t)\in S$, there is up to geometric isomorphism at most $N(c,a\cdot \deg s)$ geometrically simple pure of weight 0 objects $\cP\in \Perv(X_{s},\Qlbar)$ with  $c_{i_{s}}(\cP) \leq c$.
Since $S$ is noetherian,  it is enough to show that \cref{finitness_perv} holds after pullback along a  dominant quasi-finite morphism.\\ \indent
By \cite[Tag 0551]{SP}, we can thus assume that $S$ is integral affine with generic point $\eta$ and that the irreducible components $X_{\eta,1},\dots,X_{\eta,m}$ of the generic fibre $X_{\eta}$ are geometrically irreducible.
For $1\leq i \leq m$, let $X_i \subset X$ be the closure of $X_{\eta,i}$.
At the cost of shrinking $S$, we can assume by \cite[Tag 054Y]{SP} that $X=\bigcup_{i=1}^m X_i$.
At the cost of shrinking $S$ further, we can assume by \cite[Tag 0559]{SP} that the $f|_{X_i} : X_i \to S$ have geometrically irreducible fibers.
By \cref{simple_perverse_irr_comp}, at the cost of considering the $f|_{X_i} : X_i \to S$ separately, we can assume that $f : X\to S$ has geometrically irreducible generic fibers.
At the cost of performing a further quasi-finite pullback, we can assume that  $X_{\eta}$ admits a rational point.
By spreading out, we can further assume that $f : X\to S$ admits a section.
In particular for every $s\in S$, the set $X_s(\bF(s))$ is not empty.
By generic flatness, we can also assume that $f : X\to S$ is flat. \\ \indent
Let $\ell\neq p$ be a prime and let $s\in S$ be a closed point.
By \cref{SFFK} applied to $X_s$, there is an integer $C_1$ depending only on $\rk E$ and not on $s$ nor $\ell$ such that for every $\cP\in \Perv(X_s,\Qlbar)$ with $c_{i_{s}}(\cP) \leq c$ and every $n\geq 1$, we have 
$$
\left| ||t_{\cP,n}||^2 -1 \right| \leq C_1\cdot c_{i_s}(\Qlbar) \cdot c_{i_s}(\cP)^2 \cdot |\bF(s)|^{-n/2}\leq C_2 \cdot c^2 \cdot 2^{-n/2}
$$
where we used \cref{complexity_trivial_sheaf}.
Hence, there is  $n_0(c)\geq 1$ increasing as a function of $c$ such that for every $n\geq n_0(c)$, every prime $\ell\neq p$, every closed point $s\in S$ and every $\cP\in \Perv(X_s,\Qlbar)$ with $c_{i_{s}}(\cP) 
\leq c$, we have 
$$
3/4 <||t_{\cP,n}||^2 <5/4 \ .
$$
By \cref{SFFK} applied to $X_s$, there is an integer $C_3$ depending only on $\rk E$ and not on $s$ nor $\ell$ such that for every $\cP,\cQ\in \Perv(X_s,\Qlbar)$ geometrically simple pure of weight 0 and not geometrically isomorphic with $c_{i_{s}}(\cP), c_{i_{s}}(\cQ)  \leq c$, every $n\geq 1$, we have 
$$ 
\left|<t_{\cP,n},t_{\cQ,n}>\right|\leq C_3\cdot c_{i_s}(\Qlbar) \cdot c_{i_s}(\cP) \cdot c_{i_s}(\cQ) \cdot |\bF(s)|^{-n/2} \leq C_4 \cdot c^2\cdot 2^{-n/2}  \ .
$$
Hence, there is $n_1(c)\geq 1$  increasing as a function of $c$ such that for every $n\geq n_1(c)$, every prime $\ell\neq p$, every closed point $s\in S$ and every $\cP,\cQ\in \Perv(X_s,\Qlbar)$ geometrically simple pure of weight 0 and not geometrically isomorphic with $c_{i_{s}}(\cP), c_{i_{s}}(\cQ)  \leq c$, we have 
$$ 
\left|<t_{\cP,n},t_{\cQ,n}>\right| < 1/2 \ .
$$
Put $n(c):=\max(n_0(c),n_1(c))$ and let $C,d\in \bN$ as in     \cref{relative_Lang_Weil} for $f : X\to S$.
By assumption, for every closed point $s\in S$ and every $n\geq 1$, we have 
$$
1\leq |X_s(\bF(s)_n)|\leq C\cdot |\bF(s)|^{nd} \ .
$$
For $m\geq 1$, let $N(m)\in \bN$ be the maximal number of unit vectors $x,y\in \bC^m$ such that $|<x,y>|< 2/3$.
By the above discussion, for every prime $\ell \neq p$ and every closed point $s\in S$, there are at most 
$$
N(|X_s(\bF(s)_n)|)\leq N(C\cdot |\bF(s)|^{n(c)d})=N(C\cdot |\bF|^{n(c)d \deg s})
$$
geometrically simple pure of weight 0 objects $\cP\in \Perv(X_s,\Qlbar)$ with  $c_{i_{s}}(\cP) \leq c$.
\end{proof}

\begin{thm}\label{Deligne_finitness}
Let $(X/S,\Sigma)$ be a relative projective stratified scheme over $\bF$.
Then, there is a function $N : \bQ[\Coh(X)] \times \bN^+ \times \bN^+ \to \bN^+$ such that for every $\cE\in \bQ[\Coh(X)]$, every $r\geq 0$, every prime $\ell \neq p$ and every closed point $s\in S$, there are up to geometric isomorphism at most $N(\cE,r,\deg s)$ geometrically simple pure of weight 0 objects in $\Perv^{\leq r}_{\Sigma_{s}}(X_{s},\cE_{s},\Qlbar)$.
\end{thm}

\begin{proof}
Follows immediately from \cref{bound_complexity} and \cref{finitness_perv}.
\end{proof}

\begin{rem}
For $\cE\in \Coh(X)$, the function $N : \bQ[\Coh(X)] \times \bN^+ \times \bN^+ \to \bN^+$ constructed in the proof of \cref{Deligne_finitness} is so that the induced function $N_{\cE} : \bQ^+ \times \bN^+ \times \bN^+ \to \bN^+$ defined by 
$$
N_{\cE}(c,r,d)=N(c\cdot \cE,r,d)
$$
is increasing.
\end{rem}

The purity assumption in \cref{Deligne_finitness} can be removed at the cost of using the following consequence of the Langlands correspondence.
\begin{thm}[{\cite{Drin88,Drin89,Laf}}]\label{Langlands}
Let $U$ be a connected normal scheme of finite type over $\bF_q$.
Let $\cL \in \Loc(U,\Qlbar)$ and assume that $\cL$ is simple with finite order determinant.
Then $\cL$ is pure of weight 0.
\end{thm}

\begin{lem}\label{purity_after_twist}
Let $X$ be a scheme of finite type over $\bF_q$.
Then, every simple perverse sheaf on $X$ is geometrically isomorphic to a simple perverse sheaf which is pure of weight $0$.
\end{lem}

\begin{proof}
Let $\cP\in \Perv(X,\Qlbar)$ be a simple perverse sheaf.
It is enough to show that there is a character $\chi : G_{\bF_q}\to \Qlbar^{\times}$  such that the twist 
$$
\chi \cdot  \cP:=p^*\chi\otimes_{\Qlbar}\cP
$$ 
is pure of weight $0$, where $p: X\to \Spec \bF_q$ is the structural morphism.
We can suppose that $X$ is reduced.
By \cite[Corollary 5.5]{KW}, the complex $\cP$ is of the form $i_* j_{!*} \cL[d]$ where $i : Y\hookrightarrow X$ is an irreducible closed subset, where $j : U\hookrightarrow Y$ is a smooth dense open subset of dimension $d$ and where $\cL\in \Loc(U,\Qlbar)$ is simple.
By  \cite[Corollary 6.5.6]{fu}, for any character $\chi$ coming from $\bF_q$, we have 
$$
i_* (\chi \cdot  j_{   !*} \cL[d])\simeq  \chi \cdot i_*  j_{   !*} \cL[d] \simeq \chi \cdot  \cP \ .
$$
Since $i_\ast$ preserves pure complexes of weight 0 in virtue of \cite[Stabilités 5.1.14]{BBD}, we can assume that $Y=X$.
By \cite[Proposition 1.3.4]{Weil2}, the sheaf $\cL$ admits a twist $\chi \cdot \cL$ with finite order determinant.
Then, \cref{Langlands} applied to $\chi \cdot \cL$  on the smooth connected scheme $U$ ensures that $\chi \cdot \cL$ is pure of weight 0.
At the cost of twisting $\chi \cdot \cL$ further by a character of weight $-d$ coming from $\bF_q$, there is a character $\chi$ coming from $\bF_q$ such that $\chi \cdot \cL$ is pure of weight $-d$.
In particular $\chi \cdot \cL[d]$ is pure of weight $0$.
On the other hand, we have
$$
j_{   !*} (\chi \cdot \cL[d]) \simeq \chi \cdot  j_{   !*} \cL[d] \simeq  \chi \cdot \cP \ .
$$
Since $j_{!*}$ preserves pure perverse sheaves of a given weight by \cite[Corollaire 5.4.3]{BBD}, the conclusion follows.
\end{proof}

\begin{thm}\label{Deligne_finitness_with_Langlands}
Let $(X/S,\Sigma)$ be a relative projective stratified scheme over $\bF$.
Then, there is a function $N : \bQ[\Coh(X)] \times \bN^+ \times \bN^+ \to \bN^+$ such that for every $\cE\in \bQ[\Coh(X)]$, every $r\geq 0$, every prime $\ell \neq p$ and every closed point $s\in S$, there are up to geometric isomorphism at most $N(\cE,r,\deg s)$ geometrically simple  objects in $\Perv^{\leq r}_{\Sigma_{s}}(X_{s},\cE_{s},\Qlbar)$.
\end{thm}

\begin{proof}
Combine \cref{Deligne_finitness} and \cref{purity_after_twist}.
\end{proof}

\section{Appendix : concrete estimates for $\bA^n_k$}

\begin{lemma}\label{explicit_bound_b}
In the setting of \cref{defin_bi}, we have $b_2(x)=x^2+7x+9$ and for every $n\geq 3$, we have the following inequality of functions on $\bR^+$ :
$$
b_n(x)\leq (x+3n-3)\prod_{j=1}^{n-1}(x+3j+1) \ .
$$
\end{lemma}
\begin{proof}
The equation $b_2(x)=x^2+7x+9$ is from the definition. 
In the rest of the proof, we consider all polynomials with indeterminate $x$ as functions defined in $\mathbb R_{\geq 0}$. 
For  $n\geq 3$, we have
\begin{align*}
b_{n}(x)&=\sum^{n-1}_{\substack{i=0\\  i\neq n \modulo 2 }}(x+2)\cdot b_{i}(x+3)
+\sum^{n-1}_{\substack{i=0\\ i= n \modulo 2}} (x+3)\cdot b_{i}(x) + \sum^{n-1}_{\substack{i=0\\  i\neq n \modulo 2}}b_{i}(x)\\
&\leq \sum^{n-1}_{\substack{i=0\\  i\neq n \modulo 2 }}(x+2)\cdot b_{i}(x+3)
+\sum^{n-1}_{\substack{i=0\\ i= n \modulo 2}} (x+3)\cdot b_{i}(x+3) + \sum^{n-1}_{\substack{i=0\\  i\neq n \modulo 2}}b_{i}(x+3)\\
&\leq \sum^{n-1}_{\substack{i=0\\  i\neq n \modulo 2 }}(x+3)\cdot b_{i}(x+3)+\sum^{n-1}_{\substack{i=0\\ i= n \modulo 2}} (x+3)\cdot b_{i}(x+3) \\
&=\left(\sum^{n-1}_{i=0} b_{i}(x+3)\right)\cdot(x+3).
\end{align*}

Let $c_0(x)=1, c_1(x)=x$ be polynomials in $\mathbb Z[x]$. 
For $n\geq 2$, we inductively define $c_{n}(x)$ by
$$
c_{n}(x)=\left(\sum^{n-1}_{i=0} c_{i}(x+3)\right)\cdot(x+3).$$
For  $n\geq 1$, we see that the coefficients of $c_n(x)$ are positive integers and $\deg(c_n(x))=n$. 
In particular, we have $c_2(x)=x^2+7x+12$. 
For  $n\geq 3$, we have 
\begin{align*}
c_n(x)&=\left(\sum^{n-1}_{i=0} c_{i}(x+3)\right)\cdot(x+3)\\
&=\left(c_{n-1}(x+3)+\sum^{n-2}_{i=0} c_{i}(x+3)\right)\cdot(x+3)\\
&=\left(\left(\sum^{n-2}_{i=0} c_{i}(x+6)\right)\cdot(x+6)+\sum^{n-2}_{i=0} c_{i}(x+3)\right)\cdot(x+3)\\
&\leq \left(\sum^{n-2}_{i=0} c_{i}(x+6)\right)\cdot(x+7)(x+3)\\
&\leq \left(\sum^{n-2}_{i=0} c_{i}(x+6)\right)\cdot(x+6)(x+4)=c_{n-1}(x+3)\cdot(x+4). 
\end{align*}
Hence for $n\geq 3$, we have 
$c_n(x)\leq (x+3n-3)\prod_{i=1}^{n-1}(x+3i+1)$.
Notice that $b_0(x)\leq c_0(x), b_1(x)\leq c_1(x), b_2(x)\leq c_2(x)$. 
Suppose that for every $0\leq n\leq k$, we have $b_k(x)\leq c_k(x)$. 
When $n=k+1$, we have 
\begin{align*}
b_n(x)&=b_{k+1}(x)\leq \left(\sum^{k}_{i=0} b_{i}(x+3)\right)\cdot(x+3)\\
&\leq \left(\sum^{k}_{i=0} c_{i}(x+3)\right)\cdot(x+3)=c_{k+1}(x)=c_n(x).
\end{align*}
By induction, we have $b_n(x)\leq c_n(x)$ for every $n\geq 0$. 
Hence for every $n\geq 3$, we have 
\begin{align*}
b_n(x)\leq c_n(x)\leq (x+3n-3)\prod_{i=1}^{n-1}(x+3i+1).
\end{align*}
\end{proof}

\end{document}